\documentclass[11pt,reqno]{amsart}
\usepackage[margin=0.8in]{geometry}
\usepackage{amsmath,amssymb,amsthm,graphicx,amsxtra, setspace}
\usepackage[utf8]{inputenc}
\usepackage{mathrsfs}
\usepackage{hyperref}
\usepackage{upgreek}
\usepackage{mathtools}
\usepackage[dvipsnames]{xcolor}
\allowdisplaybreaks

\usepackage[pagewise]{lineno}

\usepackage{graphicx,eurosym}
\usepackage{hyperref}
\usepackage{mathtools}

\usepackage[cyr]{aeguill}

\colorlet{darkblue}{blue!50!black}

\hypersetup{
	colorlinks,%
	citecolor=blue,%
	filecolor=red,%
	linkcolor=darkblue,%
	urlcolor=blue,%
	pdfnewwindow=true,%
	pdfstartview={FitH}
}

\colorlet{darkblue}{red!100!black}

\newtheorem{theorem}{Theorem}[section]
\newtheorem{lemma}[theorem]{Lemma}

\newtheorem{definition}[theorem]{Definition}

\newtheorem{remark}[theorem]{Remark}
\newtheorem{hypothesis}[theorem]{Hypothesis}

\newcommand{\Tr}{\mathop{\mathrm{Tr}}}

\let\originalleft\left
\let\originalright\right
\renewcommand{\left}{\mathopen{}\mathclose\bgroup\originalleft}
\renewcommand{\right}{\aftergroup\egroup\originalright}

\theoremstyle{definition}

\def\1{\mathcal{O}}

\def\wi{\widehat}
\def\vi{\widetilde}
\def\g{\boldsymbol{g}}

\def\d{\mathrm{d}}

\def\I{\mathrm{I}}

\def\D{\mathrm{D}}
\def\A{\mathrm{A}}

\def\W{\mathrm{W}}
\def\R{\mathbb{R}}
\def\E{\mathbb{E}}

\def\H{\mathbb{H}}
\def\V{\mathbb{V}}
\def\e{\varepsilon}
\def\2{\mathcal{E}}
\def\L{\mathrm{L}}
\def\u{\boldsymbol{u}}
\def\v{\boldsymbol{v}}
\def\w{\boldsymbol{w}}

\def\C{\mathrm{C}}
\def\X{\mathbf{X}}
\def\Y{\mathbf{Y}} 

\def\P{\mathbb{P}}
\def\N{\mathbb{N}}
\def\x{\boldsymbol{x}}
\def\y{\boldsymbol{y}}
\def\G{\mathrm{G}}

\def\Z{\mathrm{Z}}
\def\PP{\mathrm{P}}
\def\U{\mathbb{U}}

\def\x{\boldsymbol{x}}

\def\vp{\varpi}
\def\s{\sigma}
\def\M{\mathrm{M}}

\def\ZZ{\mathbf{Z}}
\def\DD{\mathcal{D}}
\def\LL{\mathcal{L}}
\def\S{\mathrm{S}}

\def\ee{\boldsymbol{e}}
\def\bphi{\boldsymbol{\phi}}

\newcommand{\Addresses}{{
		\footnote{
			\noindent \textsuperscript{1,2,3}Department of Mathematics, Indian Institute of Technology Roorkee-IIT Roorkee,
			Haridwar Highway, Roorkee, Uttarakhand, 247667, INDIA.\par\nopagebreak
			\noindent  \textit{e-mail:} \texttt{Manil T. Mohan: maniltmohan@ma.iitr.ac.in, maniltmohan@gmail.com.}
			
			\textit{e-mail:} \texttt{Ankit Kumar: akumar14@mt.iitr.ac.in, ankitkumar.2608@gmail.com.}
			
				\textit{e-mail:} \texttt{Kush Kinra:
			kkinra@ma.iitr.ac.in, kushkinra@gmail.com.}
		
			\noindent \textsuperscript{*}Corresponding author.
			
			\textit{Keywords:} Stochastic partial differential equations, locally monotne, Wong-Zakai approximation, support theorem, Gaussian noise.
			
			Mathematics Subject Classification (2020): Primary 60H15, 76D03; Secondary 37H05, 35R60.

}}}
\begin{document}	
	
	\title[Wong-Zakai approximation]{Wong-Zakai approximation for a class of  SPDE\lowercase{s} with fully local monotone coefficients and its application
		\Addresses}

	\author[A. Kumar, K. Kinra and M. T. Mohan]
	{Ankit Kumar\textsuperscript{1}, Kush Kinra\textsuperscript{2} and Manil T. Mohan\textsuperscript{3*}}

	\maketitle

\begin{abstract}
	In this article, we establish the \textsl{Wong-Zakai approximation} result for a class of stochastic partial differential equations (SPDEs) with fully local monotone coefficients perturbed by a multiplicative Wiener noise. This class of SPDEs encompasses various fluid dynamic models and also includes quasi-linear SPDEs, the convection-diffusion equation, the Cahn-Hilliard equation, and the two-dimensional liquid crystal model.  It has been established that the class of SPDEs in question is well-posed, however, the existence of a unique solution to the associated approximating system cannot be inferred from the solvability of the original system. We employ a Faedo-Galerkin approximation method, compactness arguments, and Prokhorov's and Skorokhod's representation theorems to ensure the existence of a \textsl{probabilistically weak solution} for the approximating system. Furthermore, we also demonstrate that the solution is pathwise unique. Moreover, the classical Yamada-Watanabe theorem allows us to conclude the existence of a \textsl{probabilistically strong solution} (analytically weak solution) for the approximating system. Subsequently, we establish the Wong-Zakai approximation result for a class of SPDEs with fully local monotone coefficients.  We utilize the Wong-Zakai approximation to establish the topological support of the distribution of solutions to the SPDEs with fully local monotone coefficients. Finally, we explore the physically relevant stochastic fluid dynamics models that are covered by this work's functional framework.
\end{abstract}
\section{Introduction}\label{Sec1}\setcounter{equation}{0}

\subsection{The  model and literature}
Approximating solutions to stochastic differential equations (SDEs) can be achieved through various methods, including the Euler scheme, the Crank-Nicholson scheme, and the Wong-Zakai approximation, etc., see the works \cite{HLN,HLN1,TMRZ,EWMZ}, etc., for the details of each scheme. This paper aims to investigate the Wong-Zakai approximation for a specific class of stochastic partial differential equations (SPDEs) (to be specified later). 
Due to the continuous yet nondifferentiable nature of Brownian motion, there is a need for an approximation that can capture its properties in some way.
 In 1965, E. Wong and M. Zakai proposed a novel scheme for approximating the solution of an SDE perturbed by a one-dimensional Brownian motion. This approach involved replacing the Brownian motion with a suitable smooth approximation and making a drift correction in the original SDE (see \cite{EWMZ}). This method is now commonly known as the \textsl{Wong-Zakai approximation}. The extension of the Wong-Zakai approximation to the multi-dimensional case was achieved in \cite{SWNI}, building upon the original work of \cite{EWMZ}. Moreover, the Wong-Zakai approximation has received widespread attention in infinite-dimensional cases as well, as evidenced by the works \cite{Ganguly,IPR,KT2,KT,KT1}, etc., which discuss its application to stochastic (partial) differential equations driven by infinite-dimensional noise. Notably, the author in \cite{KT1} provided an extensive overview of the Wong-Zakai approximation for various types of SDEs, including SDEs with finite-dimensional state and noise, SDEs with infinite-dimensional state and finite-dimensional noise, and SDEs with infinite-dimensional state and noise, etc.
Furthermore, \cite{KT1} (and \cite{KT}) also introduced two new types of correction terms in the Wong-Zakai approximation.

The Wong-Zakai approximation studied for two-dimensional hydrodynamical models, such as the magnetic B\'enard problem, Navier-Stokes equations, and magneto-hydrodynamic equations, in \cite{ICAM2}. In \cite{HP}, the authors demonstrated the Wong-Zakai approximation to 1D parabolic nonlinear SPDEs driven by space-time white noise. In addition, the Wong-Zakai approximation of various physical models can be found in \cite{BMM,SGUM,KKMTM,LS,LQ,TTM,Yastrzhembskiy}, etc. and their references. 

In 1974, E. Pardoux extended the theory of the monotone operator from the deterministic to the stochastic case (see \cite{EP1,EP2}). Various authors have investigated the application of variational approaches to SPDEs in the literature, such as \cite{IG,JRMR,MRFYW,MRXZ}, etc., and their cited references. Many works are available on the existence and uniqueness of SPDEs under a variational framework, some of them we have cited above. Moreover, the variational approach has been exploited for the SPDEs with local monotone coefficients, for instance see \cite{ZBWLJZ,WLMR1,WLMR2}, etc. In addition, the works \cite{WL,WLMR3} established the well-posedness of SPDEs with local monotone coefficients with Lyapunnov condition and generalized coercivity conditions, respectively. Recently, the authors in \cite{TMRZ} studied the Wong-Zakai approximation of a class of SPDEs with locally monotone coefficients perturbed by a trace class noise, which covers several stochastic models of fluid dynamics like two-dimensional hydrodynamical type systems (stochastic magneto-hydrodynamic equations, stochastic Navier-Stokes equations, etc.), the $p$-Laplace evolution equation, stochastic porous media equations, etc. Additionally, the functional framework of \cite{TMRZ} covers  two-dimensional stochastic convective Brinkman-Forchheimer equations with a subcritical exponent (for instance see \cite[Remark 2.11]{KKMTM}).

 Let us examine a class of SPDEs with fully local monotone coefficients within the Gelfand triplet $\V\subset\H\subset\V^*$ that are driven by a multiplicative Gaussian noise:
\begin{align}\label{1}
	\d \Y(t)=\A(t,\Y(t))\d t+\sigma(t,\Y(t))\d\W(t), \  \text{ for } \  t\in[0,T],
\end{align}
where $\H$ and $\V$ represent a separable Hilbert space and a reflexive Banach space, respectively, such that the continuous embedding of $\V\hookrightarrow\H$ is dense and $\V^*$ denote the dual of $\V$, and $\W(\cdot)$ is a Hilbert space-valued cylindrical Wiener process on a filtered probability space  $(\Omega,\mathscr{F},\{\mathscr{F}_t\}_{t\geq 0},\P)$ (cf. Subsection \ref{sub2.1}). The maps $\A(\cdot,\cdot)$ and $\sigma(\cdot,\cdot)$ in the above equation are measurable (cf. Subsection \ref{sub2.1}). We mention that the models covered by the work \cite{TMRZ} also come under the functional framework of this work for \eqref{1}. In fact, the functional framework of this work is more general than that of \cite{TMRZ}.

For almost a decade, it has been an open problem to determine the well-posedness of SPDEs with fully local monotone coefficients driven by a multiplicative noise, as noted in \cite{ZBWLJZ,WL4,WLMR2}. Recently, in \cite{MRSSTZ}, the authors proved the well-posedness results for a class of SPDEs with fully local monotone coefficients driven by a multiplicative Gaussian noise under certain assumptions on the noise coefficient. An extension of the results of \cite{MRSSTZ} to include SPDEs driven by L\'evy noise was later provided by authors in \cite{AKMTM4}. In the work \cite{AKMTM5}, the authors explored the asymptotic behavior (large deviation principle) of solutions for a class of SPDEs with fully local monotone coefficients driven by L\'evy noise. In \cite{AKMTM6,TPSS}, the authors considered a class of SPDEs with fully local monotone coefficients driven by a multiplicative Gaussian noise to establish the small time asymptotics and the large deviation principle for solutions of this class of SPDEs, respectively.


The Wong-Zakai approximation can be used to describe the topological support of solutions of SDEs and SPDEs. In 1972, Stroock and Varadhan introduced the concept of topological support for the solution of a SDE in their pioneering work \cite{SV}. Moreover, the support theorem for SDEs is well studied in the works \cite{GP,SWNI} etc. It is worth mentioning here that the Wong-Zakai approximation helps us to find the topological support of solutions of SDEs and SPDEs (cf. \cite{GP,Mackevicius,MS}). Furthermore, one can also say that the support theorem is an application of the Wong-Zakai approximation result, see \cite{ICAM2,KKMTM,TMRZ,Yastrzhembskiy}, etc. We are also demonstrating the support for a solution to the system \eqref{1} with the help of the Wong-Zakai approximation result (see Section \ref{Sec3}). Theorem \ref{thrm2} and Lemma  \ref{lem4.02} help us to prove the topological support of the distribution to the solution of the system \eqref{1} (see Theorem \ref{thrm4.03}).

For any kind of approximation scheme, there is always a question of the rate of convergence of the approximation. In the literature, some authors discussed the rate of convergence of the Wong-Zakai approximation for different SPDEs, that is, they investigated the rate of convergence of the solution of the approximating system to the solution of the original system in terms of the rates of convergence of the driving processes (say, $\W_n(\cdot)$) approximating the original Wiener process (say, $\W(\cdot)$), see the works \cite{BKBAAL,IGAS,IPR,NT}, etc. Particularly, in \cite{IPR}, it has been shown that the rate of convergence of the Wong-Zakai approximation for a class of SPDEs is essentially the same as the rate of convergence of $\W_n(\cdot)$ approximating the $\W(\cdot)$ under certain conditions. It is an interesting open problem to examine the rate of convergence of the Wong-Zakai approximation of system \eqref{1} which will be a part of our future work.


\vskip 0.2 cm
\subsection{Difficulties and approaches}\label{111}  As discussed already, the existence of \textsl{a unique probabilistically strong solution} to the system \eqref{1} with fully local monotone coefficients (see Hypothesis \ref{hyp1} below) has been established in a recent paper \cite{MRSSTZ}. The unique solution for the approximating system (see \eqref{1.6} below) cannot be concluded from the solvability of the original system, as it requires estimating the drift parts $\A$ and $\sigma\dot{\W}^m$ in \eqref{1.6} separately. This is due to the fact that the approximating system is considered as a pathwise deterministic system (the diffusion coefficient is $0$).
The solvability of the system that approximates SPDEs with locally monotone coefficients relies on compactness and monotonicity criteria. However, the aforementioned approach is not immediately applicable for SPDEs with fully local monotone coefficients. We demonstrate the existence of a unique, probabilistically strong solution of the system \eqref{1.6} using compactness principles combined with techniques from pseudo-monotone operator theory.

It is noteworthy that we take into account a general stochastic system (shown in \eqref{4.4}) in our analysis, which encompasses the system \eqref{1.6} and thereby allows us to demonstrate the topological support of the distribution corresponding to the solution of \eqref{1}. The proof of solvability result of the system \eqref{4.4} is based on a standard Faedo-Galerkin approximation technique. We first derive uniform moment estimates for the Galerkin approximating solutions (as displayed in Lemma \ref{lem4.2}), and then establish the tightness of the laws of these solutions in appropriate spaces (as mentioned in Lemma \ref{lem4.3}). Utilizing Prokhorov's theorem (see \cite[Section 5]{PB}) and Skorokhod's representation theorem (see \cite[Theorem C.1]{ZBWLJZ}) along with the pseudo-monotonicity of the mapping $\Y(\cdot)\mapsto \A(\cdot,\Y(\cdot))$, we further show that system \eqref{4.4} admits a probabilistically weak solution. Further, we demonstrate pathwise uniqueness of the solutions to the system \eqref{4.4} (as presented in Theorem \ref{thrm2.9}). Finally, the classical Yamada-Watanabe theorem (\cite[Theorem 2.1]{MRBSXZ}) guarantees the existence of a unique probabilistically strong solution to the system \eqref{4.4}.


Note that the approximating system does not contain any diffusion term (see \eqref{1.6} below), which leads to difficulty in estimating the difference between the solutions of the original and the approximating systems while proving the Wong-Zakai approximation result for the system \eqref{1}. We use the identity
$\displaystyle\int_{0}^{t}=\sum\limits_{l=0}^{\lfloor\frac{t}{\vp}\rfloor}\int_{l\vp}^{(l+1)\vp \land t}$  for  $t\in[0,T],$
where $\vp=\frac{T}{2^m}$, and the structure of the finite-dimensional approximation of the cylindrical Wiener process $\W(\cdot)$ (see \eqref{1.3} below) to make comparable the approximating system with the original system (see Remark \ref{Remark4.2}).

\subsection{Novelties and applications}
Our aim in this work is to demonstrate the Wong-Zakai approximation of the system \eqref{1} with fully local monotone coefficients, which covers various stochastic models related in applied fields such as fluid dynamics, etc. To the best of our understanding, there is no result available in the literature regarding the Wong-Zakai approximation of the model \eqref{1} with fully local monotone coefficients. The framework of this work encompasses all the stochastic models discussed in \cite{TMRZ}. We further emphasize that the assumptions on the noise and the correction term of the approximated system \eqref{4.4} are the same as those in \cite{TMRZ} (as indicated in Hypotheses \ref{hyp1} and \ref{hyp2}). Thus, the examples of noise given in \cite[Section 3]{TMRZ} are also applicable in our framework. The main differences between our work and that done by \cite{TMRZ}, who achieved a similar result for a class of SPDEs with locally monotone coefficients driven by trace class noise, are as follows:
\begin{enumerate}
	\item [$\bullet$]  The coefficient $\A(\cdot,\Y(\cdot))$ appearing in \eqref{1} satisfies
	\begin{align}\label{11}
		2\langle \A(\cdot,\Y_1)-\A(\cdot,\Y_2),\Y_1-\Y_2\rangle & \leq \big(f(\cdot)+\rho(\Y_1)+\eta(\Y_2)\big)\|\Y_1-\Y_2\|_{\H}^2,
	\end{align}
	that is, it allows the right hand side of \eqref{11} to depend on both variables $\Y_1$ and $\Y_2$ (see condition (H.1) of Hypothesis \ref{hyp1}). It covers several interesting examples like the Cahn-Hilliard equation, quasilinear SPDEs, a two-dimensional liquid crystal model, and many more (for instance, one can see \cite{WLMR1,MRSSTZ} or \cite{AKMTM6}), which were not covered in the work \cite{TMRZ}.
	\item [$\bullet$]  The function $f$ appearing in \eqref{11} is assumed to be in the space $\L^1(0,T;\R^+)$ which is more general as compared to the Hypotheses of \cite{TMRZ}.
\end{enumerate}

The major outcomes of this work are listed below:
\begin{enumerate}
	\item [1.] We demonstrate the Wong-Zakai approximation result for the system \eqref{1} under the Hypotheses \ref{hyp1} and \ref{hyp2} (see Theorem \ref{thrm3.3} below).
	\item [2.] We also find the topological support of the distribution to the solution of the system \eqref{1} (under the Hypotheses \ref{hyp1} and \ref{hyp2}) with the help of the Wong-Zakai approximation (Theorem \ref{thrm4.03}).
\end{enumerate}

The results of this work are applicable to the stochastic versions of hydrodynamic models like Burgers equations, two-dimensional Navier-Stokes equations, two-dimensional magneto-hydrodynamic equations, etc.,  porous media equations,  $p$-Laplacian equations, fast-diffusion equations, power law fluids and three-dimensional tamed Navier-Stokes equations etc. There are several other stochastic models satisfying the framework of this work, which we have discussed in Section \ref{Application}. 

\subsection{Organization of the paper} The article is organized in the following manner: In Section \ref{Sec2}, we provide the functional framework followed by the Hypotheses \ref{hyp1} and \ref{hyp2} required to obtain the main results of this paper. Then, we state our first main theorem, that is, the well-posedness of the system \eqref{1.6} (Theorem \ref{thrm2}). To obtain the well-posedness of the system \eqref{1.6}, we define a general system \eqref{4.4} (the system \eqref{1.6} is the particular case of the system \eqref{4.4}) and establish the well-posedness result for this general system \eqref{4.4} (Theorem \ref{thrm4.1}). Section \ref{Seclem4.1} is devoted to the proof of Theorem \ref{thrm4.1} with the help of a standard Faedo-Galerkin approximation technique, tightness arguments, and Prokhorov's and Skorokhod's representation theorems (Lemmas \ref{lem4.2}-\ref{lem4.7} and Theorems \ref{thrm2.8}-\ref{thrm2.9}). Then, the Theorem \ref{thrm2} is a particular case of Theorem \ref{thrm4.1}. In Section \ref{Sec3}, we state and prove our main result (Theorem \ref{thrm3.3}), that is, the Wong-Zakai approximation result for the system \eqref{1.1}. In Section \ref{Sec4}, we discuss the support of a solution to the system \eqref{1.1}, which is a consequence of the Wong-Zakai approximation result (Theorem \ref{thrm4.03}). In the final section, we provide some stochastic models which fall under the functional framework of this article.

	\section{Mathematical formulation and solvability results}\label{Sec2}\setcounter{equation}{0}
	\subsection{Functional framework}\label{sub2.1}
	In this work, we consider the following class of SPDEs with fully local monotone coefficients in a  Gelfand triplet $\V\hookrightarrow\H\hookrightarrow\V^*$ driven by a multiplicative Gaussian noise: 
	\begin{equation}\label{1.1}
		\left\{
		\begin{aligned}
			\d \Y(t)&=\A(t,\Y(t))\d t+\s(\Y(t))\d\W(t), \ \ t\in[0,T],\\
			\Y(0)&=\boldsymbol{y}_0,
		\end{aligned}
		\right.
	\end{equation}
	where  $\H$ and $\V$ represent a separable Hilbert space and a reflexive Banach space, respectively, such that the continuous  embedding of  $\V\hookrightarrow \H$ is dense. Let $\V^*$ and $\H^*\cong \H$ denote  the dual of the spaces $\V$ and $\H$, respectively. The norms of $\H,\V$ and $\V^*$ are denoted by $\|\cdot\|_\H,\|\cdot\|_\V$ and $\|\cdot\|_{\V^*}$, respectively. Let the symbol $(\cdot,\cdot)$ represents the inner product in the Hilbert space $\H$ and the symbol $\langle \cdot,\cdot\rangle$ denotes the duality paring between $\V$ and $\V^*$. Also, we have $\langle \x,\y\rangle =(\x,\y)$, whenever $\x\in\H$ and $\y\in\V$. Let $\U$ be an another separable Hilbert space, and $\L_2(\U,\H)$ be the space of all Hilbert-Schimdt operators from $\U$ to $\H$ with the norm $\|\cdot\|_{\L_2}$ and the inner product $(\cdot,\cdot)_{\L_2}$.  The mappings 
	\begin{align*}
		\A:[0,T]\times\V\to\V^* \ \text{ and } \ \s:\V\to\L_2(\U,\H),
	\end{align*}
	are measurable and $\W(\cdot)$ is a $\U$-valued cylindrical Wiener process on a filtered probability space\break $(\Omega,\mathscr{F},\{\mathscr{F}_t\}_{t\geq 0},\P)$.	
	Next, we define the finite-dimensional approximation of the cylindrical Wiener process $\W(\cdot)$. We know that the $\mathbb{U}$-valued cylindrical Wiener process $\W(t)$ can be written in the following infinite sum (see \cite[Section 4.1.2]{DaZ}):
	\begin{align}\label{1.2}
		\W(t)=\sum_{i=1}^\infty\beta_i(t)\ee_i, \ \ t\in[0,T],
	\end{align}where $\beta_i$'s are standard independent real-valued Brownian motions and $\{\ee_i\}_{i\in\N}$ is an orthonormal basis of the Hilbert space $\U$. For any $m\in\N$, we set $\vp=\frac{T}{2^m}$ and define 
\begin{align}\label{1.3}
	\dot{\W}^m(t)=\sum_{i=1}^m\vp^{-1} \bigg(\beta_i\Big(\lfloor\frac{t}{\vp}\rfloor\vp\Big) -\beta_i \Big(\big(\lfloor\frac{t}{\vp}\rfloor-1\big)\vp\Big)\bigg)\ee_i=:\sum_{i=1}^m\dot{\beta}_i^m(t)\ee_i,\ \ t\in[0,T],
\end{align}where $\lfloor t\rfloor$ (resp. $\lceil t\rceil$) denotes the largest integer which is no more than $t$ (resp. smallest integer which is larger than $t$).  Also, we set
\begin{equation}\label{1.4}
\beta_i(t)=	\left\{\begin{aligned}
	&0, \ &&\text{ for } \ t\leq 0,\\
	&\beta_i(T), \ &&\text{ for } \ t\geq T,
	\end{aligned}
	\right.
\end{equation}
which implies $\dot{\beta}_i(t)=0$, for $t>T$. Therefore, $\{\dot{\beta}_i^m(t)\}_{i=1}^m$ are $\mathscr{F}_t$-adapted and hence $\dot{\W}^m(\cdot)$.
	
For $i=1,\ldots,m$, define  $\s_i:\H\to\H$  as $\s_i(\x):=\s(\x)\ee_i$, for $\x\in\H$. We assume that for each $i$, $\s_i(\cdot)$ is Fr\'echet differentiable and the derivative is denoted by $\D\s_i:\H\to \mathcal{L}(\H,\H)$. Let us define the map
\begin{align}\label{1.5}
	\wi{\Tr}_m:\H\to\H, \ \text{ by } \ \wi{\Tr}_m(\x)=\sum_{i=1}^m\D\s_i(\x)\s_i(\x), \ \ \text{ for all }\ \
	  \x\in\H.
\end{align}We consider the following approximating system:
\begin{equation}\label{1.6}
	\left\{
	\begin{aligned}
		\d \Y^m(t)&=\A(t,\Y^m(t))\d t+\s(\Y^m(t))\dot{\W}^m(t)\d t-\frac{1}{2}\wi{\Tr}_m(\Y^m(t))\d t,\\ 
		\Y^m(0)&=\boldsymbol{y}_0,
	\end{aligned}
	\right.
\end{equation}where $\dot{\W}^m(\cdot)$ and $\wi{\Tr}_m(\cdot)$ have been defined in \eqref{1.3} and \eqref{1.5}, respectively.

	\subsection{Hypothesis and solvability results}   The system \eqref{1.1} with locally monotone coefficients is considered in the works \cite{WLMR1,WLMR2,WLMR3}, etc., where as the system \eqref{1.1} with fully local monotone coefficients is considered in the work \cite{MRSSTZ} (cf. \cite{AKMTM4,AKMTM5,AKMTM6,TPSS}). We consider the following assumptions on the coefficients $\A(\cdot,\cdot)$ and $\sigma(\cdot)$:
	
	\begin{hypothesis}\label{hyp1} Let $f\in\L^1(0,T;\R^+)$ and $\beta\in(1,\infty)$.
		\begin{itemize}
			\item[(H.1)] (\textsl{Hemicontinuity}). The map $\R\ni \lambda \mapsto \langle \A(t,\x_1+\lambda \x_2),\x\rangle\in\R$ is continuous for any $\x_1,\x_2,\x\in\V$ and for a.e. $t\in[0,T]$.
			\item[(H.2)] (\textsl{Local monotonicity}). There exist non-negative constants $\zeta$ and $C$ such that for any $\x_1,\x_2,\x\in\V$ and a.e. $t\in[0,T]$, 
			\begin{align}\label{1.9}
				\langle \A(t,\x_1)-\A(t,\x_2),\x_1-\x_2\rangle & \leq \big(f(t)+\rho(\x_1)+\eta(\x_2)\big)\|\x_1-\x_2\|_\H^2, \\ \label{1.10}
				\|\s(\x_1)-\s(\x_2)\|_{\L_2}^2& \leq \big(\kappa(\x_1)+\varkappa(\x_2)\big)\|\x_1-\x_2\|_\H^2, \\ 
				\label{1.11}
				|\rho(\x)|+|\eta(\x)| &\leq C(1+\|\x\|_\V^\beta)(1+\|\x\|_\H^\zeta),\\
				\label{1.12}
			|\kappa(\x)|+|\varkappa(\x)| &\leq C(1+\|\x\|_\H^\zeta),	
			\end{align} where $\rho$ and $\eta$  (resp. $\kappa$ and $\varkappa$) are two measurable functions from $\V$ (resp. $\H$) to $\R$.
			\item[(H.3)]	\textsl{(Coercivity)}. There exists a positive constant $L_\A$ such that for any $\x\in\V$ and a.e.  $t\in[0,T]$, 
			\begin{align}\label{1.13}
				\langle \A(t,\x),\x\rangle 
				\leq f(t)(1+\|\x\|_{\H}^2)-L_\A\|\x\|_{\V}^\beta.
			\end{align}
			\item[(H.4)] \textsl{(Growth)}. There exist non-negative constants  $\alpha$ and $C$ such that for any $\x\in\V$ and a.e. $t\in[0,T]$,
			\begin{align}\label{1.14}
				\|\A(t,\x)\|_{\V^*}^{\frac{\beta}{\beta-1}}\leq \big(f(t)+C\|\x\|_{\V}^\beta\big)\big(1+\|\x\|_{\H}^\alpha\big).
			\end{align} 
			\item[(H.5)]There exists a non-negative constant $K$ such that for any  $\x\in \V$ and a.e. $t\in[0,T]$, 
			\begin{align}\label{1.15}
				\|\s(\x)\|_{\L_2}^2\leq K(1+\|\x\|_{\H}^2). 
			\end{align}
		\end{itemize}
	\end{hypothesis}

In order to obtain the well-posedness of approximating system \eqref{1.6}, and the Wong-Zakai approximation of the system \eqref{1.1}, we require the following assumptions on $\sigma_i(\cdot)$ $(i\in\N)$ and $\wi{\Tr}_m(\cdot)$ which are more general (Hypothesis \ref{hyp2} (H.7) below) in comparison to the work \cite{TMRZ}.  The examples of noise given in \cite[Section 3]{TMRZ} are also applicable in our framework.
\begin{hypothesis}\label{hyp2}
	For each $i\in\N$, the mapping $\s_i$ is twice Fr\'echet differentiable with its second Fr\'echet derivative denoted by $\D^2\s_i:\H\to\mathcal{L}(\H,\mathcal{L}(\H,\H)) \cong \mathcal{L}(\H\times\H,\H)$, satisfies the following:
	\begin{itemize}
		\item[(H.6)] for any $M>0$, there exists a positive constant $C_M$ such that 
		\begin{align*}
			\sup_{i\in\N}\sup_{\|\x\|\leq M} \big\{\|\D\s_j(\x)\|_{\mathcal{L}(\H,\H)}\vee \|\s_i(\x)\|_\H\vee \|\D^2\s_i(\x)\|_{\mathcal{L}(\H\times\H,\H)}\big\} \leq C_M,\\ 
			\D\s_i^*\big|_\V :\V\to\V,\ \ 	\sup_{i\in\N}\sup_{\|\x\|\leq M} \|\D\s_i(\x)^*\y\|_\V \leq C_M\|\y\|_\V, \ \ \y\in\V,
		\end{align*}for any $m\in\N$, 
	\begin{align*}
		\lim_{m\to\infty}\sup_{\|\x\|_\H\leq M} \|\s(\x)-\s(\x)\circ \Pi_m\|_{\L_2}=0,
	\end{align*}where $\Pi_m$ denotes the orthogonal projection onto $\U_m:=\mathrm{span}\{\ee_1,\ee_2,\ldots,\ee_m\}$ in $\U$, that is, $\Pi_m\x=\sum\limits_{j=1}^m(\x,\ee_j)\ee_j,\ \x\in \U$ and $\D\s_i(\cdot)^*$ denotes the dual of the operator $\D\s_i(\cdot)$;
\item[(H.7)] there exists a positive constant $L$ such that for every $m\in\N$ and $\x,\x_1,\x_2\in\H$, we have 
\begin{align*}
	\|\wi{\Tr}_m(\x)\|_\H^2 &\leq L\big(1+\|\x\|_\H^2\big),\\ 
	\big(\wi{\Tr}_m(\x_2)-\wi{\Tr}_m(\x_1),\x_1-\x_2\big)&\leq \big(\kappa(\x_1)+\varkappa(\x_2)\big)\|\x_1-\x_2\|_\H^2, 
\end{align*}where $\kappa(\cdot)$ and $\varkappa(\cdot)$ are same as defined in \eqref{1.12}.
	\end{itemize}
\end{hypothesis}

Let us provide the definition of a probabilistically strong solution of system \eqref{1.1}.
\begin{definition}\label{def2.3}
	Let $(\Omega,\mathscr{F},\{\mathscr{F}_t\}_{t\geq 0},\P)$ be a stochastic basis. Then, \eqref{1.1} has a \emph{probabilistically strong solution} if and only if there exists a progressively measurable process $\Y:[0,T]\times\Omega\to\H, $ with paths 
	\begin{align*}
		\Y(\cdot)\in \C([0,T];\H)\cap\L^\beta(0,T;\V),\ \P\text{-a.s., }  \ for \ \ \beta\in(1,\infty), 
	\end{align*}and the following equality holds $\P$-a.s., in $\V^*$, for all $t\in[0,T]$, 
\begin{align*}
	\Y(t)=\Y(0)+\int_0^t\A(s,\Y(s))\d s+\int_0^t\sigma(\Y(s))\d \W(s).
\end{align*}
\end{definition}

The following theorem provides the well-posedness results of the system \eqref{1.1}, we refer readers to \cite[Theorem 2.6]{MRSSTZ} for the proof.
\begin{theorem}[{\cite{MRSSTZ}}]\label{thrm1}
	Assume that Hypothesis \ref{hyp1} holds and the embedding $\V\hookrightarrow \H$ is compact.  Then, for any initial data $\boldsymbol{y}_0\in \H$, there exists a 	\textsl{probabilistically strong solution} to the system \eqref{1.1}. Furthermore, for any $p\geq 2$, the following energy estimate holds:
	\begin{align}\label{1.16}
		\E\bigg[\sup_{t\in[0,T]}\|\Y(t)\|_\H^p\bigg]+\E\bigg[\int_0^T\|\Y(t)\|_\V^\beta\d t\bigg]^\frac{p}{2}\leq C\big(1+\|\y_0\|_{\H}^p\big).
	\end{align}
\end{theorem}

\begin{remark}
	In the above result \cite[Theorem 2.6]{MRSSTZ}, the authors established the result for deterministic initial data, that is, $\y_0\in\H$. However the result holds true for the initial data $\y_0\in\L^p(\Omega;\H)$, for $p\geq 2$. 
\end{remark}

Under Hypotheses \ref{hyp1} and \ref{hyp2}, we obtain the following  well-posedness result for the approximating system \eqref{1.6}. 
	\begin{theorem}\label{thrm2}
		Assume that Hypotheses \ref{hyp1} and \ref{hyp2} hold. Then, for any initial data $\boldsymbol{y}_0\in\L^p(\Omega;\H)$ for $p>\max\big\{\frac{\beta}{\beta-1},2\big\}$, there exists a unique solution $\Y^m(\cdot)$ to the approximating system \eqref{1.6} with $\Y^m(0)=\boldsymbol{y}_0$, $\mathbb{P}$-a.s., and 
		\begin{align*}
			\Y^m(t)=\boldsymbol{y}_0+\int_0^t\A(s,\Y^m(s))\d s+\int_0^t \s(\Y^m(s))\dot{\W}^m(s)\d s-\frac{1}{2}\int_0^t\wi{\Tr}_m(\Y^m(s))\d s,\ \mathbb{P}\text{-a.s.},
		\end{align*}
in $\V^*$. 	Furthermore, for any $p>\max\big\{\frac{\beta}{\beta-1},2\big\},$ the following energy estimate holds:
	\begin{align}\label{1.17}
	\E\bigg[\sup_{t\in[0,T]}\|\Y^m(t)\|_\H^p\bigg]+\E\bigg[\int_0^T\|\Y^m(t)\|_\V^\beta\d t\bigg]\leq C_m\big(1+\E\big[\|\y_0\|_\H^p\big]\big).
	\end{align}
		\end{theorem}
	
Note that the  additional condition of $p>\frac{\beta}{\beta-1}$ helps us to obtain the drifts $\A(\cdot), \ \s\dot{\W}^m(\cdot)$  and $\wi{\Tr}_m(\cdot)$ in the same space, see the proof of Theorem \ref{thrm4.1} (Section \ref{Seclem4.1}).  
	
	Next, we consider a general system that contains the system \eqref{1.6} as a spacial case and helps us to prove the topological support of the distribution of solutions to the system \eqref{1.1} as well. Let us consider the following general system: 
	\begin{equation}\label{4.4}
		\left\{
		\begin{aligned}
			\d \X_{\g}^m(t)&=\A(t,\X_{\g}^m(t))\d t+\s_1(\X_{\g}^m(t))\d \W(t)+\s_2(\X_{\g}^m(t))\dot{\W}^m(t)\d t\\&\quad +\s_3(\X_{\g}^m(t))\g(t)\d t-\G(\X_{\g}^m(t))\d t, \ \ t\in(0,T),\\ 
			\X_{\g}^m(0)&=\boldsymbol{y}_0,
		\end{aligned}
		\right.
	\end{equation}where for any fixed time $T>0$, the maps 
	\begin{align*}
		\A:[0,T]\times\V\to\V^*, \ \ 
		\s_1,\s_2,\s_3:\H\to \L_2(\U,\H), \ \text{ and } \ \G:\H\to\H,
	\end{align*}are progressively measurable and $\g\in\L^2(0,T;\U)$.
	
	If we choose $\s_1=0, \ \s_2=\s, \ \s_3=0$ and $\G=\frac{1}{2}\wi{\Tr}_m$, then the system \eqref{4.4} reduced to the system \eqref{1.6}. Thus, Theorem \ref{thrm2} will be a particular case of the following Theorem \ref{thrm4.1}. If we write the drift coefficient  in \eqref{1.6} as $\wi{\A}(\cdot)=\A(\cdot)+\s(\cdot)\dot{\W}^m(\cdot)-\frac{1}{2}\wi{\Tr}_m(\cdot)$, and the diffusion $\wi{\s}=0$, then the system \eqref{1.6} changes to \eqref{1.1}. Here, we would like to mention that we cannot apply Theorem \ref{thrm1} directly to solve the system \eqref{1.6}. Since 
	\begin{align*}
		\big(\s(\x)\dot{\W}^m(t),\x\big) \leq \sqrt{K}\sqrt{1+\|\x\|_\H^2}\|\x\|_\H\|\dot{\W}^m(t)\|_\U,
	\end{align*} 
and we cannot find a bound of $\|\dot{\W}^m(t)\|_\U$, for all $\omega\in\Omega$ and $t\in[0,T]$, that is,	Hypothesis \ref{hyp1} (condition (H.3), coercivity) does not hold for the new drift coefficients $\wi{\A}(\cdot)$.  
	
	Now, we discuss the well-posedness of the system \eqref{4.4} in the following theorem. The proof will be given in the next section which is based on similar arguments used in \cite[Lemma 4.1]{TMRZ}.
	
	\begin{theorem}\label{thrm4.1}
		Let $T>0$, $\g\in\L^2(0,T;\U)$ and $\boldsymbol{y}_0\in\L^p(\Omega;\H)$ for $p>\max\big\{\frac{\beta}{\beta-1},2\big\}$ in Hypothesis \ref{hyp1}. Assume that the coefficient $\A(\cdot)$ satisfies Hypothesis \ref{hyp1}, $\s_1(\cdot),\ \s_2(\cdot),\ \s_3(\cdot)$ satisfy the conditions in Hypothesis \ref{hyp1} (\eqref{1.11} $\&$ \eqref{1.15}) and $\G(\cdot)$ satisfies Hypothesis \ref{hyp2} (H.7) with $\G(\cdot)$ replacing $\wi{\Tr}_m(\cdot)$. Then, there exists a unique solution $\X_{\g}^m(\cdot)$ to the system \eqref{4.4} in the sense of Definition \ref{def2.3}. Furthermore, $\X_{\g}^m\in\C([0,T];\H),\ \P$-a.s., and 
		\begin{align}\label{4.5}
	\E\bigg[\sup_{t\in[0,T]}\|\X_{\g}^m(t)\|_\H^p\bigg]+\E\bigg[\int_0^T\|\X_{\g}^m(t)\|_\V^\beta\d t\bigg]^{\frac{p}{2}}\leq C_m\big(1+\E\big[\|\y_0\|_\H^p\big]\big).
		\end{align}
	\end{theorem}
	\section{Proof of Theorem \ref{thrm4.1}}\label{Seclem4.1}\setcounter{equation}{0}
	The proof is based on a standard Feado-Galerkin approximation technique. Let $\{\bphi_1,\bphi_2,\ldots\}$ be an orthonormal basis of $\H$ and orthogonal  in  $\V$, set $\H_n:=\mathrm{span}\{\bphi_1,\bphi_2,\ldots,\bphi_n\}$. We define a projection  $\PP_n:\V^*\to\H_n$ by 
	\begin{align}\label{4.6}
		\PP_n\x=\sum_{i=1}^n\langle \x,\bphi_i\rangle \bphi_i, \ \ \x\in\V^*.
	\end{align}
	Clearly, the restriction of this projection denoted by $\PP_n\big|_\H$ is just the orthogonal projection of $\H$ onto $\H_n$. Since, $\{\ee_j\}_{j\in\N}$ is an orthonormal basis of the Hilbert space $\U$, recall the projection of $\U$ onto $\U_n$ 
	\begin{align}\label{4.006}
		\Pi_n \W(t)=\sum_{j=1}^n\langle \W(t),\ee_j\rangle_{\U} \ee_j.
	\end{align}For $\g\in\L^2(0,T;\U)$, we consider the following system in the finite-dimensional space $\H_n$, for any $n\in\N$:
\begin{equation}\label{4.7}
			\left\{
			\begin{aligned}
				\d \X_{\g,n}^m(t)&=\A^n(t,\X_{\g,n}^m(t))\d t+\s_1^n(\X_{\g,n}^m(t))\Pi_n\d \W(t)+\s_2^n(\X_{\g,n}^m(t))\dot{\W}^m(t)\d t\\&\quad +\s_3^n(\X_{\g,n}^m(t))\g(t)\d t-\G^n(\X_{\g,n}^m(t))\d t, \ \ t\in(0,T),\\ 
				\X_{\g,n}^m(0)&=\PP_n\boldsymbol{y}_0=:\boldsymbol{y}_0^n,
			\end{aligned}
			\right.
	\end{equation}
	where $\A^n(t,\cdot):=\PP_n\A(t,\PP_n(\cdot))$, similar for the other terms, $\PP_n$ and $\Pi_n$ are projections defined in \eqref{4.6} and \eqref{4.006}, respectively.
	We can also write the above system as for all $t\in[0,T]$, $\P$-a.s.,
	\begin{align}\label{4.007}\nonumber
		\X_{\g,n}^m(t)&=\PP_n\boldsymbol{y}_0+ \int_0^t \bigg\{\A^n(s,\X_{\g,n}^m(s))+\s_2^n(\X_{\g,n}^m(s))\dot{\W}^m(s)+\s_3^n(\X_{\g,n}^m(s))\g(s)\\&\qquad -\G^n(\X_{\g,n}^m(s))\bigg\}\d s+\int_0^t\s_1^n(\X_{\g,n}^m(s))\Pi_n\d \W(s).
	\end{align}For any $t\in[0,T]$, and $\E\big[\|\dot{\W}^m(t)\|_\U^2\big]=\frac{m}{\vp}$ and 
\begin{align*}
	\big((\s_2(\u)-\s_2(\v))\dot{\W}^m(t),\u-\v\big)&\leq \|\s_2(\u)-\s_2(\v)\|_{\L_2}\|\dot{\W}^m(t)\|_\U\|\u-\v\|_\H \\&\leq \big(\kappa(\u)+\varkappa(\v)\big)^\frac{1}{2}\|\dot{\W}^m(t)\|_\U\|\u-\v\|_\H^2\\&\leq 
	\big(\kappa(\u)+\varkappa(\v)+\|\dot{\W}^m(t)\|_\U^2\big)\|\u-\v\|_\H^2,
\end{align*}for $\u,\v\in\H_m$, where we have used Hypothesis \ref{hyp1} (H.2) and Young's inequality.
The existence of a unique solution of the finite-dimensional system \eqref{4.7} up to a stopping time $0<\tau^m\leq T$ has been discussed in \cite[Theorem 3.1.1]{WLMR2} (also see \cite[Theorem 3.1.1]{CPMR} and \cite[Theorem 2.6]{MRSSTZ}).
That is, the unique solution $\X_{\g,n}^m\in \mathrm{C}([0,\tau^m];\mathbb{H}_n),$ $\mathbb{P}$-a.s.,  and our aim is to show that $\tau^m=T$, $\mathbb{P}$-a.s. 

Since $\E\big[\|\dot{\W}^m(t)\|_\U^2\big]=\frac{m}{\vp}<\infty$, for all $t\in[0,T]$ and  each $m\in\mathbb{N}$, we first consider the following truncated problem: 
		\begin{equation}\label{3p5}
			\left\{
			\begin{aligned}
				\d \X_{\g,n}^{m,N}(t)&=\A^n(t,\X_{\g,n}^{m,N}(t))\d t+\s_1^n(\X_{\g,n}^{m,N}(t))\Pi_n\d \W(t)+\Phi_N(\|\dot{\W}^m(t)\|_{\U})\s_2^n(\X_{\g,n}^{m,N}(t))\dot{\W}^{m}(t)\d t\\&\quad +\s_3^n(\X_{\g,n}^{m,N}(t))\g(t)\d t-\G^n(\X_{\g,n}^{m,N}(t))\d t, \ \ t\in(0,T),\\ 
				\X_{\g,n}^{m,N}(0)&=\PP_n\boldsymbol{y}_0=:\boldsymbol{y}_0^n,
			\end{aligned}
			\right.
		\end{equation}
for $N\in\mathbb{N}$, where 
\begin{align}\label{TC}
	\Phi_N(y)=\left\{\begin{array}{cl}1,&0\leq y\leq N,\\ N+1-y,& N<y\leq N+1,\\
		0,&y\geq N+1. 
	\end{array}\right.
	\end{align}
As discussed in the previous case, there exists a unique local solution $\X_{\g,n}^{m,N}\in \mathrm{C}([0,\tau];\mathbb{H}_n),$ $\mathbb{P}$-a.s. to the problem \eqref{3p5}. For the above truncation method we are referring to the work \cite{JUK}.  
Let us first show that $\tau=T$, $\P$-a.s. The following energy estimate will help us to obtain this result.
	\begin{lemma}\label{lem4.2}
		Under the assumptions of Theorem \ref{thrm4.1}, there exist constants $C_p, C_{p,K,L},C_{p,K,L,T,N}>0$ such that 
	\begin{align}\label{4p8}\nonumber
			&\sup_{n,m\in\N}\E\bigg[\sup_{t\in[0,T]}\|\X_{\g,n}^{m,N}(t)\|^p_\H+\int_0^T\|\X_{\g,n}^{m,N}(t)\|_\H^{p-2}\|\X_{\g,n}^{m,N}(t)\|_\V^\beta\d t\bigg]\nonumber\\&\qquad+	\sup_{n,m\in\N}\E\bigg[\int_0^T\|\X_{\g,n}^{m,N}(t)\|_\V^\beta\d t\bigg]^{\frac{p}{2}}\nonumber \\&\leq \bigg(\E\big[\|\boldsymbol{y}_0\|_\H^p\big]+C_{p,K,L,T,N}+C_p\int_0^Tf(t)\d t\bigg)\nonumber\\&\qquad\times\exp\bigg\{C_{p,K,L}\int_0^T\big(1+f(t)+\|\g(t)\|_\U^2\big)\d t+C_{p,K,T,N}\bigg\}.
		\end{align}
	\end{lemma}
	\begin{proof}
		Note that 
		\begin{align*}
			\langle \A^n(t,\x),\y\rangle =\langle \A(t,\x),\y\rangle,\ \text{ for all }\  \x,\y\in\H_n.
		\end{align*}Similar equalities holds for the coefficients $\s_2^n(\cdot), \ \s_3^n(\cdot)$ and $\G^n(\cdot)$. 
		
		\vspace{2mm}
		\noindent
		\textbf{Step 1:} Let us first define the following sequence of stopping times: for $M\in\mathbb{N}$
		\begin{align}
			\tau_n^M&:=T\wedge \inf\left\{t\geq 0:\|\X_{\g,n}^{m,N}(t)\|_\H^2>M\right\}.\label{36}
		\end{align}
Applying finite-dimensional It\^o's formula to the process $\|\X_{\g,n}^{m,N}(\cdot)\|_\H^p$ (cf. \cite[Theorem 2.1]{IGDS}), we find  for all $t\in[0,\tau_{n}^{M}]$, $\mathbb{P}$-a.s.,
		\begin{align}\label{4.9}\nonumber
			\|\X_{\g,n}^{m,N}(t)\|_\H^p &\leq  \|\boldsymbol{y}_0\|_\H^p+ p\int_0^t\|\X_{\g,n}^{m,N}(s)\|_\H^{p-2}\langle \A(s,\X_{\g,n}^{m,N}(s)),\X_{\g,n}^{m,N}(s)\rangle \d s\\&\nonumber\quad + p\int_0^t\|\X_{\g,n}^{m,N}(s)\|_\H^{p-2}\big(\s_1^n(\X_{\g,n}^{m,N}(s))\Pi_n\d\W(s),\X_{\g,n}^{m,N}(s)\big) \d s\\&\nonumber\quad +p\int_0^t\|\X_{\g,n}^{m,N}(s)\|_\H^{p-2}\big(\Phi_N(\|\dot{\W}^m(s)\|_\U)\s_2^n(\X_{\g,n}^{m,N}(s))\dot{\W}^m(s),\X_{\g,n}^{m,N}(s)\big) \d s\\&\nonumber\quad +p\int_0^t\|\X_{\g,n}^{m,N}(s)\|_\H^{p-2}\big(\s_3^n(\X_{\g,n}^{m,N}(s))\g(s),\X_{\g,n}^{m,N}(s)\big)\d s \\&\nonumber\quad - p\int_0^t\|\X_{\g,n}^{m,N}(s)\|_\H^{p-2}\big(\G^n(\X_{\g,n}^{m,N}(s)),\X_{\g,n}^{m,N}(s)\big)\d s\\&\nonumber\quad 
			+\frac{p}{2}\int_0^t\|\X_{\g,n}^{m,N}(s)\|_\H^{p-2}\|\s_1^n(\X_{\g,n}^{m,N}(s))\Pi_n\|_{\L_2}^2\d s\\&\nonumber\quad +\frac{p(p-2)}{2} \int_0^t\|\X_{\g,n}^{m,N}(s)\|_\H^{p-4}\|\left(\s_1^n(\X_{\g,n}^{m,N}(s))\Pi_n\right)^*\X_{\g,n}^{m,N}(s)\|_\U^2\d s\\& =: \|\boldsymbol{y}_0\|_{\H}^p+\sum_{j=1}^7 J_{m,n}^j(t),
		\end{align}
	where we have used the fact that $\|\PP_n\boldsymbol{y}_0\|_\H=\|\boldsymbol{y}_0^n\|_\H\leq \|\boldsymbol{y}_0\|_\H$. 
		
		Let us consider the term $J_{m,n}^1(\cdot)$ and estimate it using Hypothesis \ref{hyp1} (H.3), H\"older's and Young's inequalities  as
		\begin{align}\label{4.10}
				|J_{m,n}^1(t)|&\nonumber\leq \frac{p}{2}\int_0^t\|\X_{\g,n}^{m,N}(s)\|_\H^{p-2}\big[f(s)\big(1+\|\X_{\g,n}^{m,N}(s)\|_\H^2\big)-L_\A\|\X_{\g,n}^{m,N}(s)\|_\V^\beta\big]\d s\\&\leq -\frac{p L_\A}{2}\int_0^t\|\X_{\g,n}^{m,N}(s)\|_\H^{p-2}\|\X_{\g,n}^{m,N}(s)\|_\V^\beta\d s\nonumber\\&\quad+C_p\int_0^tf(s)\d s+C_p\int_0^tf(s)\|\X_{\g,n}^{m,N}(s)\|_\H^p\d s.
		\end{align}
	We consider the term $\E\bigg[\sup\limits_{t\in[0,\tau^M_n]}\big|J_{m,n}^2(t)\big|\bigg]$ and estimate it using Burkholder-Davis-Gundy's (BDG) inequality (see \cite[Theorem 1.1]{DLB}), Hypothesis \ref{hyp1} (H.5), H\"older's and Young's inequalities as 
		\begin{align}\label{4.11}\nonumber
			&\E\bigg[\sup_{t\in[0,\tau^M_n]}\big|J_{m,n}^2(t)\big|\bigg]\\&\nonumber\leq C_p \E\bigg[\bigg(\int_0^{\tau^M_n}\|\X_{\g,n}^{m,N}(s)\|_\H^{2p-2}\|\s_1^n(\X_{\g,n}^{m,N}(s))\|_{\L_2}^2\d s \bigg)^{\frac{1}{2}}\bigg] \\&\nonumber \leq C_p \sqrt{K} \E\bigg[\bigg(\sup_{s\in[0,\tau^M_n]} \|\X_{\g,n}^{m,N}(s)\|_\H^p\int_0^{\tau^M_n}\|\X_{\g,n}^{m,N}(s)\|_\H^{p-2}\|\s_1^n(\X_{\g,n}^{m,N}(s))\|_{\L_2}^2\d s\bigg)^\frac{1}{2}\bigg]	 \\&\nonumber\leq \e\E\bigg[\sup_{s\in[0,\tau^M_n]} \|\X_{\g,n}^{m,N}(s)\|^p\bigg]+C_{\e,p,K}\E\bigg[\int_0^{\tau^M_n}\|\X_{\g,n}^{m,N}(s)\|_\H^{p-2}\big(1+\|\X_{\g,n}^{m,N}(s)\|_\H^2\big)\d s\bigg]
			\\&\leq \e \E\bigg[\sup_{s\in[0,\tau^M_n]}\|\X_{\g,n}^{m,N}(s)\|_\H^p\bigg]+C_{\e,p,K,T}+C_{\e,p,K}\E\bigg[\int_0^{\tau^M_n}\|\X_{\g,n}^{m,N}(s)\|_\H^p\d s\bigg],
		\end{align}where $\e>0$ is small enough. 
Next, we consider the term  $\E\bigg[\sup\limits_{s\in[0,\tau_n^M]}\big|J_{m,n}^3(s)\big|\bigg]$ and estimating it using the definition of truncation $\Phi_N(\cdot)$ \eqref{TC}, Hypothesis \ref{hyp1} (H.5), and Young's inequality as
	\begin{align*}
		\E\bigg[\sup_{s\in[0,\tau_n^M]}\big|J_{m,n}^3(s)\big|\bigg] &\leq \E\bigg[\int_0^{\tau_n^M}\Phi_N(\|\dot{\W}^m(s)\|_\U)\|\X_{\g,n}^{m,N}(s)\|_\H^{p-1}\|\s_2^n(\X_{\g,n}^{m,N}(s))\dot{\W}^m(s)\|_\H\d s\bigg| \bigg]  \\& \leq \E\bigg[\int_{0}^{\tau_n^M}\Phi_N(\|\dot{\W}^m(s)\|_\U)\|\X_{\g,n}^{m,N}(s)\|_\H^{p-1}\|\s_2^n(\X_{\g,n}^{m,N}(s))\|_{\L_2}\|\dot{\W}^m(s)\|_\U\d s\bigg]
	 \\& \leq C_{K,N}\E\bigg[\int_{0}^{\tau_n^M}\|\X_{\g,n}^{m,N}(s)\|_\H^{p-1}\big(1+\|\X_{\g,n}^{m,N}(s)\|_{\H}\big)\d s\bigg]\\&\leq 
	 C_{p,K,N}\bigg\{T+\E\bigg[\int_{0}^{\tau_n^M}\|\X_{\g,n}^{m,N}(s)\|_\H^{p}\d s\bigg]\bigg\}.
	\end{align*}
	Let us consider the term $J_{m,n}^4(\cdot)$, and estimate it using Hypothesis \ref{hyp1} (H.5) and Young's inequality as 
		\begin{align}\label{4.13}\nonumber
			|J_{m,n}^4(t)|& \leq p\sqrt{K}\int_0^t\|\X_{\g,n}^{m,N}(s)\|_\H^{p-1}\big(1+\|\X_{\g,n}^{m,N}(s)\|_\H\big)\|\g(s)\|_\U\d s \\&\leq C_{p,K,T} +C_{p,K}\int_0^t\|\g(s)\|_\U^2\|\X_{\g,n}^{m,N}(s)\|_\H^p\d s.
		\end{align}  We consider the term $J_{m,n}^5(\cdot)$, and estimate it using Hypothesis \ref{hyp2} (H.7) and Young's inequality as
		\begin{align}\label{4.14}
			|J_{m,n}^5(t)| \leq \e\sup_{t\in[0,\tau^M_n]}\|\X_{\g,n}^{m,N}(s)\|_\H^p+C_{\e,p,L,T}+C_{\e,p,L}\int_0^{\tau^M_n}\|\X_{\g,n}^{m,N}(s)\|_\H^p\d s.
		\end{align}
			For the remaining two terms $J_{m,n}^6(\cdot)$ and $J_{m,n}^7(\cdot),$ we use Hypothesis \ref{hyp1} (H.5) and estimate it as 
		\begin{align}\label{4.15}
			|J_{m,n}^6(t)|+|J_{m,n}^7(t)| \leq C_{p,K,T}+C_{p,K}\int_0^{\tau^M_n}\|\X_{\g,n}^{m,N}(s)\|_\H^p\d s.
		\end{align}

	Substituting \eqref{4.10}-\eqref{4.15} into \eqref{4.9},  and  applying Gronwall's inequality with a proper choice of $\e>0$, we obtain
	\begin{align}\label{4.015}\nonumber
			&\sup_{n,m\in\N}\E\bigg[\sup_{t\in[0,\tau^M_n]}\|\X_{\g,n}^{m,N}(t)\|^p_\H+\int_0^{\tau^M_n}\|\X_{\g,n}^{m,N}(t)\|_\H^{p-2}\|\X_{\g,n}^{m,N}(t)\|_\V^\beta\d t\bigg] \\&\leq \bigg(\E\big[\|\boldsymbol{y}_0\|_\H^p\big]+C_{p,K,L,T,N}+C_p\int_0^Tf(t)\d t\bigg)\nonumber\\&\quad\times\exp\bigg\{C_{p,K,L}\int_0^T\big(1+f(t)+\|\g(t)\|_\U^2\big)\d t+C_{p,K,T,N}\bigg\}.
		\end{align}
Since	$\tau_n^M\to T,$ $\mathbb{P}$-a.s., as $M\to\infty$, taking $M\to\infty$ in \eqref{4.015} and using the monotone convergence theorem, one can deduce the first part of \eqref{4p8}. 
		
		\vskip 2mm
		\noindent
		\textbf{Step 2:} Applying finite-dimensional It\^o's formula to the process $\|\X_{\g,n}^{m,N}(\cdot)\|_\H^2$, one can show that (cf. \cite[Lemma 2.10]{MRSSTZ} or \cite[Lemma 3.1]{AKMTM4})
		\begin{align}\label{4.0015}
			\sup_{n,m\in\N}\E\bigg[\int_0^T\|\X_{\g,n}^{m,N}(s)\|_\V^\beta\d s\bigg]^{\frac{p}{2}}\leq C<\infty,
		\end{align}where the constant $C=C_{p,K,T,L,N, \y_0,f}$. Combining \eqref{4.015} and the above inequality, we obtain the required result \eqref{4p8}.
	\end{proof}
 Let us now define a sequence of stopping times
	\begin{align}\label{3p17}
		\tau_N^m:=\left\{\begin{array}{l}\inf\left\{t\geq 0:\|\dot{\W}^m(t)\|_{\U}\geq N\right\},\\
		\infty,\ \text{ if the set is empty.}\end{array}\right.
		\end{align}
	By the uniqueness of solutions of the problem \eqref{3p5}, for $3\leq N_1\leq N_2$, we find 
	\begin{align*}
		\X_{\g,n}^{m,N_1}(t)=\X_{\g,n}^{m,N_2}(t), \ \text{ for all }\ t\in[0,\tau^m_{N_1}\wedge\tau^m_{N_2}], \ \P\text{-a.s.,}
		\end{align*}
	so that 
	\begin{align*}
		\tau^m_{N_1}\leq \tau^m_{N_2}, \ \P\text{-a.s.}
		\end{align*}
	Let us now define 
	\begin{align*}
		\tau^m(\omega)=\lim_{N\to\infty}\tau^m_N(\omega), \ \text{ for a.a. }\ \omega\in\Omega, 
		\end{align*}
	so that 
	\begin{align*}
		\X_{\g,n}^m(t)=\left\{\begin{array}{cl}
			\lim\limits_{N\to\infty}\X_{\g,n}^{n,N}(t),&\text{ for }\ 0\leq t<\tau^m,\\
			0,&\text{ for }\ t\geq\tau^m,
		\end{array}\right.
	\end{align*}
by the uniqueness of the solutions to the system \eqref{4.7}.  Finally, we show that
\begin{align}\label{global}
\mathbb{P}\{\omega\in\Omega:\tau^m(\omega)=T\}=1. 
\end{align} 
We know that 
\begin{align*}
	\E\left[\chi_{\{\tau^m_N<t\}}\right] =\P \left\{\omega\in\Omega: \tau^m_N(\omega)<t \right\},
	\end{align*}
	and using the definition of the stopping time $\tau^m_N$ defined in \eqref{3p17}, we have 
	\begin{align}\label{3p18}
		\E\left[\|\dot{\W}^m(t)\|_{\U}^2\right]&=\E\left[\|\dot{\W}^m(t)\|_{\U}^2\chi_{\{\tau^m_N<t\}}\right]+\E\left[\|\dot{\W}^m(t)\|_{\U}^2\chi_{\{\tau^m_N\geq t\}}\right]\nonumber\\&\geq \E\left[\|\dot{\W}^m(t)\|_{\U}^2\chi_{\{\tau^m_N<t\}}\right]\geq N^2\P \left\{\omega\in\Omega: \tau^m_N(\omega)<t \right\}. 
	\end{align}
Using  \eqref{3p18} and $\E\big[\|\dot{\W}^m(t)\|_\U^2\big]=\frac{m}{\vp}=\frac{m2^m}{T}<\infty$, we find for each $m\in\mathbb{N}$ and any $t\in[0,T]$
\begin{align*}
	\P\{\omega\in\Omega:\tau^m(\omega)<t\}&\leq \P\{\omega\in\Omega:\tau^m_N(\omega)<t\}\nonumber\\&\leq\frac{1}{N^2}\E\left[\|\dot{\W}^m(t)\|_{\U}^2\right]=\frac{m\vp^{-1}}{N^2}\to 0,\ \text{ as }\ N\to\infty,
	\end{align*}
so that \eqref{global} follows. By \eqref{4p8} and a contradiction argument to the fact that  $\mathbb{P}\{\omega\in\Omega:\tau^m(\omega)=T\}=1,$ we further have
	\begin{align}\label{4.8}
	&\sup_{n\in\N}\E\bigg[\sup_{t\in[0,T]}\|\X_{\g,n}^m(t)\|^p_\H+\int_0^T\|\X_{\g,n}^m(t)\|_\H^{p-2}\|\X_{\g,n}^m(t)\|_\V^\beta\d t\bigg]\nonumber\\&\quad+	\sup_{n\in\N}\E\bigg[\int_0^T\|\X_{\g,n}^m(t)\|_\V^\beta\d t\bigg]^{\frac{p}{2}}\leq C_m<\infty,
	\end{align}
for each $m\in\mathbb{N}$.

	Let us set 
	\begin{align*}
		&\mathcal{M}=\L^\beta(\Omega;\L^\beta(0,T;\V)), \  	\mathcal{M}^*=\L^\frac{\beta}{\beta-1}(\Omega;\L^\frac{\beta}{\beta-1}(0,T;\V^*)), \text{ and } \mathcal{H}=\L^2(\Omega;\L^2(0,T;\L_2(\U,\H))).
	\end{align*}By the estimate \eqref{4.8} and Hypothesis \ref{hyp2}, for all $n\in\N$, we have 
	\begin{align*}
		\|\X_{\g,n}^m(\cdot)\|_{\mathcal{M}}+\|\A(\cdot,\X_{\g,n}^m(\cdot))\|_{\mathcal{M}^*}\leq C_m<\infty.
	\end{align*}
	Let $\lambda:=\frac{\beta}{p(\beta-1)}<1$, since $p>\frac{\beta}{\beta-1}$. Using  Hypothesis \ref{hyp1} (H.5) and Young's inequality,  for $t\in[0,T],$ we deduce 
	\begin{align}\label{4.16}\nonumber
		\|\s_2(\x)\dot{\W}^m(t)\|_\H^{\frac{\beta}{\beta-1}}&\leq\big( \lambda \|\s_2(\x)\|_{\L_2}^{\frac{1}{\lambda}}+C_\lambda  \|\dot{\W}^m(t)\|_\U^\frac{1}{1-\lambda}\big)^{\frac{\beta}{\beta-1}} \\&\nonumber\leq 
		2^{\frac{1}{\beta-1}}\left[\lambda^{\frac{\beta}{\beta-1}}K^{\frac{\beta}{2\lambda(\beta-1)}}\big(1+\|\x\|_\H^2\big)^\frac{\beta}{2\lambda(\beta-1)}+C_\lambda ^{\frac{\beta}{\beta-1}}\|\dot{\W}^m(t)\|_\U^\frac{\beta}{(\beta-1)(1-\lambda)}\right]\\&\leq C_{\lambda,\beta,K}\left(1+\|\x\|_\H^p+\|\dot{\W}^m(t)\|_\U^\frac{\beta p}{\beta p-\beta-p}\right),
	\end{align}for all $\x\in \H$, where we have used the following inequality $$ (a+b)^{\frac{\beta}{\beta-1}} \leq 2^\frac{1}{\beta-1}\big(a^{\frac{\beta}{\beta-1}}+b^{\frac{\beta}{\beta-1}}\big), \ \text{ for }\ a,b\geq0.$$
	
	Recalling from \eqref{1.3} that the random variables $\|\dot{\W}^m(l\vp)\|_{\U}$, (for $l=1,2,\ldots,2^m$) are independent centered Gaussian random variables with $\E\big[\|\dot{\W}^m(l\vp)\|_\U^2\big]=m\vp^{-1}$. Therefore, there exists a positive constant $C_{\beta,p}$ such that
	\begin{align}\label{4.17}\nonumber
		\E\bigg[\int_0^T\|\dot{\W}^m(t)\|_\U^\frac{\beta p}{\beta p-\beta-p}\d t\bigg]&=\sum_{l=1}^{2^m}\vp\E\big[\|\dot{\W}^m(l\vp)\|_\U^\frac{\beta p}{\beta p-\beta-p}\big]\\&\nonumber=\vp C_{\beta,p} \sum_{l=1}^{2^m}\left(\E\big[\|\dot{\W}^m(l\vp)\|_\U^2\big]\right)^\frac{\beta p}{2(\beta p-\beta-p)}\\&=2^m\vp\left(\frac{m}{\vp}\right)^\frac{\beta p}{2(\beta p-\beta-p)}C_{\beta,p}.
	\end{align}Using Hypothesis \ref{hyp2} (H.7), we have 
	\begin{align}\label{4.18}
		\|\G(\x)\|_\H^\frac{\beta}{\beta-1} \leq L\big(1+\|\x\|_\H^2\big)^\frac{\beta}{2(\beta-1)}, \ \ \text{ for any } \ \ \x\in\H.
	\end{align}Let us assume $C_1$ be the positive constant such that $\|\cdot\|_{\V^*}\leq C_1\|\cdot\|_\H$ (since $\H\hookrightarrow\V^*$), again by the assumption $p>\frac{\beta}{\beta-1}$ and \eqref{4.8}, for each  $m\in \mathbb{N}$, we arrive at 
	\begin{align*}
		\|\s_2(\X_{\g,n}^m(\cdot))\dot{\W}^m(\cdot)+\G(\X_{\g,n}^m(\cdot))\|_{\mathcal{M}^*}\leq C_m <\infty, \  \text{uniformly for }   \ n\in\N.
	\end{align*}

Our next aim is to establish the tightness of the laws of $\{\X_{\g,n}^m\}_{n\in\N}$, for each $m\in\N$ in the space $\C([0,T];\V^*)\cap\L^\beta(0,T;\H)$, for $\beta\in(1,\infty)$.
	
	Let us define a sequence of stopping times for $N\in\mathbb{N}$
	\begin{align}\label{4.19}
		\tau_n^N:=T\wedge \inf\left\{t\geq0: \|\X_{\g,n}^m(t)\|_\H^2>N\right\}\wedge \inf\left\{t\geq0:\int_0^t \|\X_{\g,n}^m(s)\|_\V^\beta\d s>N\right\},
	\end{align}
with the convention that infimum over an empty set  is infinite. By Markov's inequality and \eqref{4.8}, we have 
	\begin{align}\label{4.20}
		\lim_{N\to\infty}\sup_{n\in\N}\P\big\{\tau_n^N<T\big\}=0.
	\end{align}
	\begin{lemma}\label{lem4.3}
		For each $m\in\N$, the laws of the approximating solutions $\{\X_{\g,n}^m\}_{n\in\N}$ is tight in the space $\C([0,T];\V^*)\cap\L^\beta(0,T;\H)$, for $\beta\in(1,\infty)$.
	\end{lemma}
	\begin{proof}We divided the proof into two steps.
		
		\vspace{2mm}
		\noindent
		\textbf{Step 1.} In this step,  we prove the tightness of $\{\X_{\g,n}^m\}_{n\in\N}$, for each $m\in\N$ in the space $\C([0,T];\V^*)$.

		We first prove the tightness of laws of $\{\X_{\g,n}^m\}_{n\in\N}$ in the space $\C([0,T];\V^*)$. Since,  the embedding $\H\hookrightarrow\V^*$ is compact and 
		\begin{align}\label{4.21}
			\lim_{N\to\infty}\sup_{n\in\N}\P\bigg\{\sup_{t\in[0,T]}\|\X_{\g,n}^m(t)\|_\H>\sqrt{N}\bigg\} \leq 	\lim_{N\to\infty}\sup_{n\in\N}\P\big\{\tau_n^N<T\big\}=0,
		\end{align} by \cite[Theorem 3.1]{AJ}, it is enough to prove that for every $\bphi\in\H_n,\ n\in\N$, $\{\langle\X_{\g,n}^m(\cdot),\bphi\rangle  \}$ is tight in the space $\C([0,T];\R)$. In view of \eqref{4.21} and Aldou's tightness criterion (see \cite[Theorem 1]{DA}), it suffices to prove that for any stopping time $0\leq \gamma_n\leq T$ and for any $\e>0$, 
		\begin{align}\label{4.22}
			\lim_{\delta\to0}\sup_{n\in\N}\P\big\{\big|\langle\X_{\g,n}^m(\gamma_n+\delta)-\X_{\g,n}^m(\gamma_n), \bphi\rangle \big|
			>\e\big\}=0,
		\end{align}where $\gamma_n+\delta:=T\wedge(\gamma_n+\delta)\vee0$. Using Markov's inequality, we deduce 
		\begin{align}\label{4.23}\nonumber
		&	\P\big\{\big|\langle\X_{\g,n}^m(\gamma_n+\delta)-\X_{\g,n}^m(\gamma_n), \bphi\rangle \big|
			>\e\big\}\\&\nonumber\leq \P\big\{\big|\langle\X_{\g,n}^m(\gamma_n+\delta)-\X_{\g,n}^m(\gamma_n), \bphi\rangle \big|>\e, \ \tau_n^N= T\big\}+\P\big\{\tau_n^N<T\big\} \\&\leq \frac{1}{\e^\beta}\E\big[\big|\langle\X_{\g,n}^m((\gamma_n+\delta)\wedge\tau_n^N)-\X_{\g,n}^m(\gamma_n\wedge\tau_n^N), \bphi\rangle \big|^\beta\big]+\P\big\{\tau_n^N<T\big\}. 
		\end{align}
		Using \eqref{4.007}, BDG inequality, we obtain 
		\begin{align}\label{4.24}\nonumber
			&\E\big[\big|\langle\X_{\g,n}^m((\gamma_n+\delta)\wedge\tau_n^N)-\X_{\g,n}^m(\gamma_n\wedge\tau_n^N), \bphi\rangle \big|^\beta\big] \\&\nonumber\leq C_\beta \E\bigg[\bigg(\int_{\gamma_n\wedge\tau_n^N}^{(\gamma_n+\delta)\wedge\tau_n^N} \big|\langle \A^n(s,\X_{\g,n}^m(s)),\bphi\rangle\big|\d s\bigg)^\beta\bigg]\\&\nonumber\quad +C_\beta \E\bigg[\bigg(\int_{\gamma_n\wedge\tau_n^N}^{(\gamma_n+\delta)\wedge\tau_n^N} \|\bphi\|_\H^2\|\s_1^n(\X_{\g,n}^m(s))\Pi_n\|_{\L_2}^2\d s\bigg)^{\frac{\beta}{2}}\bigg]\\&\nonumber\quad+ C_\beta \E\bigg[\bigg(\int_{\gamma_n\wedge\tau_n^N}^{(\gamma_n+\delta)\wedge\tau_n^N}\|\bphi\|_\H\|\s_2^n(\X_{\g,n}^m(s))\dot{\W}^m(s)\|_\H\d s\bigg)^\beta\bigg]\\&\nonumber\quad +C_\beta \E\bigg[\bigg(\int_{\gamma_n\wedge\tau_n^N}^{(\gamma_n+\delta)\wedge\tau_n^N}\|\bphi\|_\H\|\s_3^n(\X_{\g,n}^m(s))\g(s)\|_\H\d s\bigg)^\beta\bigg]\\&\nonumber\quad +C_\beta \E\bigg[\bigg(\int_{\gamma_n\wedge\tau_n^N}^{(\gamma_n+\delta)\wedge\tau_n^N}\|\bphi\|_\H\|\G^n(\X_{\g,n}^m(s))\|_\H\d s\bigg)^\beta\bigg]\\&=: I_n+II_n+III_n+IV_n+V_n.
		\end{align}Since $\bphi\in\H_n$, which gives $\sup\limits_{n\in\N}\|\PP_n\bphi\|_\V<\infty$.
		Using H\"older's inequality, Hypothesis \ref{hyp1} (H.4) and the definition of stopping times \eqref{4.19}, we find
		\begin{align}\label{4.25}\nonumber
			I_n&\leq C_\beta \E\bigg[|\delta|\bigg(\int_{\gamma_n\wedge\tau_n^N}^{(\gamma_n+\delta)\wedge\tau_n^N}\big|\langle \A(s,\X_{\g,n}^m(s)),\PP_n\bphi\rangle\big|^{\frac{\beta}{\beta-1}}\d s\bigg)^{\beta-1}\bigg]\\&\nonumber\leq C_\beta |\delta|\E\bigg[\bigg(\int_0^{\tau_n^N}\|\PP_n\bphi\|_\V^{\frac{\beta}{\beta-1}}\big(f(s)+C\|\X_{\g,n}^m(s)\|_\V^\beta\big)\big(1+\|\X_{\g,n}^m(s)\|_\H^\alpha\big)\d s\bigg)^{\beta-1}\bigg]\\&\leq C_{\beta,N} |\delta|.
		\end{align}Using H\"older's inequality, Hypothesis \ref{hyp1} (H.5) and the definition of stopping times \eqref{4.19}, we find
		\begin{align}\label{4.26}\nonumber
			II_n&\leq C_\beta K^\frac{\beta}{2}\E\bigg[\bigg(\int_{\gamma_n\wedge\tau_n^N}^{(\gamma_n+\delta)\wedge\tau_n^N}\|\bphi\|_\H^2\big(1+\|\X_{\g,n}^m(s)\|_\H^2\big)\d s\bigg)^{\frac{\beta}{2}}\bigg] \\&\leq C_{\beta,N,K}|\delta|^{\frac{\beta}{2}} .
		\end{align}Using H\"older's inequality, Hypothesis \ref{hyp1} (H.5), \eqref{4.16}  and \eqref{4.17} for $p=4$, we estimate the term $III_n$ as
		\begin{align}\label{4.27}\nonumber
			III_n&\leq  C_\beta\E\bigg[|\delta|^{\frac{\beta}{2}}\bigg(\int_{\gamma_n\wedge\tau_n^N}^{(\gamma_n+\delta)\wedge\tau_n^N}\|\s_2^n(\X_{\g,n}^m(s))\dot{\W}^m(s)\|_\H^2\d s\bigg)^\frac{\beta}{2}\bigg] \\&\nonumber\leq C_\beta |\delta|^{\frac{\beta}{2}} \E\bigg[\bigg(\int_{\gamma_n\wedge\tau_n^N}^{(\gamma_n+\delta)\wedge\tau_n^N}\big(1+\|\X_{\g,n}^m(s)\|_\H^4+\|\dot{\W}^m(s)\|_\H^4\big)\d s\bigg)^\frac{\beta}{2}\bigg] \\&\nonumber\leq  C_\beta |\delta|^{\frac{\beta}{2}}\left\{\E\bigg[\bigg(\int_{\gamma_n\wedge\tau_n^N}^{(\gamma_n+\delta)\wedge\tau_n^N}\big(1+\|\X_{\g,n}^m(s)\|_\H^4\big)\bigg)^{\frac{\beta}{2}}\bigg]	\right.\\&\nonumber\qquad\left.+\E\bigg[\bigg(\int_{\gamma_n\wedge\tau_n^N}^{(\gamma_n+\delta)\wedge\tau_n^N}\|\dot{\W}^m(s)\|_\H^4\d s\bigg)^\frac{\beta}{2}\bigg]\right\}\\&\leq  C_{\beta}\left(C_N+2^m\vp^{-1}m^2\right)|\delta|^{\frac{\beta}{2}}.
		\end{align}Now, we consider the penultimate term $IV_n$ on the right hand side of the inequality \eqref{4.24}, and estimate it using Hypothesis \eqref{hyp1} (H.5) and Young's inequality as
		\begin{align*}
			IV_n&\leq C_\beta 	\E\bigg[\bigg(\int_{\gamma_n\wedge\tau_n^N}^{(\gamma_n+\delta)\wedge\tau_n^N}\|\bphi\|_\H\sqrt{K}\big(1+\|\X_{\g,n}^m(s)\|_\H\big) \|\g(s)\|_\U\d s\bigg)^\beta\bigg] \\&\leq  C_{\beta,K}\E\bigg[\bigg(|\delta|+\int_{\gamma_n\wedge\tau_n^N}^{(\gamma_n+\delta)\wedge\tau_n^N}\big(1+\|\X_{\g,n}^m(s)\|_\H\big)^2 \|\g(s)\|_\U^2\d s\bigg)^\beta\bigg] \\&\leq C_{\beta,K} \bigg\{|\delta|^\beta+C_N\bigg(\int_{\gamma_n\wedge\tau_n^N}^{(\gamma_n+\delta)\wedge\tau_n^N}\|\g(s)\|_\U^2\d s\bigg)^\beta\bigg\}.
		\end{align*}Note that $\g\in\L^2(0,T;\U)$. By the absolute continuity of Lebesgue integrable functions, we obtain 
		\begin{align}\label{4.29}
			\lim_{\delta\to0}\sup_{n\in\N}IV_n=0.
		\end{align}
	Similarly, using H\"older's inequality and \eqref{4.18}
		\begin{align}\label{4.30}\nonumber
			V_n&\leq 	C_\beta \E\bigg[|\delta|\bigg(\int_{\gamma_n\wedge\tau_n^N}^{(\gamma_n+\delta)\wedge\tau_n^N}\|\G^n(\X_{\g,n}^m(s))\|_\H^\frac{\beta}{\beta-1}\d s\bigg)^{\beta-1}\bigg]\\&\nonumber\leq|\delta| C_{\beta,L} \E\bigg[\bigg(\int_{\gamma_n\wedge\tau_n^N}^{(\gamma_n+\delta)\wedge\tau_n^N}\big(1+\|\X_{\g,n}^m(s)\|_\H^2\big)^\frac{\beta}{2(\beta-1)}\d s\bigg)^{\beta-1}\bigg]\\&\leq C_{\beta,N,L}|\delta|^{\beta}.
		\end{align}Substituting \eqref{4.25}-\eqref{4.30} in \eqref{4.24}, we find 
		\begin{align}\label{4.31}
			\lim_{\delta\to0}\sup_{n\in\N}\E\left[\big|\langle\X_{\g,n}^m((\gamma_n+\delta)\wedge\tau_n^N)-\X_{\g,n}^m(\gamma_n\wedge\tau_n^N), \bphi\rangle \big|^\beta\right]=0.
		\end{align}In view of \eqref{4.20} and \eqref{4.31}, passing $\delta\to0$ and then $N\to\infty$ in \eqref{4.23}, we obtain the required result of the Step 1, that is, the laws of approximating sequence $\{\X_{\g,n}^m\}_{n\in\N}$ is tight in the space $\C([0,T];\V^*)$. 
		
		\vspace{2mm}
		\noindent
		\textbf{Step 2.} We move to  the proof of tightness of the laws of the approximated solutions $\{\X_{\g,n}^m\}_{n\in\N}$ in the space $\L^\beta(0,T;\H)$, for $\beta\in(1,\infty)$. By the estimate \eqref{4.8}, we have 
		\begin{align}\label{4.031}
			\sup_{n\in\N} \E\bigg[\int_0^T\|\X_{\g,n}^m(t)\|_\V^\beta\d t\bigg]\leq C_m<\infty, \ \ \text{ for } \ \ \beta\in(1,\infty).
		\end{align}
		In view of \cite[Lemma 5.2]{MRSSTZ}, it is enough to show that for each $m\in\N$ and  for any $\e>0$, 
		\begin{align}\label{4.32}
			\lim_{\delta\to0}\sup_{n\in\N}\P\bigg\{\int_0^{T-\delta}\|\X_{\g,n}^m(t+\delta)-\X_{\g,n}^m(t)\|_\H^\beta\d t>\e\bigg\}=0.
		\end{align}
	An application of Markov's inequality yields 
		\begin{align}\label{4.33}\nonumber
			&	\P\bigg\{\int_0^{T-\delta}\|\X_{\g,n}^m(t+\delta)-\X_{\g,n}^m(t)\|_\H^\beta\d t>\e\bigg\}\\&\nonumber\leq \P\bigg\{\int_0^{T-\delta}\|\X_{\g,n}^m(t+\delta)-\X_{\g,n}^m(t)\|_\H^\beta\d t>\e, \tau_n^N=T\bigg\} +\P\big\{\tau_n^N<T\big\}\\&\leq \frac{1}{\e}\E\bigg[\int_0^{T-\delta}\|\X_{\g,n}^m((t+\delta)\wedge\tau_n^N)-\X_{\g,n}^m(t\wedge\tau_n^N)\|_\H^\beta\d t\bigg]+\P\big\{\tau_n^N<T\big\}.
		\end{align}In view of \eqref{4.20} and \eqref{4.33}, if we are able to prove 
		\begin{align}\label{4.34}
			\lim_{\delta\to0}\sup_{n\in\N}\E \bigg[\int_0^{T-\delta}\|\X_{\g,n}^m((t+\delta)\wedge\tau_n^N)-\X_{\g,n}^m(t\wedge\tau_n^N)\|_\H^\beta\d t\bigg]=0,
		\end{align}then we are done.
		
		Now, our main focus is to establish \eqref{4.34}. We divide the proof into two cases depending on the values of $\beta$, that is, we provide  the proof for $\beta\in(1,2]$ and $\beta\geq2$, separately.
		
		\vspace{2mm}
		\noindent \textbf{Case 1.} \textit{When $\beta\in(1,2]$}.  Applying finite-dimensional It\^o's formula to the process
		$\|\X_{\g,n}^m(\cdot)-\X_{\g,n}^m(t\wedge\tau_n^N)\|_\H^2$, and then taking expectation, we find 
		\begin{align}\label{4.35}\nonumber
			&	\E\left[\|\X_{\g,n}^m((t+\delta)\wedge\tau_n^N)-\X_{\g,n}^m(t\wedge\tau_n^N)\|_\H^2\right]\\&\nonumber= \E\bigg[\int_{t\wedge\tau_n^N}^{(t+\delta)\wedge\tau_n^N}\bigg(2\langle \A(l,\X_{\g,n}^m(l)),\X_{\g,n}^m(l)-\X_{\g,n}^m(t\wedge\tau_n^N)\rangle+\big(\s_2^n(\X_{\g,n}^m(l))\dot{\W}^m(l)\\&\nonumber\qquad +\s_3^n(\X_{\g,n}^m(l))\g(l)-\G^n(\X_{\g,n}^m(l)),\X_{\g,n}^m(l)-\X_{\g,n}^m(t\wedge\tau_n^N)\big)\bigg) \d l\bigg] \\&\quad +\E\bigg[\int_{t\wedge\tau_n^N}^{(t+\delta)\wedge\tau_n^N}\|\s_1^n(\X_{\g,n}^m(l))\Pi_n\|_{\L_2}^2\d l\bigg].
		\end{align}From \eqref{4.35}, we deduce that 
		\begin{align}\label{4.36}\nonumber
			&	\E\bigg[\int_0^{T-\delta}\|\X_{\g,n}^m((t+\delta)\wedge\tau_n^N)-\X_{\g,n}^m(t\wedge\tau_n^N)\|_\H^2\d t\bigg]\\&\nonumber= \E\bigg[\int_0^{T-\delta}\bigg\{\int_{t\wedge\tau_n^N}^{(t+\delta)\wedge\tau_n^N}\bigg(2\langle \A(l,\X_{\g,n}^m(l)),\X_{\g,n}^m(l)\rangle+\|\s_1^n(\X_{\g,n}^m(l))\Pi_n\|_{\L_2}^2\bigg)\d l\bigg\}\d t\bigg]\\&\nonumber\quad +2\E\bigg[\int_0^{T-\delta}\bigg\{\int_{t\wedge\tau_n^N}^{(t+\delta)\wedge\tau_n^N}\big(\s_2^n(\X_{\g,n}^m(l))\dot{\W}^m(l),\X_{\g,n}^m(l)\big)\d l\bigg\}\d t\bigg]\\&\nonumber\quad +2\E\bigg[\int_0^{T-\delta}\bigg\{\int_{t\wedge\tau_n^N}^{(t+\delta)\wedge\tau_n^N}\big(\s_3^n(\X_{\g,n}^m(l))\g(l),\X_{\g,n}^m(l)\big)\d l\bigg\}\d t\bigg]\\&\nonumber\quad +2\E\bigg[\int_0^{T-\delta}\bigg\{\int_{t\wedge\tau_n^N}^{(t+\delta)\wedge\tau_n^N}\big(\G^n(\X_{\g,n}^m(l)),\X_{\g,n}^m(l)\big)\d l\bigg\}\d t\bigg] \\&\nonumber\quad -2\E\bigg[\int_0^{T-\delta}\bigg\{\int_{t\wedge\tau_n^N}^{(t+\delta)\wedge\tau_n^N}\langle \A(l,\X_{\g,n}^m(l)),\X_{\g,n}^m(t\wedge\tau_n^N)\rangle\d l\bigg\}\d t\bigg]\\&\nonumber \quad
			-2\E\bigg[\int_0^{T-\delta}\bigg\{\int_{t\wedge\tau_n^N}^{(t+\delta)\wedge\tau_n^N}\big( \s_2^n(\X_{\g,n}^m(l))\dot{\W}^m(l),\X_{\g,n}^m(t\wedge\tau_n^N)\big)\d l\bigg\}\d t\bigg]\\& \nonumber\quad
			-2\E\bigg[\int_0^{T-\delta}\bigg\{\int_{t\wedge\tau_n^N}^{(t+\delta)\wedge\tau_n^N}\big( \s_3^n(\X_{\g,n}^m(l))\g(l),\X_{\g,n}^m(t\wedge\tau_n^N)\big)\d l\bigg\}\d t\bigg]\\&\nonumber \quad
			+2\E\bigg[\int_0^{T-\delta}\bigg\{\int_{t\wedge\tau_n^N}^{(t+\delta)\wedge\tau_n^N}\big( \G^n(\X_{\g,n}^m(l)),\X_{\g,n}^m(t\wedge\tau_n^N)\big)\d l\bigg\}\d t\bigg]
			\\&=: \sum_{j=1}^8I_j.
		\end{align}Using Hypothesis \ref{hyp1} (H.3) and (H.5), Fubini's theorem and \eqref{4.8}, we estimate the term $I_1$ as 
		\begin{align}\label{4.37}\nonumber
			|I_1|&=\E\bigg[\int_0^{\tau_n^N}\bigg(\int_{0\vee(l-\delta)}^l\chi_{\{\tau_n^N>t\}}\d t\bigg)\bigg(2\langle \A(l,\X_{\g,n}^m(l)),\X_{\g,n}^m(l)\rangle+\|\s_1^n(\X_{\g,n}^m(l))\Pi_n\|_{\L_2}^2\bigg)\d l\bigg]\\&\nonumber\leq \delta\E\bigg[\int_0^{\tau_n^N} \big(f(l)+K\big)\big(1+\|\X_{\g,n}^m(l)\|_\H^2\big)\d l\bigg]\\&\leq \delta \bigg(\int_0^T f(l)\d l+KT\bigg)\bigg(1+\E\bigg[\sup_{l\in[0,\tau_n^N]}\|\X_{\g,n}^m(l)\|_\H^2\bigg]\bigg) \leq C\delta .
		\end{align}
	Now, we consider the term $|I_2|$ and estimate it using Fubini's theorem, Hypothesis \ref{hyp1} (H.5) and the definition of stopping time $\tau_n^N$  as
		\begin{align}\label{4.38}\nonumber
				|I_2|&\leq  2\E\bigg[\int_0^{\tau_n^N}\bigg(\int_{0\vee(l-\delta)}^l\chi_{\{\tau_n^N>t\}}\d t\bigg)\big|\big(\s_2^n(\X_{\g,n}^m(l))\dot{\W}^m(l),\X_{\g,n}^m(l)\big)\big|\d l\bigg]\\&\nonumber\leq C \delta  \E\bigg[\int_0^{\tau_n^N}\|\X_{\g,n}^m(l)\|_\H\|\s_2^n(\X_{\g,n}^m(l))\dot{\W}^m(l)\|_\H\d l\bigg]
				\\&\nonumber\leq C \delta  \E\bigg[\int_0^{\tau_n^N}\|\X_{\g,n}^m(l)\|_\H\|\s_2^n(\X_{\g,n}^m(l))\|_{\L_2}\|\dot{\W}^m(l)\|_\U\d l\bigg]  \\&\leq C \delta\bigg(1+C\bigg\{\int_0^{t}\E\big[\|\dot{\W}^m(l)\|_\U^2\big] \d l\bigg\}^{\frac{1}{2}}\bigg) \leq C\delta .
		\end{align}
	We estimate $I_3$ using similar arguments to \eqref{4.13} and \eqref{4.8} as
		\begin{align}\label{4.39}
				|I_3|&\leq 2\E\bigg[\int_0^{\tau_n^N}\bigg(\int_{0\vee(l-\delta)}^l\chi_{\{\tau_n^N>t\}}\d t\bigg)\big|\big(\s_3^n(\X_{\g,n}^m(l))\g(l),\X_{\g,n}^m(l)\big)\big|\d l\bigg]\nonumber\\&\nonumber\leq C \delta \E \bigg[\int_0^{\tau_n^N}\big(1+\|\X_{\g,n}^m(l)\|_\H\big)\|\X_{\g,n}^m(l)\|_\H\|\g(l)\|_\U\d l\bigg]  \\&\leq C \delta T^{\frac{1}{2}}\bigg(\int_0^T\|\g(l)\|_\U^2\d l\bigg)^{\frac{1}{2}}\bigg(1+\E\bigg[\sup_{l\in[0,{\tau_n^N}]}\|\X_{\g,n}^m(l)\|_\H^2\bigg]\bigg) \leq  C\delta.
		\end{align}Again, using Fubini's theorem and \eqref{4.18}, we estimate the term $I_4$ as 
		\begin{align}\label{4.40}\nonumber
			|I_4| &\leq 	2\E\bigg[\int_0^{\tau_n^N}\bigg(\int_{0\vee(l-\delta)}^l\chi_{\{\tau_n^N>t\}}\d t\bigg)\|\G^n(\X_{\g,n}^m(l))\|_\H\|\X_{\g,n}^m(l)\|_\H\d l\bigg]\\&\leq  C\delta\E\bigg[\int_0^{\tau_n^N} \big(1+2\|\X_{\g,n}^m(l)\|_\H^2\big)\d l\bigg] \leq  C\delta \bigg(1+\E\bigg[\sup_{l\in[0,\tau_n^N]}\|\X_{\g,n}^m(l)\|_\H^2\bigg]\bigg)\leq C\delta.
		\end{align}Applying Fubini's theorem and Hypothesis \ref{hyp1} (H.4), we estimate the term $I_5$ as 
		\begin{align}\label{4.41}\nonumber
			|I_5| &\leq 2\E\bigg[\bigg|\int_0^{\tau_n^N}\bigg(\int_{0\vee(l-\delta)}^l \chi_{\{\tau_n^N>t\}}\langle \A(l,\X_{\g,n}^m(l)),\X_{\g,n}^m(t\wedge\tau_n^N)\rangle \bigg)\d t\d l\bigg|\bigg]\\& \nonumber\leq 2\E\bigg[\int_0^{\tau_n^N}\|\A(l,\X_{\g,n}^m(l))\|_{\V^*}\bigg(\int_{0\vee(l-\delta)}^l\|\X_{\g,n}^m(t\wedge\tau_n^N)\|_\V\d t\bigg)\d l\bigg]\\&\nonumber\leq 2\delta^{\frac{\beta-1}{\beta}}\E\bigg[\int_0^{\tau_n^N}\|\A(l,\X_{\g,n}^m(l))\|_{\V^*}\bigg(\int_0^{\tau_n^N}\|\X_{\g,n}^m(t)\|_\V^\beta\d t\bigg)^\frac{1}{\beta}\d l\bigg] \\&\nonumber\leq 2\delta^{\frac{\beta-1}{\beta}}T^{\frac{1}{\beta}}\left\{\E\bigg[\int_0^{\tau_n^N}\|\A(l,\X_{\g,n}^m(l))\|_{\V^*}^{\frac{\beta}{\beta-1}}\d l\bigg]\right\}^{\frac{\beta-1}{\beta}}\left\{\E\bigg[\int_0^{\tau_n^N}\|\X_{\g,n}^m(t)\|_\V^\beta\d t\bigg]\right\}^{\frac{1}{\beta}}\\&\leq C\delta^{\frac{\beta-1}{\beta}}.
		\end{align}Now, we consider the term $I_6$, and estimate it using Fubini's theorem and  Hypothesis \ref{hyp1} (H.5) as
		\begin{align}\label{4.42}\nonumber
			|I_6| &\leq 2\E\bigg[\bigg|\int_0^{\tau_n^N}\int_{0\vee(l-\delta)}^l \chi_{\{\tau_n^N>t\}}\big( \s_2^n(\X_{\g,n}^m(l))\dot{\W}^m(l),\X_{\g,n}^m(t\wedge\tau_n^N)\big)\d l\d t\bigg|\bigg] \\&\nonumber\leq 2\E\bigg[\int_0^{\tau_n^N}\|\s_2^n(\X_{\g,n}^m(l))\dot{\W}^m(l)\|_\H\bigg(\int_{0\vee(l-\delta)}^l\|\X_{\g,n}^m(t\wedge\tau_n^N)\|_\H\d t\bigg)\d l\bigg]\\&\nonumber\leq 2\delta^{\frac{1}{2}}\E\bigg[\int_0^{\tau_n^N}\|\s_2^n(\X_{\g,n}^m(l))\dot{\W}^m(l)\|_\H\bigg(\int_0^{\tau_n^N}\|\X_{\g,n}^m(t)\|_\H^2\d t\bigg)^\frac{1}{2}\d l\bigg]\\&\nonumber\leq 2\delta^{\frac{1}{2}}T^{\frac{1}{2}}\left\{\E\bigg[\int_0^{\tau_n^N}\|\s_2^n(\X_{\g,n}^m(l))\dot{\W}^m(l)\|_\H^2\d l\bigg]\right\}^\frac{1}{2}\left\{\E\bigg[\int_0^{\tau_n^N}\|\X_{\g,n}^m(t)\|_\H^2\d t\bigg]\right\}^\frac{1}{2}\\&\leq C \delta^\frac{1}{2},
		\end{align}
	where we have used \eqref{4.16} also. Similarly, we can estimate the term $I_7$ as 
		\begin{align}\label{4.43}\nonumber
			|I_7| &\leq 2\E\bigg[\bigg|\int_0^{\tau_n^N}\int_{0\vee(l-\delta)}^l \chi_{\{\tau_n^N>t\}}\big( \s_3^n(\X_{\g,n}^m(l))\g(l),\X_{\g,n}^m(t\wedge\tau_n^N)\big)\d l\d t\bigg|\bigg]\\&\nonumber\leq 2\E \bigg[\int_0^{\tau_n^N}\|\s_3^n(\X_{\g,n}^m(l))\g(l)\|_\H\bigg(\int_{0\vee(l-\delta)}^l\|\X_{\g,n}^m(t\wedge\tau_n^N)\|_\H\d t\bigg)\d l\bigg]\\&\nonumber\leq 2\delta^{\frac{1}{2}}T^{\frac{1}{2}}\left\{\E\bigg[\int_0^{\tau_n^N}\|\s_3^n(\X_{\g,n}^m(l))\g(l)\|_\H^2\d l\bigg]\right\}^{\frac{1}{2}}\left\{\E\bigg[\int_0^{\tau_n^N}\|\X_{\g,n}^m(t)\|_\H^2\d t\bigg]\right\}^\frac{1}{2} \\&\nonumber\leq K\delta^{\frac{1}{2}}T^{\frac{1}{2}}\left\{\bigg(\int_0^{T}\|\g(l)\|_\U^2\d l\bigg)\bigg(1+\E\bigg[\sup_{l\in[0,\tau_n^N]}\|\X_{\g,n}^m(l)\|_\H^2\bigg]\bigg)\right\}^{\frac{1}{2}}\left\{\E\bigg[\int_0^{\tau_n^N}\|\X_{\g,n}^m(t)\|_\H^2\d t\bigg]\right\}^\frac{1}{2} \\&\leq C\delta^{\frac{1}{2}}.
		\end{align}Again, using similar arguments as in the above inequality, we estimate the term $I_8$ as 
		\begin{align}\label{4.44}\nonumber
			|I_8|&\leq 2\E\bigg[\bigg|\int_0^{\tau_n^N}\int_{0\vee(l-\delta)}^l \chi_{\{\tau_n^N>t\}}\big( \G^n(\X_{\g,n}^m(l)),\X_{\g,n}^m(t\wedge\tau_n^N)\big)\d l\d t\bigg|\bigg] \\&\nonumber\leq 2\E \bigg[\int_0^{\tau_n^N}\|\G^n(\X_{\g,n}^m(l))\|_\H\bigg(\int_{0\vee(l-\delta)}^l\|\X_{\g,n}^m(t\wedge\tau_n^N)\|_\H\d t\bigg)\d l\bigg]\\&\nonumber \leq 2\delta^{\frac{1}{2}}T^{\frac{1}{2}}\left\{\E\bigg[\int_0^{\tau_n^N}\|\G^n(\X_{\g,n}^m(l))\|_\H^2\d l\bigg]\right\}^{\frac{1}{2}}\left\{\E\bigg[\int_0^{\tau_n^N}\|\X_{\g,n}^m(t)\|_\H^2\d t\bigg]\right\}^\frac{1}{2}\\&\nonumber\leq \delta^{\frac{1}{2}}C\left\{\bigg(1+\E\bigg[\sup_{l\in[0,\tau_n^N]}\|\X_{\g,n}^m(l)\|_\H^2\bigg]\right\}^{\frac{1}{2}}\left\{\E\bigg[\int_0^{\tau_n^N}\|\X_{\g,n}^m(t)\|_\H^2\d t\bigg]\right\}^\frac{1}{2}
			\\&	\leq C\delta^{\frac{1}{2}}.
		\end{align}Combining \eqref{4.36}-\eqref{4.44}, we obtain 
		\begin{align}\label{4.45}
			\sup_{n\in\N}\E\bigg[\int_0^{T-\delta}\|\X_{\g,n}^m((t+\delta)\wedge\tau_n^N)-\X_{\g,n}^m(t\wedge\tau_n^N)\|_\H^2\d t\bigg]\leq C\big(\delta+\delta^{\frac{1}{2}}+ \delta^{\frac{\beta-1}{\beta}}\big).
		\end{align}For $\beta\in(1,2)$, we apply H\"older's inequality to find 
		\begin{align}\label{4.46}\nonumber
			&\lim_{\delta\to0}\sup_{n\in\N}\E \bigg[\int_0^{T-\delta}\|\X_{\g,n}^m((t+\delta)\wedge\tau_n^N)-\X_{\g,n}^m(t\wedge\tau_n^N)\|_\H^\beta\d t\bigg]\\&\leq C	\sup_{n\in\N}\left\{\E\bigg[\int_0^{T-\delta}\|\X_{\g,n}^m((t+\delta)\wedge\tau_n^N)-\X_{\g,n}^m(t\wedge\tau_n^N)\|_\H^2\d t\bigg]\right\}^{\frac{\beta}{2}}=0,
		\end{align}which completes the proof of \eqref{4.34}, for $\beta\in(1,2]$.
		
		\vspace{2mm}
		\noindent \textbf{Case 2.} \textit{When $\beta>2$.} Again, applying It\^o's formula to the process $\|\X_{\g,n}^m(\cdot)-\X_{\g,n}^m(t\wedge\tau_n^N)\|_\H^2$, and then taking expectation, we find
		\begin{align}\label{4.47}\nonumber
			&	\E\left[\|\X_{\g,n}^m((t+\delta)\wedge\tau_n^N)-\X_{\g,n}^m(t\wedge\tau_n^N)\|_\H^\beta\right]\\&\nonumber=\frac{\beta}{2} \E\bigg[\int_{t\wedge\tau_n^N}^{(t+\delta)\wedge\tau_n^N}\|\X_{\g,n}^m(l)-\X_{\g,n}^m(t\wedge\tau_n^N)\|_\H^{\beta-2}\bigg(2\langle \A(l,\X_{\g,n}^m(l)),\X_{\g,n}^m(l)-\X_{\g,n}^m(t\wedge\tau_n^N)\rangle\\&\nonumber\quad+\big(\s_2^n(\X_{\g,n}^m(l))\dot{\W}^m(l) +\s_3^n(\X_{\g,n}^m(l))\g(l)-\G^n(\X_{\g,n}^m(l)),\X_{\g,n}^m(l)-\X_{\g,n}^m(t\wedge\tau_n^N)\big)\bigg) \d l\bigg] \\&\nonumber\quad +\frac{\beta}{2}\E\bigg[\int_{t\wedge\tau_n^N}^{(t+\delta)\wedge\tau_n^N}\|\X_{\g,n}^m(l)-\X_{\g,n}^m(t\wedge\tau_n^N)\|_\H^{\beta-2}\|\s_1^n(\X_{\g,n}^m(l))\Pi_n\|_{\L_2}^2\d l\bigg]\\&\nonumber\quad + \frac{\beta(\beta-2)}{2}\E\bigg[\int_{t\wedge\tau_n^N}^{(t+\delta)\wedge\tau_n^N}\|\X_{\g,n}^m(l)-\X_{\g,n}^m(t\wedge\tau_n^N)\|_\H^{\beta-4}\\&\qquad\times\|(\s_1^n(\X_{\g,n}^m(l))\Pi_n)^*(\X_{\g,n}^m(l)-\X_{\g,n}^m(t\wedge\tau_n^N))\|_\U^2\d l\bigg].
		\end{align}
	Applying Fubini's theorem in \eqref{4.47}, we obtain \small{
		\begin{align}\label{4.48}\nonumber
			&\E\bigg[\int\limits_0^{T-\delta}\|\X_{\g,n}^m((t+\delta)\wedge\tau_n^N)-\X_{\g,n}^m(t\wedge\tau_n^N)\|_\H^\beta\d t\bigg]\\&\nonumber\leq  \frac{\beta}{2} \E\bigg[\int\limits_0^{T-\delta}\bigg\{\int\limits_{t\wedge\tau_n^N}^{(t+\delta)\wedge\tau_n^N}\|\X_{\g,n}^m(l)-\X_{\g,n}^m(t\wedge\tau_n^N)\|_\H^{\beta-2}\\&\nonumber\qquad \times\bigg(2\langle \A(l,\X_{\g,n}^m(l)),\X_{\g,n}^m(l)\rangle+(\beta-1)\|\s_1^n(\X_{\g,n}^m(l))\Pi_n\|_{\L_2}^2\bigg)\d l\bigg\}\d t\bigg] \\&\nonumber\quad +\frac{\beta}{2}\E\bigg[\int\limits_0^{T-\delta}\bigg\{\int\limits_{t\wedge\tau_n^N}^{(t+\delta)\wedge\tau_n^N}\|\X_{\g,n}^m(l)-\X_{\g,n}^m(t\wedge\tau_n^N)\|_\H^{\beta-2}\big(\s_2^n(\X_{\g,n}^m(l))\dot{\W}^m(l),\X_{\g,n}^m(t\wedge\tau_n^N)\big)\d l\bigg\}\d t\bigg]\\&\nonumber\quad +\frac{\beta}{2}\E\bigg[\int\limits_0^{T-\delta}\bigg\{\int\limits_{t\wedge\tau_n^N}^{(t+\delta)\wedge\tau_n^N}\|\X_{\g,n}^m(l)-\X_{\g,n}^m(t\wedge\tau_n^N)\|_\H^{\beta-2}\big(\s_3^n(\X_{\g,n}^m(l))\g(l),\X_{\g,n}^m(t\wedge\tau_n^N)\big)\d l\bigg\}\d t\bigg]\\&\nonumber\quad +\frac{\beta}{2}\E\bigg[\int\limits_0^{T-\delta}\bigg\{\int\limits_{t\wedge\tau_n^N}^{(t+\delta)\wedge\tau_n^N}\|\X_{\g,n}^m(l)-\X_{\g,n}^m(t\wedge\tau_n^N)\|_\H^{\beta-2}\big(\G^n(\X_{\g,n}^m(l)),\X_{\g,n}^m(t\wedge\tau_n^N)\big)\d l\bigg\}\d t\bigg]\\&\nonumber\quad 
			-\beta  \E\bigg[\int\limits_0^{T-\delta}\bigg\{\int\limits_{t\wedge\tau_n^N}^{(t+\delta)\wedge\tau_n^N}\|\X_{\g,n}^m(l)-\X_{\g,n}^m(t\wedge\tau_n^N)\|_\H^{\beta-2}\langle \A(l,\X_{\g,n}^m(l)),\X_{\g,n}^m(t\wedge\tau_n^N)\rangle\d l\bigg\}\d t\bigg]\\&\nonumber\quad -\frac{\beta }{2} \E\bigg[\int\limits_0^{T-\delta}\bigg\{\int\limits_{t\wedge\tau_n^N}^{(t+\delta)\wedge\tau_n^N}\|\X_{\g,n}^m(l)-\X_{\g,n}^m(t\wedge\tau_n^N)\|_\H^{\beta-2}\big(\s_2^n(\X_{\g,n}^m(l))\dot{\W}^m(l),\X_{\g,n}^m(t\wedge\tau_n^N)\big)\d l\bigg\}\d t\bigg]\\&\nonumber\quad -\frac{\beta }{2} \E\bigg[\int\limits_0^{T-\delta}\bigg\{\int\limits_{t\wedge\tau_n^N}^{(t+\delta)\wedge\tau_n^N}\|\X_{\g,n}^m(l)-\X_{\g,n}^m(t\wedge\tau_n^N)\|_\H^{\beta-2}\big(\s_3^n(\X_{\g,n}^m(l))\g(l),\X_{\g,n}^m(t\wedge\tau_n^N)\big)\d l\bigg\}\d t\bigg]\\&\nonumber\quad +\frac{\beta }{2} \E\bigg[\int\limits_0^{T-\delta}\bigg\{\int\limits_{t\wedge\tau_n^N}^{(t+\delta)\wedge\tau_n^N}\|\X_{\g,n}^m(l)-\X_{\g,n}^m(t\wedge\tau_n^N)\|_\H^{\beta-2}\big(\G^n(\X_{\g,n}^m(l)),\X_{\g,n}^m(t\wedge\tau_n^N)\big)\d l\bigg\}\d t\bigg]\\&=:\sum_{i=1}^8J_i.
		\end{align}} 
	\normalsize
Using Hypothesis \ref{hyp1} (H.3), (H.5), Fubini's theorem and \eqref{4.8}, we estimate the term $J_1$ as
		\begin{align}\label{4.49}\nonumber
			&	|J_1|\\&\nonumber\leq C_\beta\E\bigg[\int_0^{\tau_n^N}\bigg(\delta\|\X_{\g,n}^m(l)\|_\H^{\beta-2}+\delta\sup_{l\in[0,T\wedge\tau_n^N]}\|\X_{\g,n}^m(l)\|_\H^{\beta-2}\bigg) \big(f(l)+K\big)\big(1+\|\X_{\g,n}^m(l)\|_\H^2\big)\d l\bigg]\\&\nonumber\leq 
		 C_\beta\delta\E\bigg[\int_0^{\tau_n^N}\big(f(l)+K\big)\big(1+\|\X_{\g,n}^m(l)\|_\H^2\big)\|\X_{\g,n}^m(l)\|_\H^{\beta-2}\d l\bigg]\\&\nonumber\quad+  C_\beta\delta \E\bigg[\sup_{l\in[0,\tau_n^N]}\left\{\|\X_{\g,n}^m(l)\|_\H^{\beta-2}\big(1+\|\X_{\g,n}^m(l)\|_\H^2\big)\right\}\int_0^T\big(f(l)+K\big)\d l\bigg] \\&\leq C \delta .
		\end{align}Using Hypothesis \ref{hyp1} (H.5), \eqref{4.38}-\eqref{4.40} and the arguments similar to \eqref{4.49}, we obtain the following bounds: 
		\begin{align}\label{4.50}
			|J_2|\leq C\delta , \ |J_3| \leq C\delta , \ \text{ and } \  |J_4|\leq C\delta .
		\end{align}Now, we consider the term $J_5$ and estimate it using Fubini's theorem as
		\begin{align}\label{4.51}\nonumber
			&|J_5|\\&\nonumber\leq \beta  \E\bigg[\bigg|\int_0^{\tau_n^N}\int_{0\vee(l-\delta)}^{l}\chi_{\{\tau_n^N>t\}}\|\X_{\g,n}^m(l)-\X_{\g,n}^m(t\wedge\tau_n^N)\|_\H^{\beta-2}\\&\nonumber\qquad\quad \times\langle \A(l,\X_{\g,n}^m(l)),\X_{\g,n}^m(t\wedge\tau_n^N)\rangle\d t\d l\bigg|\bigg]\\&\nonumber\leq C_\beta \E\bigg[\int_0^{\tau_n^N}\int_{0\vee(l-\delta)}^{l}\chi_{\{\tau_n^N>t\}}\|\X_{\g,n}^m(l)\|_\H^{\beta-2}\|\A(l,\X_{\g,n}^m(l))\|_{\V^*}\|\X_{\g,n}^m(t\wedge\tau_n^N)\|_\V\d t\d l\bigg]\\&\nonumber\quad + C_\beta \E\bigg[\int_0^{\tau_n^N}\int_{0\vee(l-\delta)}^{l}\chi_{\{\tau_n^N>t\}}\|\X_{\g,n}^m(t\wedge\tau_n^N)\|_\H^{\beta-2}\|\A(l,\X_{\g,n}^m(l))\|_{\V^*}\|\X_{\g,n}^m(t\wedge\tau_n^N)\|_\V\d t\d l\bigg]\\& =: J_{51}+J_{52}.
		\end{align}Using H\"older's inequality and Hypothesis \ref{hyp1} (H.5), we estimate the term $J_{51}$ as
		\begin{align}\label{4.52}\nonumber
			&	|J_{51}|\\&\nonumber\leq C_\beta\E\bigg[\sup_{l\in[0,\tau_n^N]}\|\X_{\g,n}^m(l)\|_\H^{\beta-2}\int_0^{\tau_n^N}\|\A(l,\X_{\g,n}^m(l))\|_{\V^*}\bigg(\int_{0\vee(l-\delta)}^{l}\|\X_{\g,n}^m(t)\|_\V\d t\bigg)\d l\bigg]
			\\&\nonumber\leq  C_\beta\delta^\frac{\beta-1}{\beta}\E\bigg[\sup_{l\in[0,\tau_n^N]}\|\X_{\g,n}^m(l)\|_\H^{\beta-2}\int_0^{\tau_n^N}\|\A(l,\X_{\g,n}^m(l))\|_{\V^*}\d l\bigg(\int_0^{\tau_n^N}\|\X_{\g,n}^m(t)\|_\V^\beta\d t\bigg)^{\frac{1}{\beta}}\bigg]\\&\nonumber\leq 
			C_\beta\delta^\frac{\beta-1}{\beta}T^\frac{\beta-1}{\beta}\left\{\E\bigg[\sup_{l\in[0,\tau_n^N]}\|\X_{\g,n}^m(l)\|_\H^\frac{\beta(\beta-2)}{2}\bigg]\right\}^{\frac{2}{\beta}}\left\{\E\bigg[\int_0^{\tau_n^N}\|\A(l,\X_{\g,n}^m(l))\|_{\V^*}^{\frac{\beta}{\beta-1}}\d l\bigg]\right\}^{\frac{\beta-1}{\beta}}\\&\nonumber\qquad\times \left\{\E\bigg[\int_0^{\tau_n^N}\|\X_{\g,n}^m(t)\|_\V^\beta\d t\bigg]^{\frac{1}{2}}\right\}^{\frac{2}{\beta}}\\&\leq C\delta^\frac{\beta-1}{\beta}.
		\end{align}
	Similarly, $J_{52}$ can be estimated as 
		\begin{align}\label{4.53}\nonumber
			|J_{52}|&\leq C_\beta\delta^\frac{\beta-1}{\beta}\E\bigg[\int_0^{ \tau_n^N}\|\A(l,\X_{\g,n}^m(l))\|_{\V^*}\d l \bigg(\int_0^{T\wedge \tau_n^N}\|\X_{\g,n}^m(t)\|_\V^\beta\d t\bigg)^\frac{1}{\beta}\bigg]\\&\leq C\delta^\frac{\beta-1}{\beta}.
		\end{align}Substituting \eqref{4.52} and \eqref{4.53} in \eqref{4.51}, we obtain 
		\begin{align}\label{4.54}
			|J_5| \leq C\delta^{\frac{\beta-1}{\beta}}.
		\end{align}
	A calculation similar to \eqref{4.53} along with Hypothesis \ref{hyp1} (H.5) and \eqref{4.42}-\eqref{4.44} gives
		\begin{align}\label{4.55}
			|J_6|+|J_7|+|J_8|\leq C  \delta^\frac{1}{2}. 
		\end{align}Combining \eqref{4.48}-\eqref{4.55}, we arrive at 
		\begin{align}\label{4.56}
			\sup_{n\in\N} \E\bigg[\int_0^{T-\delta}\|\X_{\g,n}^m(l)-\X_{\g,n}^m(t\wedge\tau_n^N)\|_\H^\beta\d s\bigg]\leq C\big(\delta+\delta^\frac{1}{2}+\delta^\frac{\beta-1}{\beta}\big).
		\end{align}Hence, \eqref{4.34} holds for $\beta\geq 2$.
	\end{proof}
	
	Let us set 
	\begin{align*}
	\Lambda =\C([0,T];\V^*)\cap \L^\beta(0,T;\H)\times\C([0,T];\U_1),
	\end{align*}where $\U_1$ is an another separable Hilbert space such that the embedding $\U\hookrightarrow\U_1$ is Hilbert-Schmidt (see \cite[Remark 2.5.1]{WLMR2}). By Lemma \ref{lem4.3}, for $\W^n(\cdot):=\W(\cdot)$ ($n\in\N$), we deduce that the family of the laws $\{\mathscr{L}(\X_{\g,n}^m,\W^n)\}_{n\in\N}$ of the random vectors $\{(\X_{\g,n}^m,\W^n)\}_{n\in\N}$ is tight in $\Lambda$. By Prokhorov's theorem (see \cite[Section 5]{PB}) and a version of Skorokhod's representation theorem (see \cite[Theorem C.1]{ZBEHPAR1} or  \cite[Theorem A.1]{PNKTRT}), we can find a new probability space $(\wi{\Omega},\wi{\mathscr{F}},\wi{\P})$ and a subsequence of the random vectors $\{(\wi{\X}_{\g,n}^m,\wi{\W}^n)\}_{n\in\N}$ (still denoted by the same index) and $(\wi{\X}_{\g}^m,\wi{\W})$ on the space $\Lambda$ such that 
	\begin{enumerate}
		\item $\mathscr{L}(\wi{\X}_{\g,n}^m,\wi{\W}^n)=\mathscr{L}(\X_{\g,n}^m,\W^n)$, for all $n\in\N$;
		\item $(\wi{\X}_{\g,n}^m,\wi{\W}^n)\to (\wi{\X}_{\g}^m,\wi{\W})$ in $\Lambda$ with the probability 1 on the space $(\wi{\Omega},\wi{\mathscr{F}},\wi{\P})$ as $n\to\infty$;
		\item $\wi{\W}^n(\wi{\omega})=\wi{\W}(\wi{\omega})$, for all $\wi{\omega}\in\wi{\Omega}$.
	\end{enumerate}Thus, we have 
	\begin{align}\label{4.57}
		\|\wi{\X}_{\g,n}^m-\wi{\X}_{\g}^m\|_{\L^\beta(0,T;\H)}+\|\wi{\X}_{\g,n}^m-\wi{\X}_{\g}^m\|_{\C([0,T];\V^*)}\to0, \ \ \wi{\P}\text{-a.s.}
	\end{align}Now, our focus is to show that for each $m\in\N$, the random vector $(\wi{\X}_{\g}^m,\wi{\W})$ is a solution of \eqref{1.6}. Let us denote the filtration by $\{\wi{\mathscr{F}}_t\}_{t\geq0}$ satisfying the usual conditions and generated by $\{\wi{\X}^{m}_{\g,n}(l),\wi{\X}_{\g}^m(l),\wi{\W}(l): l \leq t\}$.
	
	Then, $\wi{\W}(\cdot)$ is an $\{\wi{\mathscr{F}}_t\}$-cylindrical Wiener process on the Hilbert space $\U$. The equation \eqref{4.007} is satisfied by the random vector $(\wi{\X}_{\g,n}^m,\wi{\W}^n)=(\wi{\X}_{\g,n}^m,\wi{\W})$, and hence it follows that 
	\begin{align}\label{4.58}\nonumber
		\wi{\X}_{\g,n}^m(t)&=\PP_n\boldsymbol{y}_0+\int_0^t \A^n(s,\wi{\X}_{\g,n}^m(s))\d s+\int_0^t \s_1^n(\wi{\X}_{\g,n}^m(s))\Pi_n\d \wi{\W}(s)\\&\quad+\int_0^t \s_2^n(\wi{\X}_{\g,n}^m(s))\dot{\W}^m(s)\d s +\int_0^t \s_3^n(\wi{\X}_{\g,n}^m(s))\g(s)\d s-\int_0^t \G^n(\wi{\X}_{\g,n}^m(s))\d s,
	\end{align}for $ t\in(0,T)$. It also satisfies the energy estimate \eqref{4.8}, that is, for any $p>\max \big\{\frac{\beta}{\beta-1},2\big\}$, we have
	\begin{align}\label{4.59}
		\sup_{n\in\N}\left\{\wi{\E}\bigg[\sup_{t\in[0,T]}\|\wi{\X}_{\g,n}^m(t)\|_\H^p\bigg]+\wi{\E}\bigg[\int_0^T\|\wi{\X}_{\g,n}^m(t)\|_\V^\beta\d t\bigg]^{\frac{p}{2}}\right\}\leq C_m<\infty. 
	\end{align}Using the fact that $\|\cdot\|_\H$ and $\|\cdot\|_\V$ are lower semicontinuous in $\V^*$, the convergence \eqref{4.57} and Fatou's lemma yield 
	\begin{align}\label{4.60}\nonumber
			\wi{\E}\bigg[\sup_{t\in[0,T]}\|\wi{\X}_{\g}^m(t)\|_\H^p\bigg] &\leq 	\wi{\E}\bigg[\sup_{t\in[0,T]}\liminf_{n\to\infty}\|\wi{\X}_{\g,n}^m(t)\|_\H^p\bigg]\leq 	\wi{\E}\bigg[\liminf_{n\to\infty}\sup_{t\in[0,T]}\|\wi{\X}_{\g,n}^m(t)\|_\H^p\bigg]\\&\leq \liminf_{n\to\infty}	\wi{\E}\bigg[\sup_{t\in[0,T]}\|\wi{\X}_{\g,n}^m(t)\|_\H^p\bigg]\leq C_m<\infty.
	\end{align}
	Similarly, from \eqref{4.59}, we can find 
	\begin{align}\label{4.61}
		\wi{\E}\bigg[\int_0^T\|\wi{\X}_{\g}^m(t)\|_\V^\beta\d t\bigg]^{\frac{p}{2}}\leq C_m<\infty.
	\end{align}

	Using \eqref{4.59}, Hypothesis \ref{hyp1} (H.4) and (H.5), and the Banach-Alaoglu theorem for each $m\in\N$, there exists a subsequence indexed by $n_k(m)\to\infty$ (denoting by $n_k$), we obtain the following weak convergences: 
	\begin{equation}\label{4.62}
		\left\{
		\begin{aligned}
			\wi{\X}_{\g,n_k}^m &\xrightharpoonup{\ast} \bar{\X}_{\g}^m, &&\text{ in } \L^p(\Omega;\L^\infty(0,T;\H)),\\
			\wi{\X}_{\g,n_k}^m &\xrightharpoonup{} \bar{\X}_{\g}^m, &&\text{ in } \mathcal{M},
		\end{aligned} 
				\right.
		\end{equation}
	and
		\begin{equation}\label{462}
		\left\{
		\begin{aligned}
			\A^{n_k}(\cdot,\wi{\X}_{\g,n_k}^m) &\xrightharpoonup{} \mathscr{N}_{\g}^m, &&\text{ in }  \mathcal{M}^*,\\
			\s_2^{n_k}(\wi{\X}_{\g,n_k}^m)\dot{\W}^m&\xrightharpoonup{} \mathscr{M}_{\g}^m, &&\text{ in }\mathcal{M}^*,\\
			\s_3^{n_k}(\wi{\X}_{\g,n_k}^m)\g &\xrightharpoonup{} \mathscr{P}_{\g}^m, &&\text{ in }\mathcal{M}^*,\\
			\G^{n_k}(\wi{\X}_{\g,n_k}^m) &\xrightharpoonup{} \mathscr{Q}_{\g}^m, &&\text{ in }\mathcal{M}^*,\\
			\s_1^{n_k}(\wi{\X}_{\g,n_k}^m) &\xrightharpoonup{} \mathscr{R}_{\g}^m, &&\text{ in }\mathcal{H}.
		\end{aligned} 
		\right.
	\end{equation}Therefore, we have 
	\begin{align*}	\int_0^{\cdot}\s_1^{n_k}(\wi{\X}_{\g,n_k}^m)\Pi_{n_k}\d\W \xrightharpoonup{} 	\int_0^{\cdot} \mathscr{R}_{\g}\d \W, \ \ \text{ in }\ \  \L^2(\Omega;\L^\infty(0,T;\H)).
	\end{align*} Let us fix 
	\begin{align}\label{4.63}
		\vi{\X}_{\g}^m(t):= \boldsymbol{y}_0+\int_0^t\left(\mathscr{N}_{\g}^m(s)+\mathscr{M}_{\g}^m(s)+\mathscr{P}_{\g}^m(s)-\mathscr{Q}_{\g}^m(s)\right)\d s+\int_0^{t} \mathscr{R}_{\g}^m(s)\d \W(s).
	\end{align}Using the weak limit defined in \eqref{4.62}, one can verify that (cf. \cite{CPMR,MRSSTZ} or \cite{AKMTM4})
	\begin{align*}
		\vi{\X}_{\g}^m(t,\omega)=\bar{\X}_{\g}^m(t,\omega)=\wi{\X}_{\g}^m(t,\omega), \ \ \text{ for } \ \ \d t\otimes\P\text{-a.e.} \ (t,\omega).
	\end{align*}
	
	Let us now  show our main results in the filtered probability space $(\wi{\Omega},\wi{\mathscr{F}}_t,\{\wi{\mathscr{F}}_t\}_{t\geq0},\wi{\P})$. From now onward we drop the hat notation, for example, we write $\{\wi{\X}_{\g,n}^m\}_{n\in\N}$ and $\wi{\X}_{\g}^m$ as $\{\X_{\g,n}^m\}_{n\in\N}$ and $\X_{\g}^m$, respectively. Thus, we can rewrite \eqref{4.57} as 
	\begin{align}\label{4.64}
		\|\X_{\g,n}^m-\X_{\g}^m\|_{\L^\beta(0,T;\H)}+\|\X_{\g,n}^m-\X_{\g}^m\|_{\C([0,T];\V^*)}\to0, \ \ \wi{\P}\text{-a.s.}
	\end{align}
	
	Using \cite[Theorem 4.2.5]{WLMR2}, we obtain that $\X_{\g}^m(\cdot)$ is an $\H$-valued continuous $\{\mathscr{F}_t\}_{t\geq0}$-adapted process. Thus, it remains to show the following:
	\begin{equation}\label{4.65}
		\left\{
		\begin{aligned}
			\A(\cdot,\X_{\g}^m)+\s_2(\X_{\g}^m)\dot{\W}^m+\s_3(\X_{\g}^m)\g-\G(\X_{\g}^m)&=\mathscr{N}_{\g}^m+\mathscr{M}_{\g}^m+\mathscr{P}_{\g}^m-\mathscr{Q}_{\g}^m,\\
			\s_1(\X_{\g}^m)&=\mathscr{R}_{\g}^m, \ \ \ \ \d t\otimes\P.
		\end{aligned}
		\right.
	\end{equation}	Note that $	\|\X_{\g,n}^m-\X_{\g}\|_{\L^\beta(0,T;\H)}\to0$, as $n\to\infty$. In view of  \eqref{4.58}, \eqref{4.59} and Vitali's convergence theorem, we get 
	\begin{align}\label{4.66}
		\lim_{n\to\infty}\E\bigg[\int_0^T\|\X_{\g,n}^m(t)-\X_{\g}^m(t)\|_\H^\varrho\d t\bigg]=0, \ \ \text{ for all } \ \ \varrho\in[1,\beta],
	\end{align}since for all $1<p<\infty$, we have 
	\begin{align*}
		\E\bigg[\bigg(\int_0^T\|\X_{\g,n}^m(t)\|_\H^\beta\d t\bigg)^p\bigg]\leq T^p\E\bigg[\sup_{t\in[0,T]}\|\X_{\g,n}^m(t)\|_\H^{ p\beta }\bigg]\leq C_m<\infty.
	\end{align*}Therefore, there exists a subsequence $\{\X_{\g,n}^m\}_{n\in\N}$ (still denoting by the same index), we have the following convergence:
	\begin{align}\label{4.67}
		\lim_{n\to\infty}\|\X_{\g,n}^m(t,\omega)-\X_{\g}^m(t,\omega)\|_\H=0, \  \text{ a.e. } \ (t,\omega)\in [0,T]\times \Omega.
	\end{align}
	Now, we verify the limit in the following series of lemmas.
	\begin{lemma}\label{lem4.4}
			$\s_1(\X_{\g}^m(\cdot))=\mathscr{R}_{\g}^m(\cdot)$, $\d t\otimes\P$-a.e.
	\end{lemma}
	\begin{proof}
		Using Hypothesis \ref{hyp1} (H.5) and \eqref{4.67}, we deduce 
		\begin{align}\label{4.68}
			\lim_{n\to\infty}\E\bigg[\int_0^T \|\s_1^n(\X_{\g,n}^m(t))-\s_1(\X_{\g}^m(t))\|_{\L_2}^2\d t\bigg]=0. 
		\end{align}By \eqref{4.63} and the uniqueness of weak limit implies the required result, that is, $\s_1^m(\cdot)=\mathscr{R}_{\g}^m(\cdot)$, $\d t\otimes\P$-a.e. The convergence in \eqref{4.68} can be verified as follows:
		\begin{align}\label{4.69}\nonumber
			&\E\bigg[\int_0^T \|\s_1^n(\X_{\g,n}^m(t))-\s_1(\X_{\g}^m(t))\|_{\L_2}^2\d t\bigg]\\&\nonumber\leq 2\E\bigg[\int_0^T \|\s_1^n(\X_{\g,n}^m(t))-\s_1^n(\X_{\g}^m(t))\|_{\L_2}^2\d t\bigg]+ 2\E\bigg[\int_0^T \|(\I-\PP_n)\s_1(\X_{\g}^m(t))\|_{\L_2}^2\d t\bigg] \\&\to 0, \ \ \text{ as } \ \ n\to\infty, 
		\end{align}where we have used Hypothesis \ref{hyp1} (H.2), \eqref{4.67} in the first term of the right hand side of the above inequality and Lebesgue dominated convergence theorem for the final term in the right hand side of \eqref{4.69}, since we know that $\|\I-\PP_n\|_{\mathcal{L}(\H)}\to0$, as $n\to\infty$.
	\end{proof}
	\begin{lemma}\label{lem4.5}
		$\s_2(\X_{\g}^m(\cdot))\dot{\W}^m(\cdot)=\mathscr{M}_{\g}^m, \ \d t\otimes\P$-a.e.
	\end{lemma}
	\begin{proof}Using Young's inequality, Hypothesis \ref{hyp1} (H.2) and \eqref{4.17}, we obtain 
		\begin{align}\label{4.70}\nonumber
			&\E\bigg[\int_0^T\|\s_2^n(\X_{\g,n}^m(t))\dot{\W}^m(t)-\s_2(\X_{\g}^m(t))\dot{\W}^m(t)\|_\H^2\d t\bigg]\\&\nonumber=\E\bigg[\int_0^T\|(\s_2^n(\X_{\g,n}^m(t))-\s_2(\X_{\g}^m(t)))\dot{\W}^m(t)\|_\H^2\d t\bigg]\\&\nonumber\leq 2\E\bigg[\int_0^T\|(\s_2^n(\X_{\g,n}^m(t))-\s_2^n(\X_{\g}^m(t)))\dot{\W}^m(t)\|_\H^2\d t\bigg]\\&\nonumber\quad +2\E\bigg[\int_0^T\|(\I-\PP_n)\s_2(\X_{\g}^m(t))\dot{\W}^m(t)\|_\H^2\d t\bigg]\\&\nonumber\leq  2\E\bigg[\int_0^T \|\s_2^n(\X_{\g,n}^m(t))-\s_2^n(\X_{\g}^m(t))\|_{\L_2}^2\|\dot{\W}^m(t)\|_\U^2\d t\bigg]\\&\nonumber\quad +2\E\bigg[\int_0^T\|(\I-\PP_n)\s_2(\X_{\g}^m(t))\|_{\L_2}^2\|\dot{\W}^m(t)\|_\U^2\d t\bigg]\\&\to0, \ \ \text{ as } \ \ n\to\infty,
		\end{align}
	Arguments similar to \eqref{4.69} and the uniqueness of weak limit completes the proof.
	\end{proof}
	\begin{lemma}\label{lem4.6}
		$\s_3(\X_{g}^m(\cdot))\g(\cdot)=\mathscr{P}_{\g}^m(\cdot)$ and $\G(\X_{\g}^m(\cdot))=\mathscr{Q}_{\g}^m(\cdot), \ \d t\otimes\P$-a.e.
	\end{lemma}
	\begin{proof}
		The proof of this lemma will follows on the similar lines as in Lemmas \ref{lem4.4} and \ref{lem4.5}, so we are omitting here.
	\end{proof}Now, let us recall a convergence result from \cite{MRSSTZ}.
	\begin{lemma}[{\cite[Lemma 2.16]{MRSSTZ}}]\label{lem4.7}
		For each $m\in\N,$ we have \begin{align}\label{4.71}\nonumber
			\X_{\g,n}^m &\xrightharpoonup{\ast} \X_{\g}^m, \ \ \text{in}\ \ \mathcal{M},\\ \nonumber
			\A^n(\cdot,\X_{\g,n}^m(\cdot)) &\xrightharpoonup{} \mathscr{N}_{\g}^m(\cdot), \ \ \text{in} \ \ \mathcal{M}^*,\\
			\liminf_{n\to\infty} \E\bigg[\int_0^T\langle \A(t,\X_{\g,n}^m(t)),\X_{\g,n}^m(t)\rangle \d t\bigg] &\geq \E\bigg[\int_0^T\langle \mathscr{N}_{\g}^m(t),\X_{\g}^m(t)\rangle \d t\bigg],
		\end{align}then $\mathscr{N}_{\g}^m(\cdot)=\A(\cdot,\X_{\g}^m(\cdot)),\ \d t\otimes\P$-a.e.
	\end{lemma}
	Let us establish the existence of a probabilistically weak solution to the problem \eqref{4.4}.
	\begin{theorem}\label{thrm2.8}
		There exists a probabilistically weak solution to the problem \eqref{4.4} which satisfies the uniform energy estimate \eqref{4.8}.	
	\end{theorem}
	\begin{proof}
		Our aim is to prove that the limit $\X_{\g}^m(\cdot)$ of the approximated solutions $\{\X_{\g,n}^m\}_{n\in\N}$ obtained above is a probabilistically weak solution to the problem \eqref{4.4}. In order to establish this, we need to verify \eqref{4.71}, using Lemmas \ref{lem4.4}-\ref{lem4.6}. We already have that equations \eqref{4.58} and \eqref{4.63} are satisfied by $ \X_{\g,n}^m(\cdot)$ and $\X_{\g}^m(\cdot)$, respectively.
		
		Applying the finite and infinite-dimensional It\^o's formulae (see \cite[Theorem 2.1]{IGDS} and \cite[Theorem 1]{IGNV}) to the processes $\ \X_{\g,n}^m(\cdot)$ and $\X_{\g}^m(\cdot)$, respectively, and then taking expectations on both sides, we find
		\begin{align}\label{4.72}\nonumber
			\E\big[\|\X_{\g,n}^m(t)\|_\H^2\big]&=\E\big[\|\boldsymbol{y}_0^n\|_\H^2\big]+\E\bigg[\int_0^t \bigg\{\big\langle \A(s,\X_{\g,n}^m(s)),\X_{\g,n}^m(s)\big\rangle +\|\s_1(\X_{\g,n}^m(s))\Pi_n\|_{\L_2}^2\\&\nonumber\qquad+ \big(\s_2^n(\X_{\g,n}^m(s))\dot{\W}^m(s),\X_{\g,n}^m(s)\big)+\big(\s_3^n(\X_{\g,n}^m(s))\g(s),\X_{\g,n}^m(s)\big)\\&\qquad -\big(\G^n(\X_{\g,n}^m(s)),\X_{\g,n}^m(s)\big) \bigg\}\d s\bigg], \\ \nonumber
			\E\big[\|\X_{\g}^m(t)\|_\H^2\big]&=\E\big[\|\boldsymbol{y}_0\|_\H^2\big]+\E\bigg[\int_0^t \bigg\{\big\langle \mathscr{N}_{\g}(s),\X_{\g}^m(s)\big\rangle +\|\mathscr{R}_{\g}^m(s)\|_{\L_2}^2+ \big(\mathscr{M}_{\g}^m(s),\X_{\g}^m(s)\big)\\&\label{4.73}\qquad+\big(\mathscr{P}_{\g}^m(s),\X_{\g}^m(s)\big) -\big(\mathscr{Q}_{\g}^m(s),\X_{\g}^m(s)\big) \bigg\}\d s\bigg].
		\end{align}
	Using the convergence \eqref{4.64}, the lower semicontinuity of $\|\cdot\|_\H$ in $\V^*$ and Fatou's lemma, we arrive at 
	\begin{align}\label{4.74}
				\E \big[\|\X_{\g}^m(t)\|_\H^2\big] \leq 	\E\big[\liminf_{n\to\infty}\|\X_{\g,n}^m(t)\|_\H^2\big]\leq 	\liminf_{n\to\infty}\E\big[\|\X_{\g,n}^m(t)\|_\H^2\big].
		\end{align}By Lemmas \ref{lem4.4}-\ref{lem4.7} and comparing \eqref{4.72} and \eqref{4.73}, we obtain \eqref{4.71}. Moreover, the uniform energy estimates obtained in \eqref{4.8} remains valid for $\X_{\g}^m(\cdot)$, for each $m\in\N$.
	\end{proof}
	
	\begin{theorem}\label{thrm2.9}
		Under the assumptions of Theorem \ref{thrm4.1}, the pathwise uniqueness holds for the solutions to the problem \eqref{4.4}.
	\end{theorem}
	\begin{proof}
		Let $\X_{\g}^m(\cdot)$ and $\Y_{\g}^m(\cdot)$ be the two solutions to the problem \eqref{4.4} defined on the probability space $(\Omega,\mathscr{F},\{\mathscr{F}_t\}_{t\geq0},\P)$, with the initial data $\X_{\g}^m(0)=\boldsymbol{y}_0^1$ and $\Y_{\g}^m(0)=\boldsymbol{y}_0^2$, respectively. Let us define 
		\begin{align*}
			\Psi(t)&:=\exp\bigg\{-\int_0^t\big(f(s)+\rho(\X_{\g}^m(s))+\eta(\Y_{\g}^m(s))+\kappa(\X_{\g}^m(s))+\varkappa(\Y_{\g}^m(s))\\&\qquad+\|\dot{\W}^m(s)\|_\U^2+\|\g(s)\|_\U^2\big) \d s\bigg\}.
		\end{align*}Applying It\^o's formula to the process $\Psi(\cdot)\|\X_{\g}^m(\cdot)-\Y_{\g}^m(\cdot)\|_\H^2$, we find
		\begin{align}\label{4.75}\nonumber
			&\Psi(t)\|\X_{\g}^m(t)-\Y_{\g}^m(t)\|_\H^2\\&\nonumber=\|\boldsymbol{y}_0^1-\boldsymbol{y}_0^2\|_\H^2+ \int_0^t\Psi(s)\bigg\{2\langle \A(s,\X_{\g}^m(s))-\A(s,\Y_{\g}^m(s)),\X_{\g}^m(s)-\Y_{\g}^m(s)\rangle\\&\nonumber\qquad+\|\s_1(\X_{\g}^m(s))-\s_1(\Y_{\g}^m(s))\|_{\L_2}^2+\big((\s_2(\X_{\g}^m(s))-\s_2(\Y_{\g}^m(s)))\dot{\W}^m(s),\X_{\g}^m(s)-\Y_{\g}^m(s)\big)\\&\nonumber\qquad+ \big((\s_3(\X_{\g}^m(s))-\s_2(\Y_{\g}^m(s)))\g(s),\X_{\g}^m(s)-\Y_{\g}^m(s)\big) \\&\nonumber\qquad-\big(\G(\X_{\g}^m(s))-\G(\Y_{\g}^m(s)),\X_{\g}^m(s)-\Y_{\g}^m(s)\big)\\&\nonumber\qquad- \big(f(s)+\rho(\X_{\g}^m(s))+\eta(\Y_{\g}^m(s))+\kappa(\X_{\g}^m(s))+\varkappa(\Y_{\g}^m(s))+\|\dot{\W}^m(s)\|_\U^2+\|\g(s)\|_\U^2\big)\\&\nonumber\quad\qquad\times\|\X_{\g}^m(s)-\Y_{\g}^m(s)\|_\H^2\bigg\}\d s\\&\nonumber\quad + 2\int_0^t\Psi(s)\big((\s_1(\X_{\g}^m(s))-\s_1(\Y_{\g}^m(s)))\d \W(s),\X_{\g}^m(s)-\Y_{\g}^m(s)\big)\\&\leq \|\boldsymbol{y}_0^1-\boldsymbol{y}_0^2\|_\H^2+2\int_0^t\Psi(s)\big((\s_1(\X_{\g}^m(s))-\s_1(\Y_{\g}^m(s)))\d \W(s),\X_{\g}^m(s)-\Y_{\g}^m(s)\big),
		\end{align}where we have used Hypotheses \ref{hyp1} (H.2), (H.5) and \ref{hyp2} (H.7). Let $\{\tau_k\}\uparrow\infty$ be a sequence of stopping times in such a way that the local martingale term appearing in the above inequality is a martingale. Taking expectations on both side if the above inequality, we find 
		\begin{align}\label{4.76}
			\E\big[\Psi(t\wedge\tau_k)\|\X_{\g}^m(t\wedge\tau_k)-\Y_{\g}^m(t\wedge\tau_k)\|_\H^2\big]\leq \E\big[\|\boldsymbol{y}_0^1-\boldsymbol{y}_0^2\|_\H^2\big].
		\end{align}Letting $k\to\infty$ and using Fatou's lemma, we obtain
		\begin{align}\label{4.77}
			\E\big[\Psi(t)\|\X_{\g}^m(t)-\Y_{\g}^m(t)\|_\H^2\big]\leq \E\big[\|\boldsymbol{y}_0^1-\boldsymbol{y}_0^2\|_\H^2\big],
		\end{align}where we have used the fact that 
	\begin{align*}
			\int_0^T\big(f(s)+\rho(\X_{\g}^m(s))+\eta(\Y_{\g}^m(s))&+\kappa(\X_{\g}^m(s))+\varkappa(\Y_{\g}^m(s))\\&+\|\dot{\W}^m(s)\|_\U^2+\|\g(s)\|_\U^2\big) \d s<\infty, \ \P\text{-a.s.}
		\end{align*}The inequality \eqref{4.77} gives the pathwise uniqueness of the solutions to the problem \eqref{4.4}.
	\end{proof}
	\begin{proof}[Proof of Theorem \ref{thrm4.1}]
		In Theorem \ref{thrm2.8}, we proved the existence of a probabilistically weak solution to the problem \eqref{4.4} which satisfies the uniform energy estimate \eqref{4.8}. Moreover, under Hypotheses \ref{hyp1} and \ref{hyp2}, the pathwise uniqueness of probabilistically weak solution to the problem \eqref{4.4} is established in Theorem \ref{thrm2.9}.  Therefore, in view of Theorems \ref{thrm2.8}-\ref{thrm2.9}, an application of the classical Yamada-Watanabe theorem (\cite[Theorem 2.1]{MRBSXZ}) yields  the existence of a unique probabilistically strong solution to the problem \eqref{4.4}.
	\end{proof}
	\begin{remark}
		In this remark, we discuss an alternative way to establish the existence and uniqueness results for the system \eqref{4.4}.  The main ingredient of this proof is to use the Girsanov theorem (see \cite[Theorem 10.14]{DaZ}). A similar result using the Girsanov theorem has been established in \cite[Subsection 2.5]{ICAM2} for stochastic two-dimensional hydrodynamical systems.  For any constant $\gamma,$ and $\s_2=\gamma\s_1$,	let us recall the system \eqref{4.4} with the modification: 
		\begin{equation}\label{6.1}
			\left\{
			\begin{aligned}
				\d \X_{\g}^m(t)&=\A(t,\X_{\g}^m(t))\d t+\s_1(\X_{\g}^m(t))\d \W(t)+\gamma\s_1(\X_{\g}^m(t))\dot{\W}^m(t)\d t\\&\quad +\s_3(\X_{\g}^m(t))\g(t)\d t-\G(\X_{\g}^m(t))\d t, \ \ t\in(0,T),\\ 
				\X_{\g}^m(0)&=\boldsymbol{y}_0,
			\end{aligned}
			\right.
		\end{equation}where the mappings are defined in Section \ref{Sec2}.
		For any $m\in\N$, and $t\in[0,T]$, let us set 
		\begin{align*}
			\mathrm{Y}^m(t):=\exp\bigg\{\gamma\int_0^t\dot{\W}^m(s)\d\W(s)-\frac{\gamma^2}{2}\int_0^t\|\dot{\W}^m(s)\|_\U^2\d s\bigg\}
		\end{align*}and 
		\begin{align}\label{6.2}
			\wi{\W}(t):=\W(t)+\gamma\int_0^t\dot{\W}^m(s)\d s,
		\end{align}where $\dot{\W}^m(\cdot)$ is defined in \eqref{1.3}. Using Girsanov's theorem, the process $\wi{\W}(\cdot)$ defined in \eqref{6.2} is a cylindrical Wiener process under the probability measure $\wi{\P}$ and the measure is absolutely continuous with respect to the measure $\P$, that is $$\frac{\d {\wi{\P}}}{\d \P}\bigg|_{\mathscr{F}_t}=\mathrm{Y}^m(t),$$
		for $t\in[0,T]$ (for more details see Section \ref{Sec4} below). Under the above change of measure, the system \eqref{6.1} reduces to the following system:
		\begin{equation}\label{6.3}
			\left\{
			\begin{aligned}
				\d \X_{\g}^m(t)&=\A(t,\X_{\g}^m(t))\d t+\s_1(\X_{\g}^m(t))\d \wi{\W}(t)+\s_3(\X_{\g}^m(t))\g(t)\d t\\&\quad -\G(\X_{\g}^m(t))\d t, \ \ t\in(0,T),\\ 
				\X_{\g}^m(0)&=\boldsymbol{y}_0.
			\end{aligned}
			\right.
		\end{equation}
		By taking $\mathcal{A}(\cdot,\X)=\A(\cdot,\X)+\sigma_3(\X)\g-\G(\X)$, one can see that the system \eqref{6.3} is similar to the one given in \eqref{1.1},  whose existence of unique pathwise strong solution is known from Theorem \ref{thrm1}. Therefore, for any $p\geq 2$ and $\beta\in(1,\infty)$, the following energy estimate holds:
	\begin{align}\label{383}
		\wi{\E}\bigg[\sup_{t\in[0,T]}\|\X_{\g}^m(t)\|_\H^p\bigg]+\wi{\E}\bigg[\int_0^T\|\X_{\g}^m(t)\|_\V^\beta\d t\bigg]^{\frac{p}{2}}\leq C<\infty,
	\end{align}
where the constant $C$ is independent of $m$. 	Therefore, $\X_{\g}^m(\cdot)$ is the unique solution of the system \eqref{6.1} and  $\X_{\g}^m(\cdot)$  satisfies the energy estimate \eqref{383} with the expected values under the given probability measure $\mathbb{P}$, but the constant $C$ in the right hand depends on $m$ (see \eqref{4.8}).
			\end{remark}

\section{Wong-Zakai approximation for the system \eqref{1.1}}\label{Sec3}\setcounter{equation}{0}
In this section, we state and prove our main result of this article. Before proceeding further, let us recall the following results from \cite{ICAM2}.
\begin{lemma}[{\cite[Lemma 2.1]{ICAM2}}]\label{lem3.1}
	For $T>0$, there exists a constant $\delta_0>0$ such that for every $\delta>\frac{\delta_0}{\sqrt{T}},\ t\in[0,T]$,
	\begin{align}\label{3.1}
		\lim_{m\to\infty} \P\bigg\{\sup_{1\leq i\leq m}\sup_{s\in[0,t]}|\dot{\beta}_i(s)|>\delta m^{\frac{1}{2}}2^{\frac{m}{2}}\bigg\}&=0,\\\label{3.2}
			\lim_{m\to\infty} \P\bigg\{\sup_{s\in[0,t]}\|\dot{\W}^m(s)\|_\U>\delta m2^{\frac{m}{2}}\bigg\}&=0.
	\end{align}
\end{lemma}
	\begin{lemma}[{\cite[Lemma 4.1]{ICAM2}}]\label{lemAM}
		Let $\vartheta(t)=\vartheta(\omega,t)$ be a random, a.s. continuous, non-decreasing process on $[0,T]$. Let $\tau_{\lambda}=T\wedge \inf\{t\geq 0: \vartheta(t)\geq \lambda\}$. Then, $\P\{\vartheta(T)\geq \lambda\}=\P\{\vartheta(\tau_\lambda) \geq \lambda\}$. Let $\tau_*$ be a stopping time such that $\tau_*\in[0,T]$ and $\P\{\tau_*<T\}\leq \e$. Then 
		\begin{align*}
			\P\{\vartheta(T)\geq \lambda\}\leq \P\{\vartheta(\tau_\lambda\wedge \tau_*)\geq \lambda\}+\e.
		\end{align*}
\end{lemma}
\begin{remark}\label{Remark4.2}
	If we compare \eqref{1.1} and \eqref{1.6}, the most significant difference is  that the term \break$	\displaystyle\int_0^{\cdot}\s(\Y^m(s))\dot{\W}^m(s)\d s$ cannot be expressed as stochastic integral directly. Indeed, we have the following identity (see \cite[Remark 2.7]{TMRZ} for more details):
	\begin{align}\label{3.4}\nonumber
		&\int_0^t\s\bigg(\Y^m\big((\lfloor\frac{s}{\vp}\rfloor-1)\vp\big)\bigg)\dot{\W}^m(s)\d s\\&=\int_0^t\bigg(\frac{1}{\vp}\int_{\lceil\frac{s}{\vp}\rceil\vp}^{\big(\lceil\frac{s}{\vp}\rceil+1\big)\vp}\chi_{\{l\leq t\}}\d l\bigg)\s\bigg(\Y^m\big(\lfloor \frac{s}{\vp}\rfloor\vp\big)\bigg)\circ\Pi_m\d\W(s),
	\end{align}to compare with the corresponding diffusion term $\displaystyle\int_0^t\s(\Y(s))\d\W(s)$.
\end{remark}

Let us now prove our main result, that is, the Wong-Zakai approximation result.  We follow the arguments similar to \cite[Thorem 3.1 and Section 4]{ICAM2} and \cite[Theorem 2.6]{TMRZ} to obtain our main result.

 \begin{theorem}\label{thrm3.3}
	Assume that Hypotheses \ref{hyp1} and \ref{hyp2} hold. Let the  initial data $\boldsymbol{y}_0$  be in the space $\L^p(\Omega;\H)$ for $p>\max\big\{\frac{\beta}{\beta-1},2\big\}$. Let $\Y(\cdot)$ and $\Y^m(\cdot)$ be the solutions to the system \eqref{1.1} and \eqref{1.6}, with the same initial data $\boldsymbol{y}_0$, respectively. Then
	for any $\lambda\in(0,1]$, we have 
	\begin{align}\label{45}
		\lim_{m\to\infty} \P\bigg\{\sup_{t\in[0,T]}\|\Y(t)-\Y^m(t)\|_\H^2\geq \lambda\bigg\}=0.
	\end{align}
\end{theorem}

\begin{proof}
		In order to verify the convergence \eqref{45},  by an application of Markov's inequality,  it is sufficient to show  that 	\begin{align}\label{3.5}
			\lim_{m\to\infty}\E\bigg[\sup_{t\in[0,T]}\|\Y(t)-\Y^m(t)\|_\H^2\bigg]=0.
		\end{align}
For $M\geq 0,\ m\in\N, \ \delta>\frac{\delta_0}{\sqrt{T}}$, and $\lambda\in(0,1]$, we define the following stopping times:
	\begin{align*}
		\tau_M^1&=T\wedge\inf\bigg\{t\geq0: \|\Y(t)\|_\H^2+\int_0^t\|\Y(s)\|_\V^\beta\d s>M\bigg\},\\
		\tau_{m}^2&=T\wedge\inf\bigg\{t\geq0: \|\Y(t)-\Y^m(t)\|_\H^2+\int_0^t\|\Y(s)-\Y^m(s)\|_\V^\beta\d s\geq \lambda\bigg\} 
		,\\
		\tau_{m}^3&=T\wedge\inf\bigg\{t\geq0:\bigg[\sup_{1\leq i\leq m}\sup_{s\in[0,t]}|\dot{\beta}_i(s)|\bigg]\vee\bigg[m^{-\frac{1}{2}}\sup_{s\in[0,t]}\|\dot{\W}^m(s)\|_\U\bigg]>\delta m^{\frac{1}{2}}2^{\frac{m}{2}}\bigg\}.
	\end{align*} In the sequel, the constant $M$ will be chosen to make sure that, except on small sets, $\tau_M^1$ is equal to $T$. Once this is achieved, only the dependency of $m$ will be relevant. Thus, once $M$ has been chosen in terms of the limit process $\Y$, we let 
\begin{align}\label{3.3}
	\tau_m=\tau_M^1\wedge \tau_{m}^2\wedge \tau_m^3.
\end{align}From the definition of $\tau_M^1$ and $\tau_m^2$, we have
\begin{align}
	\sup_{s\in[0,\tau_m]} \big\{\|\Y(s)\|_\H^2\vee \|\Y^m(s)\|_\H^2\big\}+\int_0^{\tau_m}\big\{\|\Y(s)\|_\V^\beta\vee \|\Y^m(s)\|_\V^\beta\big\}\d s \leq 2(M+1),
\end{align}and from $\tau_m^3$, we obtain 
\begin{align}
	 \bigg[\sup_{s\in[0,\tau_m]}\sup_{1\leq i\leq m}|\dot{\beta}_i(s)|\bigg] \vee \bigg[m^{-\frac{1}{2}}\sup_{s\in[0,\tau_m]}\|\dot{\W}^m(s)\|_\U\bigg] \leq \delta m^\frac{1}{2}2^{\frac{m}{2}}.
\end{align}
Let us now apply Lemma \ref{lemAM} with $\tau_*=\tau_M^1\wedge \tau_m^3$ and 
\begin{align*}
	\vartheta(t)=\sup_{s\in[0,t]}\|\Y(s)-\Y^m(s)\|_\H^2.
\end{align*}
Since $\Y\in\C([0,T];\H)$, $\P$-a.s. and $\int_0^T\|\Y(s)\|_\V^\beta\d s<\infty$, $\P$-a.s., the map $\vartheta$ is a.s. continuous and 
\begin{align*}
	\{\tau_*<T\}\subset \{\tau_M^1<T\}\cup \{\tau_m^3<T\}\subset \bigg\{\sup_{s\in[0,\tau_M^1]}\|\Y(s)\|_\H^2+\int_0^{\tau_M^1}\|\Y(s)\|_\V^\beta\d s> M\bigg\}\cup \Omega_m^c(T), 
\end{align*}where 
\begin{align*}
	\Omega_m(t)=\bigg\{\sup_{1\leq i\leq m}\sup_{s\in[0,t]}|\dot{\beta}_i(s)|\leq \delta m^\frac{1}{2}2^{\frac{m}{2}}\bigg\}\cap \bigg\{\sup_{s\in[0,t]}\|\dot{\W}^m(s)\|_\U\leq \delta m2^{\frac{m}{2}}\bigg\}.
\end{align*}
By the energy estimate \eqref{1.16} and an application of Markov's inequality yield 
\begin{align*}
	\P\{\tau_*<T\}\leq CM^{-1}+\P\{\Omega_m^c(T)\}.  
\end{align*}Therefore, for any given $\e>0$, we can choose $M$ large enough such that $CM^{-1}<\frac{\e}{2}$. Using Lemma \ref{lem3.1}, we conclude that there exists $m_0\geq 1$ such that for all integers $m\geq m_0$, $\P\{\Omega_m^c(T)\}<\frac{\e}{2}$. Lemma \ref{lemAM} shows that in order to prove \eqref{45}, we only need to show the following: Fix $M$, for every $\lambda>0$, 
\begin{align}\label{PC2}
	\lim_{m\to\infty}\P\bigg\{\sup_{t\in [0,\tau_m]}\|\Y(t)-\Y^m(t)\|_\H^2\geq \lambda\bigg\}=0.
\end{align}
To check the convergence \eqref{PC2},  by an application of Markov's inequality, it is sufficient to show  the following:
\begin{align}\label{3.6}
	\lim_{m\to\infty}\E\bigg[\sup_{t\in[0,\tau_{m}]}\|\Y(t)-\Y^m(t)\|_\H^2\bigg]=0,
\end{align}and the Wong-Zakai approximation result can be completed.   In fact, by similar arguments,  we only need to prove \eqref{3.6} to obtain \eqref{3.5}. The proof is divided into the following steps:

\vspace{2mm}
\noindent
\textbf{Step 1.} According to the definition of stopping times $\tau_m$, for some fixed constant $\delta>0$ with $\delta>\frac{\delta_0}{\sqrt{T}}$, we can find a positive constant $C_M$ such that for all $t\in[0,\tau_m]$ and $i=1,\ldots,m$,
\begin{equation}\label{3.8}
	\left\{
	\begin{aligned}
		\|\Y(t)\|_\H^2+\|\Y^m(t)\|_\H^2 &\leq C_M,\\
		\int_0^t\big\{\|\Y(s)\|_\V^\beta+\|\Y^m(s)\|_\V^\beta\big\}\d s&\leq C_M,\\
		|\dot{\beta}_i^m(t)|+m^{-\frac{1}{2}}\|\dot{\W}^m(t)\|_\U&\leq 2\delta m^{\frac{1}{2}}2^\frac{m}{2}.
	\end{aligned}
	\right.
\end{equation}The  above properties will be used repeatedly throughout the sequel. By the equality \eqref{3.4}, we have the following decomposition
\begin{align*}
	\Y^m(t)-\Y(t) &=\int_0^t\big( \A(s,\Y^m(s))-\A(s,\Y(s))\big)\d s\\&\quad+ \int_0^t\bigg\{\bigg(\frac{1}{\vp}\int_{\lceil \frac{s}{\vp}\rceil\vp }^{(\lceil \frac{s}{\vp}\rceil+1)\vp }\chi_{\{l\leq t\}}\d l\bigg)\s(\Y^m(\lfloor{\frac{s}{\vp}}\rfloor\vp))\Pi_m-\s(\Y(s))\bigg\}\d \W(s)\\&\quad +\int_0^t \bigg\{\big[\s(\Y^m(s))-\s(\Y^m((\lfloor\frac{s}{\vp}\rfloor-1)\vp))\big]\dot{\W}^m(s)-\frac{1}{2}\wi{\Tr}_m(\Y^m(s)) \bigg\}\d s.
\end{align*}Applying It\^o's formula (cf. \cite[Theorem 1]{IGNV}) to the process $\|\Y^m(\cdot)-\Y(\cdot)\|_\H^2$, we find 
\begin{align}\label{3.9}\nonumber
&	\|	\Y^m(t)-\Y(t) \|_\H^2\\&\nonumber=\int_0^t2\langle  \A(s,\Y^m(s))-\A(s,\Y(s)),\Y^m(s)-\Y(s)\rangle  \d s\\&\nonumber\quad+ \int_0^t\bigg\|\bigg(\frac{1}{\vp}\int_{\lceil \frac{s}{\vp}\rceil\vp }^{(\lceil \frac{s}{\vp}\rceil+1)\vp }\chi_{\{l\leq t\}}\d l\bigg)\s(\Y^m(\lfloor{\frac{s}{\vp}}\rfloor\vp))\Pi_m-\s(\Y(s))\bigg\|_{\L_2}^2\d s\\&\nonumber\quad +2
	\int_0^t\bigg(\bigg\{\bigg(\frac{1}{\vp}\int_{\lceil \frac{s}{\vp}\rceil\vp }^{(\lceil \frac{s}{\vp}\rceil+1)\vp }\chi_{\{l\leq t\}}\d l\bigg)\s(\Y^m(\lfloor{\frac{s}{\vp}}\rfloor\vp))\Pi_m-\s(\Y(s))\bigg\}\d \W(s),\Y^m(s)-\Y(s)\bigg)
	\\&\nonumber\quad +2\int_0^t \bigg(\bigg[\s(\Y^m(s))-\s(\Y^m((\lfloor\frac{s}{\vp}\rfloor-1)\vp))\bigg]\dot{\W}^m(s)-\frac{1}{2}\wi{\Tr}_m(\Y^m(s)),\Y^m(s)-\Y(s) \bigg)\d s\\&=:\sum_{j=1}^4I_j(m,t).
\end{align}
\vspace{2mm}
\noindent
\textbf{Step 2.} \textsf{Claim:} There exists a positive constant $C_{M,T,\|f\|_{\L^1}}$ such that 
\begin{align}\label{3.10}
	\E\bigg[\int_0^{\tau_m}\|\Y(l)-\Y(\lfloor\frac{l}{\vp}\rfloor\vp)\|_\H^2\d l\bigg]&\leq C_{M,T,\|f\|_{\L^1}}2^{-\frac{3m}{4}},\\ 
	\label{3.11}
	\E\bigg[\int_0^{\tau_m}\|\Y^m(l)-\Y^m(\lfloor\frac{l}{\vp}\rfloor\vp)\|_\H^2\d l\bigg]&\leq C_{M,T,\|f\|_{\L^1}}2^{-\frac{3m}{4}},\\
		\label{3.12}
	\E\bigg[\int_0^{\tau_m}\|\Y(l)-\Y((\lfloor\frac{l}{\vp}\rfloor-1)\vp)\|_\H^2\d l\bigg]&\leq C_{M,T,\|f\|_{\L^1}}2^{-\frac{3m}{4}},\\
		\label{3.13}
	\E\bigg[\int_0^{\tau_m}\|\Y^m(l)-\Y^m((\lfloor\frac{l}{\vp}\rfloor-1)\vp)\|_\H^2\d l\bigg]&\leq C_{M,T,\|f\|_{\L^1}}2^{-\frac{3m}{4}},\\
	\label{3.14}
	\E\bigg[\int_0^{\tau_m}\|\Y(l)-\Y(\lceil\frac{l}{\vp}\rceil\vp)\|_\H^2\d l\bigg]&\leq C_{M,T,\|f\|_{\L^1}}2^{-\frac{3m}{4}},\\ 
	\label{3.15}
	\E\bigg[\int_0^{\tau_m}\|\Y^m(l)-\Y^m(\lceil\frac{l}{\vp}\rceil\vp)\|_\H^2\d l\bigg]&\leq C_{M,T,\|f\|_{\L^1}}2^{-\frac{3m}{4}}.
\end{align}Note carefully that the inequalities \eqref{3.10}-\eqref{3.11} and \eqref{3.14}-\eqref{3.15} are not the same.
The above inequalities will be used to estimate the terms in \eqref{3.9}. 

\noindent
\textsf{Proof of \eqref{3.10}:} Applying It\^o's formula (cf. \cite[Theorem 1]{IGNV}) to the process $\|\Y(\cdot)-\Y(\lfloor\frac{s}{\vp}\rfloor\vp)\|_\H^2$, integrating with respect to $s$ and taking expectation, we get for $t\in(0,\tau_m]$,
\begin{align}\label{3.16}\nonumber
	&\E\bigg[\int_0^{t}\|\Y(s)-\Y(\lfloor\frac{s}{\vp}\rfloor\vp)\|_\H^2\d s\bigg]\\&\nonumber=2\E\bigg[\int_0^t\int_{\lfloor\frac{s}{\vp}\rfloor\vp}^{s}\langle \A(l,\Y(l)),\Y(l)-\Y(\lfloor\frac{s}{\vp}\rfloor\vp)\rangle \d l\d s\bigg]\\&\nonumber\quad +\E\bigg[\int_0^t\int_{\lfloor\frac{s}{\vp}\rfloor\vp}^{s}\|\s(\Y(l))\|_{\L_2}^2\d l\d s\bigg]+2\E\bigg[\int_0^t\int_{\lfloor\frac{s}{\vp}\rfloor\vp}^{s}\big(\s(\Y(l))\d\W(l),\Y(l)-\Y(\lfloor \frac{s}{\vp}\rfloor\vp)\big)\d s\bigg]\\& =: J_1(m)+J_2(m)+J_3(m).
\end{align}

\vspace{2mm}
\noindent
\textsl{Estimate for $J_1(m)$:} Using H\"older's inequality, Hypothesis \ref{hyp1} (H.4), Fubini's theorem and \eqref{3.8}, we get
\begin{align}\label{3.17}\nonumber
	&J_1(m)\\&\nonumber\leq 2\bigg\{\E\bigg[\int_0^{\tau_m}\int_{\lfloor\frac{s}{\vp}\rfloor\vp}^{s}\|\Y(l)-\Y(\lfloor\frac{s}{\vp}\rfloor\vp)\|_\V^\beta\d l\d s\bigg]\bigg\}^{\frac{1}{\beta}}	\\&\nonumber\qquad\times\bigg\{\E\bigg[\int_0^{\tau_m}\int_{\lfloor\frac{s}{\vp}\rfloor\vp}^{s}\|\A(l,\Y(l))\|_{\V^*}^\frac{\beta}{1-\beta}\d l\d s\bigg]\bigg\}^\frac{\beta-1}{\beta}
	\\&\nonumber\leq 2\bigg\{\E\bigg[\int_0^{\tau_m}\int_{\lfloor\frac{s}{\vp}\rfloor\vp}^{s}\|\Y(l)-\Y(\lfloor\frac{s}{\vp}\rfloor\vp)\|_\V^\beta\d l\d s\bigg]\bigg\}^{\frac{1}{\beta}}	\\&\nonumber\qquad\times\bigg\{\E\bigg[\int_0^{\tau_m}\int_{\lfloor\frac{s}{\vp}\rfloor\vp}^{s}\big(f(l)+C\|\Y(l)\|_\V^\beta\big)\d l\d s\bigg] \bigg\}^\frac{\beta-1}{\beta}
		\\&\nonumber\leq 2\bigg\{\vp\E\bigg[\int_0^{T}\|\Y(s)-\Y(\lfloor\frac{s}{\vp}\rfloor\vp)\|_\V^\beta\d s\bigg]\bigg\}^{\frac{1}{\beta}}\bigg\{\vp\E\bigg[\int_0^{T}\big(f(s)+C\|\Y(s)\|_\V^\beta\big)\d s\bigg]\bigg\}^\frac{\beta-1}{\beta} \\&\leq C_{M,T,\|f\|_{\L^1}}2^{-m}.
\end{align}In order to obtain the penultimate term in the above inequality, we have used the following calculations: 
\begin{align*}
&	\int_0^{\tau_m}\int_{\lfloor \frac{s}{\vp}\rfloor\vp}^s\|\Y(l)-\Y(\lfloor \frac{s}{\vp}\rfloor\vp)\|_\V^\beta\d l \d s \\& =\sum_{k=0}^{\lfloor \frac{\tau_m}{\vp}\rfloor}\int_{k\vp}^{(k+1)\vp\wedge \tau_m}\int_{\lfloor \frac{s}{\vp}\rfloor\vp}^s \|\Y(l)-\Y(\lfloor \frac{s}{\vp}\rfloor\vp)\|_\V^\beta\d l \d s\\&
=\sum_{k=0}^{\lfloor \frac{\tau_m}{\vp}\rfloor}\int_{k\vp}^{(k+1)\vp\wedge \tau_m}\int_{k\vp}^s \|\Y(l)-\Y(k\vp)\|_\V^\beta\d l \d s \\& \leq  
\sum_{k=0}^{\lfloor \frac{\tau_m}{\vp}\rfloor}\int_{k\vp}^{(k+1)\vp\wedge \tau_m}\int_{k\vp}^{(k+1)\vp\wedge \tau_m} \|\Y(l)-\Y(k\vp)\|_\V^\beta\d l \d s
 \\& \leq   \vp 
\sum_{k=0}^{\lfloor \frac{\tau_m}{\vp}\rfloor}\int_{k\vp}^{(k+1)\vp\wedge \tau_m} \|\Y(s)-\Y(\lfloor \frac{s}{\vp}\rfloor\vp)\|_\V^\beta\d s \\& 
=\vp \int_0^{\tau_m}\|\Y(s)-\Y(\lfloor \frac{s}{\vp}\rfloor\vp)\|_\V^\beta\d s.
\end{align*}
\vspace{2mm}
\noindent
\textsl{Estimate for $J_2(m)$:} Using Hypothesis \ref{hyp1} (H.5) and \eqref{3.8}, we obtain
\begin{align}\label{3.18}
	J_2(m)\leq C_{M,T}2^{-m}.
\end{align}

\vspace{2mm}
\noindent
\textsl{Estimate for $J_3(m)$:} We apply BDG inequality, Hypothesis \ref{hyp1} (H.5) and \eqref{3.8} to estimate the term $J_3(m)$ as
\begin{align}\label{3.19}\nonumber
	J_3(m)&\leq 2\E\bigg[\sup_{t\in[0,\tau_m]}\bigg|\int_0^t\int_{\lfloor\frac{s}{\vp}\rfloor\vp}^{s}\big(\s(\Y(l))\d\W(l),\Y(l)-\Y(\lfloor \frac{s}{\vp}\rfloor\vp)\big)\d s\bigg]\\&\leq C \E\bigg[\int_0^{\tau_m}\int_{\lfloor\frac{s}{\vp}\rfloor\vp}^{s}\|\s(\Y(l))\|_{\L_2}^2\|\Y(l)-\Y(\lfloor \frac{s}{\vp}\rfloor\vp)\|_\H^2\d l\d s\bigg]^\frac{1}{2} \\&\label{3.20} \leq C_{M,T}2^{-\frac{m}{2}}.
\end{align}Combining \eqref{3.16}-\eqref{3.20}, we deduce 
\begin{align}\label{3.21}
	\E\bigg[\int_0^{\tau_m}\|\Y(s)-\Y(\lfloor\frac{s}{\vp}\rfloor\vp)\|_\H^2\d s\bigg]&\leq C_{M,T,\|f\|_{\L^1}}2^{-\frac{m}{2}}.
\end{align}Again, using Fubini's theorem in \eqref{3.19}, Hypothesis \ref{hyp1} (H.5) and \eqref{3.20}, we find 
\begin{align}\label{3.22}\nonumber
	J_3(m)&\leq C\E\bigg[\vp\int_0^{\tau_m}\|\s(\Y(l))\|_{\L_2}^2\|\Y(l)-\Y(\lfloor\frac{l}{\vp}\rfloor\vp)\|_\H^2\d l\bigg]^\frac{1}{2}\\&\nonumber\leq 
	C\E\bigg[\vp\int_0^{\tau_m}\big(1+\|\Y(l)\|_\H^2\big)\|\Y(l)-\Y(\lfloor\frac{l}{\vp}\rfloor\vp)\|_\H^2\d l\bigg]^\frac{1}{2}\\&\leq C_{M,T}2^{-\frac{3m}{4}}.
\end{align}
Hence, from \eqref{3.16}-\eqref{3.19} and \eqref{3.22}, we obtain the required estimate \eqref{3.10}. Using similar arguments, one can prove that the inequalities \eqref{3.11}-\eqref{3.15} hold.

\vspace{2mm}
\noindent
\textbf{Step 3.} In this step, we estimate \eqref{3.9} term by term. 

\vspace{2mm}
\noindent
\textsl{Estimate for $I_1(m,t)$:} By Hypothesis \ref{hyp1} (H.2), we reach at 
\begin{align}\label{3.23}
|I_1(m,t)| \leq \int_0^t\big(f(s)+\rho(\Y^m(s))+\eta(\Y(s))\big)\|\Y^m(s)-\Y(s)\|_\H^2\d s.
\end{align}

\vspace{2mm}
\noindent
\textsl{Estimate for $I_2(m,t)$:} For $s\in[0,t\wedge\tau_m]$, by Hypothesis \ref{hyp1} (H.2), we have 
\begin{align}\label{3.24}\nonumber
	&|I_2(m,t)| \\&\nonumber \leq 2\int_0^{t\wedge\tau_m}\bigg\|\bigg(\frac{1}{\vp}\int_{\lceil \frac{s}{\vp}\rceil\vp }^{(\lceil \frac{s}{\vp}\rceil+1)\vp }\chi_{\{l> t\wedge\tau_m\}}\d l\bigg)\s\big(\Y^m(\lfloor{\frac{s}{\vp}}\rfloor\vp)\big)\bigg\|_{\L_2}^2\d s\\&\nonumber\quad + 4\int_0^{t\wedge\tau_m}\|\s\big(\Y^m(\lfloor{\frac{s}{\vp}}\rfloor\vp)\big)-\s(\Y^m(s))\|_{\L_2}^2\d s+8\int_0^{t\wedge\tau_m}\|\s(\Y^m(s))\Pi_m-\s(\Y^m(s))\|_{\L_2}^2\d s\\&\nonumber \quad + 8\int_0^{t\wedge\tau_m}\|\s(\Y^m(s))-\s(\Y(s))\|_{\L_2}^2\d s\\&\nonumber\leq  \underbrace{2\int_{t\wedge\tau_m-2\vp}^{t\wedge\tau_m}\|\s\big(\Y^m(\lfloor{\frac{s}{\vp}}\rfloor\vp)\big)\|_{\L_2}^2\d s}_{=:J_4(m)}+\underbrace{8\int_0^{\tau_m}\|\s(\Y^m(s))\Pi_m-\s(\Y^m(s))\|_{\L_2}^2\d s}_{=:J_5(m)}\\&\nonumber\quad +\underbrace{4 \int_0^{\tau_m}\bigg(\kappa(\Y^m(\lfloor{\frac{s}{\vp}}\rfloor\vp))+\varkappa(\Y^m(s))\bigg)\|\Y^m(\lfloor{\frac{s}{\vp}}\rfloor\vp)-\Y^m(s)\|_\H^2\d s}_{=:J_6(m)}\\&\quad + 8\int_0^{\tau_m}\big(\kappa(\Y^m(s))+\varkappa(\Y(s))\big)\|\Y^m(s)-\Y(s)\|_\H^2\d s.
\end{align}
Using the definition of $\vp$ (for $J_4(m)$), Hypothesis \ref{hyp2} (H.6) and \eqref{3.8} (for $J_5(m)$), and Hypothesis \ref{hyp1} (H.2), \eqref{3.8} and \eqref{3.11} (for $J_6(m)$), we find 
\begin{align}\label{3.25}
	\lim_{m\to\infty} J_7(m)=0, \ \  \text{ where } \ J_7(m)=\E\big[J_4(m)+J_5(m)+J_6(m)\big].
\end{align}From \eqref{3.24}, we have 
\begin{align}\label{3.26}
	\E\bigg[\sup_{t\in[0,\tau_m]}|I_2(m,t)|\bigg]\leq J_7(m)+8\E\bigg[\int_0^{\tau_m}\big(\kappa(\Y^m(s))+\varkappa(\Y(s))\big)\|\Y^m(s)-\Y(s)\|_\H^2\d s\bigg].
\end{align}

\vspace{2mm}
\noindent
\textsl{Estimate for $I_3(m,t)$:} Using BDG and Young's inequalities, we find
\begin{align}\label{3.27}\nonumber
	\E\bigg[\sup_{t\in[0,\tau_m]}|I_3(m,t)|\bigg]&\leq 4\E\bigg[\sup_{t\in[0,\tau_m]}\|\Y^m(t)-\Y(t)\|_\H\cdot\bigg\{\sup_{t\in[0,\tau_m]}|I_2(m,t)|\bigg\}^{\frac{1}{2}}\bigg]\\&\leq 
	\frac{1}{2}\E\bigg[\sup_{t\in[0,\tau_m]}\|\Y^m(t)-\Y(t)\|_\H^2\bigg]+8\E\bigg[\sup_{t\in[0,\tau_m]}|I_2(m,t)|\bigg].
\end{align}

\vspace{2mm}
\noindent
\textsl{Estimate for $I_4(m,t)$:} We claim that 
\begin{align}\label{3.28}
	J_8(m):=\E\bigg[\sup_{t\in[0,\tau_m]}|I_4(m,t)|\bigg]\to0,\ \ \text{ as } \ \ m\to\infty.
\end{align}
The proof of the claim \eqref{3.28} depends on finding an appropriate term from 
\begin{align*}
	\bigg[\s\big(\Y^m(s)\big)-\s\big(\Y^m\big((\lfloor\frac{s}{\vp}\rfloor-1)\vp\big)\big)\bigg]\dot{\W}^m(s),
	\end{align*}which can be compensated with the correction term $-\frac{1}{2}\wi{\Tr}_m(\Y^m(s))$. To obtain the appropriate term, by \eqref{1.3} and \eqref{1.5}, we equivalently write 
\begin{align}\label{3.29}
	\wi{\Tr}_m(\Y^m)=\sum_{i=1}^m\D\s_i(\Y^m)\s_i(\Y^m), \ \ \s(\Y^m)\dot{\W}^m=\sum_{i=1}^m\s_i(\Y^m)\dot{\beta}_i^m.
\end{align}
We know by Hypothesis \ref{hyp2} that $\s_i$ is twice Fr\'echet differentiable  for all $i\in\N$. Using second order Taylor's formula to $\s_i$ (see \cite[Theorem 7.9.1]{PGC}), we find 
\begin{align}\label{3.30}\nonumber
&	\s_i\big(\Y^m(s)\big)-\s_i\big(\Y^m\big((\lfloor\frac{s}{\vp}\rfloor-1)\vp\big)\big) \\&\nonumber=\D\s_i \big(\Y^m\big((\lfloor\frac{s}{\vp}\rfloor-1)\vp\big)\big)\bigg[\Y^m(s)-\Y^m\big((\lfloor\frac{s}{\vp}\rfloor-1)\vp\big) \bigg]\\&\nonumber\quad +\int_0^1(1-\theta)\D^2\s_i\big(\theta\Y^m(s)+(1-\theta)\Y^m\big((\lfloor\frac{s}{\vp}\rfloor-1)\vp\big)\big)\\&\qquad\times \bigg\{\Y^m(s)-\Y^m\big((\lfloor\frac{s}{\vp}\rfloor-1)\vp\big),\Y^m(s)-\Y^m\big((\lfloor\frac{s}{\vp}\rfloor-1)\vp\big)\bigg\}\d \theta,
\end{align}where $\D^2\s_i(\v)\{\v_1,\v_2\}$ represents the value of the second order Fr\'echet derivative $\D^2\s_i(\v)$ on the elements $\v_1$ and $\v_2$. In view of \eqref{1.6}, we are able to write 
\begin{align}\label{3.31}
	\Y^m(s)-\Y^m\big((\lfloor\frac{s}{\vp}\rfloor-1)\vp\big)=\int_{(\lfloor\frac{s}{\vp}\rfloor-1)\vp}^s \bigg(\A(l,\Y^m(l))+\s(\Y^m(l))\dot{\W}^m(l)-\frac{1}{2}\wi{\Tr}(\Y^m(l))\bigg)\d l.
\end{align}Using \eqref{1.3} and second equality of \eqref{3.29}, we arrive at 
\begin{align}\label{3.32}
&	\int_{(\lfloor\frac{s}{\vp}\rfloor-1)\vp}^s\s(\Y^m(l))\dot{\W}^m(l)\d l\nonumber\\&=\sum_{i=1}^m\bigg[\dot{\beta}_i^m\big((\lfloor\frac{s}{\vp}\rfloor-1)\vp\big)\int_{(\lfloor\frac{s}{\vp}\rfloor-1)\vp}^{\lfloor\frac{s}{\vp}\rfloor\vp} \s_j\big(\Y^m(l)\big)\d l+\dot{\beta}_i^m(s)\int_{\lfloor\frac{s}{\vp}\rfloor\vp}^s\s_i\big(\Y^m(l)\big)\d l\bigg].
\end{align}
Substituting the values from  \eqref{3.29}-\eqref{3.32} into $I_4(m,t)$, we find 
\begin{align}\label{3.33}
I_4(m,t)=: T_1(m,t)+T_2(m,t)+T_3(m,t)+T_4(m,t)+T_5(m,t)+T_6(m,t),
\end{align} where 
\begin{align*}
	T_1(m,t)&:= \sum_{i=1}^m\int_0^t\dot{\beta}_i^m(s)	\\&\nonumber\qquad\quad \times\bigg\langle \D\s_i\big(\Y^m\big((\lfloor\frac{s}{\vp}\rfloor-1)\vp\big)\big)\int_{(\lfloor\frac{s}{\vp}\rfloor-1)\vp}^s\A(l,\Y^m(l))\d l,\Y^m(s)-\Y(s)\bigg\rangle\d s, \\ 
	T_2(m,t)&:= \sum_{i=1}^m\sum_{j=1}^m\int_0^t\dot{\beta}_i^m(s)\dot{\beta}_j^m\big((\lfloor\frac{s}{\vp}\rfloor-1)\vp\big)
	\\&\qquad\quad \times\bigg(\D\s_i\big(\Y^m\big((\lfloor\frac{s}{\vp}\rfloor-1)\vp\big)\big)\int_{(\lfloor\frac{s}{\vp}\rfloor-1)\vp}^{\lfloor\frac{s}{\vp}\rfloor\vp}\s_j\big(\Y^m(l)\big)\d l,\Y^m(s)-\Y(s)\bigg)\d s,\\
		T_3(m,t)&:= \sum_{i=1}^m\sum_{1\leq j\leq m, j\neq i}^m\int_0^t\dot{\beta}_i^m(s)\dot{\beta}_j^m(s)\\&\qquad\quad \times\bigg(\D\s_i\big(\Y^m\big((\lfloor\frac{s}{\vp}\rfloor-1)\vp\big)\big)\int_{\lfloor\frac{s}{\vp}\rfloor\vp}^s\s_j\big(\Y^m(l)\big)\d l,\Y^m(s)-\Y(s)\bigg)\d s,\\
			T_4(m,t)&:= \sum_{i=1}^m\int_0^t\bigg(\big[\dot{\beta}^m(s)\big]^2\D\s_i\big(\Y^m\big((\lfloor\frac{s}{\vp}\rfloor-1)\vp\big)\big)\int_{\lfloor\frac{s}{\vp}\rfloor\vp}^s\s_j\big(\Y^m(l)\big)\d l\\&\qquad\quad  -\frac{1}{2}\D\s_i\big(\Y^m(s)\big)\s_i\big(\Y^m(s)\big),\Y^m(s)-\Y(s)\bigg)\d s, \\ 
				T_5(m,t)&:= -\frac{1}{2}\sum_{i=1}^m\int_0^t\dot{\beta}_i^m(s)\bigg(\D\s_i\big(\Y^m\big((\lfloor\frac{s}{\vp}\rfloor-1)\vp\big)\big)\\&\qquad\quad \times \int_{(\lfloor\frac{s}{\vp}\rfloor-1)\vp}^s\wi{\Tr}_m\big(\Y^m(l)\big)\d l,\Y^m(s)-\Y(s)\bigg)\d s,\\
					T_6(m,t)&:= \sum_{i=1}^m\int_0^t\dot{\beta}_i^m(s)\bigg(\int_0^1(1-\theta)\D^2\s_i\big(\theta\Y^m(s)+(1-\theta)\Y^m\big((\lfloor\frac{s}{\vp}\rfloor-1)\vp\big)\big) \\&\qquad
					\bigg\{\Y^m(s)-\Y^m\big((\lfloor\frac{s}{\vp}\rfloor-1)\vp\big),\Y^m(s)-\Y^m\big((\lfloor\frac{s}{\vp}\rfloor-1)\vp\big)\bigg\}\d \theta,\Y^m(s)-\Y(s)\bigg)\d s.
\end{align*}Now, we estimate $T_k(m,t)$ for $k=1,2,\ldots,6$ term by term. For $T_1(m,t)$, we have 
\begin{align*}
&	T_1(m,t)\\&=\sum_{i=1}^m\int_0^t\dot{\beta}_i^m(s)\bigg\langle \int_{(\lfloor\frac{s}{\vp}\rfloor-1)\vp}^s\A(l,\Y^m(l)),\D\s_i\big(\Y^m\big((\lfloor\frac{s}{\vp}\rfloor-1)\vp\big)\big)^*\big[\Y^m(s)-\Y(s)\big]\bigg\rangle\d l\d s.
\end{align*}Using \eqref{3.8} and  Hypothesis \ref{hyp2} (H.6), we obtain
\begin{align}\label{3.34}\nonumber
	\E\bigg[\sup_{t\in[0,\tau_m]}|T_1(m,t)|\bigg]&\nonumber\leq C_M \sum_{i=1}^m\E\bigg[\int_0^{\tau_m}|\dot{\beta}_i^m(s)|\bigg(\int_{(\lfloor\frac{s}{\vp}\rfloor-1)\vp}^s\|\A(l,\Y^m(l))\|_{\V^*}^\frac{\beta}{\beta-1}\d l\bigg)^\frac{\beta-1}{\beta}\\&\nonumber\qquad\times\bigg(\int_{(\lfloor\frac{s}{\vp}\rfloor-1)\vp}^s\|\Y^m(s)-\Y(s)\|_\V^\beta\d l\bigg)^\frac{1}{\beta}\d s\bigg]\\&\nonumber\leq C_M m^\frac{3}{2}2^\frac{m}{2}\vp^\frac{1}{\beta}\bigg\{\E\bigg[\int_0^{\tau_m}\d s\int_{(\lfloor\frac{s}{\vp}\rfloor-1)\vp}^s\big(f(l)+C\|\Y^m(l)\|_\V^\beta\big)\d l\bigg]\bigg\}^\frac{\beta-1}{\beta}
\\&\nonumber\leq C_M m^\frac{3}{2}2^\frac{m}{2}\vp\bigg\{\E\bigg[\int_0^{\tau_m}\big(f(l)+C\|\Y^m(l)\|_\V^\beta\big)\d l\bigg]\bigg\}^\frac{\beta-1}{\beta}
\\&\leq 
C_{M,T,\|f\|_{\L^1}}m^\frac{3}{2}2^{-\frac{m}{2}},
\end{align}where we have used Fubini's theorem and \eqref{1.16} in the following manner
\begin{align}\label{3.35}\nonumber
&\E\bigg[\int_0^{\tau_m}\d s \int_{(\lfloor\frac{s}{\vp}\rfloor-1)\vp}^s\big(f(l)+C\|\Y^m(l)\|_\V^\beta\big)\d l\bigg]\\&\nonumber\leq \E\bigg[\sum_{i=1}^{\lfloor \frac{\tau_m}{\vp}\rfloor\vp}\int_{(i-1)\vp}^{(i+1)\vp\wedge \tau_m}\d l\int_{i\vp}^{(i+1)\vp\wedge\tau_m}\big(f(l)+\|\Y^m(l)\|_\V^\beta\big)\d s\bigg]\\&\leq 2\vp\E\bigg[\int_0^{\tau_m}\big(f(l)+\|\Y^m(l)\|_\V^\beta\big)\d l\bigg].
\end{align}For the remaining terms $T_k(m,t),\ k=2,3,\ldots,6$, we refer the readers to the work \cite{TMRZ} (see proof of (2.19) in \cite[Section 2]{TMRZ}). Since the calculations are same, therefore we are not repeating here. For the sake the completeness, we are providing the bounds for each $T_k(m,t)$, for  $k=2,3,\ldots,6$, 
\begin{equation}\label{3.36}
	\left\{
	\begin{aligned}
		\E\bigg[\sup_{t\in[0,\tau_m]}|T_2(m,t)|+\sup_{t\in[0,\tau_m]}|T_3(m,t)|\bigg] &\leq C_{M,T}m^32^{-\frac{3m}{8}}, \\ 
				\E\bigg[\sup_{t\in[0,\tau_m]}|T_4(m,t)|\bigg] &\leq C_{M,T}m^22^{-\frac{3m}{8}}, \\
					\E\bigg[\sup_{t\in[0,\tau_m]}|T_5(m,t)|\bigg] &\leq C_{M,T}m^\frac{3}{2}2^{-\frac{m}{2}}, \\
						\E\bigg[\sup_{t\in[0,\tau_m]}|T_6(m,t)|\bigg] &\leq C_{M,T}m^\frac{3}{2}2^{-\frac{m}{4}}.
		\end{aligned}
	\right.
\end{equation}Combining \eqref{3.34}-\eqref{3.36}, and passing $m\to\infty$, we obtain the convergence in \eqref{3.28}.

\vspace{2mm}
\noindent
\textbf{Step 4.} Taking supremum from $0$ to $\tau_m$ and then expectation, respectively, both side of \eqref{3.9}, and using \eqref{3.23} and \eqref{3.26}-\eqref{3.28} in the final estimate, we find 
\begin{align}\label{3.37}\nonumber
	&\E\bigg[\sup_{s\in[0,\tau_m]} \|\Y^m(s)-\Y(s)\|_\H^2\bigg]\\&\leq  J_7(m)+J_8(m)+\E\bigg[\int_0^{\tau_m}\big(f(s)+\rho(\Y^m(s))+\eta(\Y(s))\big)\|\Y^m(s)-\Y(s)\|_\H^2\d s\bigg].
\end{align}Set $\mathsf{Y}(t):=\sup\limits_{s\in[0,\tau_m]} \|\Y^m(s)-\Y(s)\|_\H^2$ and $\displaystyle\mathsf{X}(t):=\int_0^t\big(f(s)+\rho(\Y^m(s))+\eta(\Y(s))\big)\d s$, for $t\in[0,T]$. Then, $\mathsf{Y}(\cdot)$ and $\mathsf{X}(\cdot)$ are adapted, non-negative and continuous. Using \eqref{3.8}, Hypothesis \ref{hyp1} (H.2), we find a positive constant $C'_M$ such that $\mathsf{X}(t)\leq C'_M$ uniformly for $t\in[0,\tau_m]$. From \eqref{3.37}, we have 
\begin{align}\label{3.38}
	\E\big[\mathsf{Y}(\tau_m)\big]\leq J_7(m)+J_8(m)+\E\bigg[\int_0^{\tau_m}\mathsf{Y}(s)\d\mathsf{X}\bigg].
\end{align}Using \cite[Lemma 2]{BGRM} in the above inequality, we obtain 
\begin{align}\label{3.39}
	\E\bigg[\int_0^{\tau_m}\mathsf{Y}(s)\d\mathsf{X}\bigg] \leq \big\{J_7(m)+J_8(m)\big\}e^{C'_M}\int_0^{C'_M}e^y\d y\to 0, \ \ \text{ as } \ \ m\to\infty,
\end{align}where we have used the convergences established in \eqref{3.25} and \eqref{3.28}. Hence, \eqref{3.38} with \eqref{3.25}, \eqref{3.28} and \eqref{3.39} imply that 
\begin{align*}
	\lim_{m\to\infty}	\E\bigg[\sup_{s\in[0,\tau_m]} \|\Y^m(s)-\Y(s)\|_\H^2\bigg]=0,
\end{align*}as required. 
\end{proof}

\section{Support of solutions  to the problem \eqref{1.1}} \label{Sec4}\setcounter{equation}{0}
In this section, we discuss the support of the solution to the problem \eqref{1.1} with the help of Wong-Zakai approximation result obtained in Section \ref{Sec3}. Let $T>0$ and $\W(\cdot)$ be a $\U$-cylindrical Wiener process on a probability space $(\Omega,\mathscr{F},\{\mathscr{F}_t\}_{t\geq0},\P)$, with the filteration  $\{\mathscr{F}_t\}_{t\geq0}$ generated by $\W(\cdot)$. For $m\in\N$, $t\in[0,T]$, set
\begin{align*}
	\Z^m(t):=\exp\bigg\{\int_0^t\dot{\W}^m(s)\d \W(s)-\frac{1}{2}\int_0^t\|\dot{\W}^m(s)\|_\U^2\d s\bigg\}, 
\end{align*}and 
\begin{align}\label{4.1}
	\widetilde{\W}^m(t):=\W(t)-\int_0^t\dot{\W}^m(s)\d s, 
\end{align}where $\dot{\W}^m(\cdot)$ is defined  \eqref{1.3}. Since the real-valued random variables $\dot{\beta}_i(l\vp), \ i,l\in\N$ are independent and for each $i,l\in\N$, $\vp^{\frac{1}{2}}\dot{\beta}_i(l\vp)$ is standard Gaussian,  for every $m\in\N$, we have 
\begin{align*}
	\sup_{t\in[0,T]} \E\left[e^{\xi\|\dot{\W}^m(t)\|_\U^2}\right]  = \sup_{t\in[0,T]}  \prod_{1\leq i \leq m}\E\big[e^{\xi|\dot{\beta}_i^m(t)|^2}\big]= \big\{\E\big[e^{\frac{\xi}{\vp}|\M|^2}\big]\big\}^m<\infty,
\end{align*}
for some standard Gaussian random variable $\M$ and for a small positive number  $\xi$ (Fernique’s theorem, \cite[Theorem 2.7]{DaZ}). Using Girsanov's theorem, the process $\{\widetilde{\W}^m(t)\}_{t\geq 0}$ defined in \eqref{4.1} is a cylindrical Wiener process under the probability measure $\P^m$ and the measure  $\P^m$ is absolute continuous with respect to the measure $\P$, that is, 
\begin{align*}
	\frac{\d \P^m}{\d \P}\bigg|_{\mathscr{F}_t}=\Z^m(t), \ \  \text{ for } \ \ t\in[0,T].
\end{align*}Similarly, for any $\g\in\L^2(0,T;\U)$, we define the following processes
\begin{align*}
	\Z_{\g}^m(t):=\exp\bigg\{-\int_0^t\g(s)\d\widetilde{\W}^m(s)-\frac{1}{2}\int_0^t\|\g(s)\|_\U^2\d s\bigg\}, 
\end{align*}and 
\begin{align}\label{4.2}
\widetilde{\W}_{\g}^m(t):=\widetilde{\W}^m(t)+\int_0^t\g(s)\d s, \ \text{ for } \ t\in[0,T],\ m\in\N.
\end{align}Again, using Girsanov's theorem, we find an another probability measure $\P_{\g}^m\ll \P^m\ll \P$, such that 
\begin{align}\label{4.3}
		\frac{\d \P_{\g}^m}{\d \P^m}\bigg|_{\mathscr{F}_t}=\Z_{\g}^m(t), \ \  \text{ for } \ \ t\in[0,T],
\end{align}and the process $\{\widetilde{\W}_{\g}^n(t)\}_{t\geq 0}$ is a cylindrical Wiener process under the probability measure $\P_{\g}^m$. 

	For $\g\in\L^2(0,T;\U)$, we consider the following two systems, which can be seen as two particular cases of the system  \eqref{4.4}
	\begin{equation}\label{4.06}
		\left\{
		\begin{aligned}
		\d \ZZ_{\g}(t)&=\A(t,\ZZ_{\g}(t))\d t+\s(\ZZ_{\g}(t))\g(t)\d t-\frac{1}{2}\wi{\Tr}_m(\ZZ_{\g}(t))\d t,\\
		\ZZ_{\g}(0)&=\boldsymbol{y}_0,
		\end{aligned}
		\right.
	\end{equation}and 
	\begin{equation}\label{4.07}
	\left\{
	\begin{aligned}
		\d \ZZ_{\g}^m(t)&=\A(t,\ZZ_{\g}^m(t))\d t+\s(\ZZ_{\g}^m(t))\g(t)\d t+\s(\ZZ_{\g}^m(t))\d\W(t)-\s(\ZZ_{\g}^m(t))\dot{\W}^m(t)\d t,\\
		\ZZ_{\g}(0)&=\boldsymbol{y}_0,
	\end{aligned}
	\right.
\end{equation}
where the operators $\A,\ \s, \ \dot{\W}$ and $\wi{\Tr}$ are defined in Section \ref{Sec2}. Moreover,  the well-posedness of the systems \eqref{4.06} and \eqref{4.07} has been established in Theorem \ref{thrm4.1}. Let $\ZZ_{\g}(\cdot)$ and $\ZZ_{\g}^m(\cdot)$ be the solutions to the systems \eqref{4.06} and \eqref{4.07}, respectively. Then, $\ZZ_{\g}, \ZZ_{\g}^m\in\C([0,T];\H),\ \P$-a.s., and the following uniform energy estimates hold:
\begin{align}\label{4.08}
\E\bigg[\sup_{t\in[0,T]}\|\ZZ_{\g}(t)\|_\H^p\bigg]+\E\bigg[\int_0^T\|\ZZ_{\g}(t)\|_\V^\beta\d t\bigg]^{\frac{p}{2}}&\leq C<\infty,\\ \label{4.09}
	\E\bigg[\sup_{t\in[0,T]}\|\ZZ_{\g}^m(t)\|_\H^p\bigg]+\E\bigg[\int_0^T\|\ZZ_{\g}^m(t)\|_\V^\beta\d t\bigg]^{\frac{p}{2}}&\leq C_m<\infty.
\end{align}
The following result can be obtained in a similar manner as Theorem \ref{thrm3.3}, which demonstrates the Wong-Zakai approximation results for the system \eqref{4.06}.

\begin{lemma}\label{lem4.02}
	Assume that  the Hypotheses \ref{hyp1} and \ref{hyp2} hold  and the initial data $\boldsymbol{y}_0\in\L^p(\Omega;\H)$, for $p>\max\big\{\frac{\beta}{\beta-1},2\big\}$. For $\g\in\L^2(0,T;\U)$, let  $\ZZ_{\g}, \ZZ_{\g}^m$ be the solutions to the problems \eqref{4.06} and \eqref{4.07}, respectively, with the same initial data $\boldsymbol{y}_0\in\L^p(\Omega;\H)$. Then 
	\begin{align}\label{4.010}
\lim_{m\to\infty}\E\bigg[\sup_{t\in[0,T]}\|\ZZ_{\g}^m(t)-\ZZ_{\g}(t)\|_\H^2\bigg]=0.		
	\end{align}
\end{lemma}
\begin{proof}
	The proof of \eqref{4.010} will be on the similar lines as Theorem \ref{thrm3.3}. In comparison with \eqref{3.9}, we only  need to control the term $\displaystyle\int_0^t\big(\big[\s(\ZZ_{\g}^m(s))-\s(\ZZ_{\g}(s))\big]\g(s),\ZZ_{\g}^m(s)-\ZZ_{\g}(s)\big)\d s$ and it can be handled in the following way: 
	\begin{align*}
	&	\int_0^t\big(\big[\s(\ZZ_{\g}^m(s))-\s(\ZZ_{\g}(s))\big]\g(s),\ZZ_{\g}^m(s)-\ZZ_{\g}(s)\big)\d s \\&\leq \int_0^t\big(\kappa(\ZZ_{\g}^m(s))+\varkappa(\ZZ_{\g}(s))+\|\g(s)\|_\U^2\big)\|\ZZ_{\g}^m(s)-\ZZ_{\g}(s)\|_\H^2\d s.
	\end{align*}Since $\g\in\L^2(0,T;\U)$, the above inequality can be controlled and similar arguments as in the proof of Theorem \ref{thrm3.3} give the required convergence \eqref{4.010}.
\end{proof}
Set $\DD=\C([0,T];\H)$ and  $\LL=\{\ZZ_{\g}, \ \g\in\L^2(0,T;\U)\}$. Note that $\LL\subset\DD$. The proof of the following result is based on similar arguments as in  the proof of \cite[Theorem 4.4]{TMRZ}, and for the sake of completeness we are providing a proof here.
\begin{theorem}\label{thrm4.03}
	Assume that Hypotheses \ref{hyp1} and \ref{hyp2} and \eqref{1.3} hold and the initial data $\boldsymbol{y}_0\in\L^p(\Omega;\H)$ for $p>\max\big\{\frac{\beta}{\beta-1},2\big\}$. Let $\Y(\cdot)$ be the solution to the system \eqref{1.1}. Then 
	\begin{align*}
		\mathrm{supp}(\P\circ\Y^{-1})=\bar{\LL},
	\end{align*}where $\bar{\LL}$ is the closure of $\LL$ in $\DD$ and $\mathrm{supp}(\P\circ\Y^{-1})$ represents the support of the distribution $\P\circ\Y^{-1}$.
\end{theorem}
\begin{proof}
	Let $\Y(\cdot)$ and $\Y^m(\cdot)$ be the unique solutions to the problems \eqref{1.1} and \eqref{1.6}, respectively. Since $\ZZ_{\g}(\cdot)$ represents the solution to the problem \eqref{4.06}, then, choosing $\g=\dot{\W}^m$ in \eqref{4.06},  we obtain $\ZZ_{\dot{\W}^m}(\cdot)=\Y^m(\cdot),\ \P$-a.s. Moreover, by Theorem \ref{thrm3.3}, for any $\mu>0$, we have
	\begin{align*}
		\lim_{m\to\infty}\P\big\{\|\ZZ_{\dot{\W}^m}-\Y\|_{\DD}\geq \mu\big\}=\lim_{m\to\infty}\P\big\{\|\Y^m-\Y\|_{\DD}\geq \mu	\big\}=0.
	\end{align*}Since $\P$-a.s., $\dot{\W}^m\in\L^2(0,T;\U),$ we have
\begin{align}\label{4.011}
	\mathrm{supp}(\P\circ\Y^{-1})\subset\bar{\LL}.
\end{align} 
\vspace{2mm}
\noindent
\textbf{Claim:}  $\mathrm{supp}(\P\circ\Y^{-1})\supset\bar{\LL}$.

Using \cite[Remark 2.5.1]{WLMR2}, we obtain an another Hilbert space $\U_1\supset\U$ such that the embedding from $\big(\U,\langle \cdot,\cdot\rangle_{\U}\big)$ to $\big(\U_1,\langle \cdot,\cdot\rangle_{\U_1}\big)$ is Hilbert-Schmidt.  Therefore, we can find a set $\{\ee_i\}_{i\in\N}\subset\U$, $0< \lambda_i\uparrow\infty$ such that $\{\ee_i\}_{i\in\N}$ is an orthonormal basis of $\U$ and $\{\sqrt{\lambda_i}\ee_i\}_{i\in\N}$ is an orthonormal basis in the Hilbert space $\U_1$. Fix such $\U_1$ and define $\mathbb{W}^{\U_1}:=\C([0,\infty);\U_1)$ and 
\begin{align*}
	\mathbb{W}_0^{\U_1}:=\big\{\w\in\mathbb{W}^{\U_1}:\w(0)={0}\big\},
\end{align*}and the metric associated with $\mathbb{W}_0^{\U_1}$ is defined by 
\begin{align*}
	d(\v_1,\v_2):=\sum_{i=1}^{\infty}2^{-i}\bigg(\max_{0\leq t\leq i}\|\v_1(t)-\v_2(t)\|_{\U_1}\wedge1\bigg), \ \ \v_1,\v_2\in 	\mathbb{W}_0^{\U_1}.
\end{align*}The space $	\mathbb{W}_0^{\U_1}$ is a Polish space with respect the metric $d(\cdot,\cdot)$. Denote the Borel sigma-algebra of $\mathbb{W}_0^{\U_1}$ by   $\mathcal{B}(\mathbb{W}_0^{\U_1})$. It follows that $\W\in\mathbb{W}_0^{\U_1}, \ \P$-a.s. Let $\{\mathcal{B}_t(\mathbb{W}_0^{\U_1})\}_{t\geq0}$ be the normal filtration generated by the canonical process $\mathtt{w}$.  We obtain an another  complete probability space 
\begin{align*}
	\big(\mathbb{W}_0^{\U_1},\cup_{t\geq0}\mathcal{B}_t(\mathbb{W}_0^{\U_1}),\mathcal{B}_t(\mathbb{W}_0^{\U_1}),\bar{\P}\big),
\end{align*}where $\bar{\P}$ denotes the distribution of  $\mathtt{w}$ in $\C([0,\infty);\U_1)$, that is, 
\begin{align}\label{4.012}
	\bar{\P}\circ \mathtt{w}^{-1}=\P\circ\W^{-1}.
\end{align}Let $\boldsymbol{y}_0$ be $\mathscr{F}_0\slash\mathcal{B}(\H)$-measurable and $\boldsymbol{y}_0\in\L^p(\Omega;\H)$ for $p>\max\big\{\frac{\beta}{\beta-1},2\big\}$. By Theorem \ref{thrm2} and the Yamada-Watanabe Theorem (see \cite[Theorem E.1.8]{CPMR}), we can find a measurable map
\begin{align*}
	\mathcal{A}_{\P\circ\boldsymbol{y}_0^{-1}}:(\H\times \mathbb{W}_0^{\U_1},\mathcal{B}(\H)\otimes \mathcal{B}_t(\mathbb{W}_0^{\U_1}))\to (\DD,\mathcal{B}(\DD)),
\end{align*}where $\DD=\C([0,T];\H)$ such that $\Y:=	\mathcal{A}_{\P\circ\boldsymbol{y}_0^{-1}}(\boldsymbol{y}_0,\W)$ is the solution of the problem \eqref{1.1} with the initial data $\Y(0)=\boldsymbol{y}_0,\ \P$-a.s. For any $\g\in\L^2(0,T;\U)$, define the map $\S_{\g}^m$ on $(\mathbb{W}_0^{\U_1},\mathcal{B}(\mathbb{W}_0^{\U_1}))$ by 
\begin{align*}
	\S_{\g}^m(\x)=\x-\int_0^{\cdot}\dot{\x}^m(s)\d s+\int_0^{\cdot}\g(s)\d s, \ \ \x\in\mathbb{W}_0^{\U_1},
\end{align*}where 
\begin{align*}
	\dot{\x}^m(t):=\sum_{i=1}^m\bigg\langle \frac{\lambda_i}{\vp}\bigg[\x\big(\lfloor\frac{t}{\vp}\rfloor\vp\big)-\x\big((\lfloor\frac{t}{\vp}\rfloor-1)\vp\big)\bigg],\ee_i\bigg\rangle_{\U_1}, \ \ t\in[0,T]. 
\end{align*}
From \eqref{4.1} to \eqref{4.3}, we conclude that the maps $\S_{\g}^m$ is a measurable transformation of Wiener space $\mathbb{W}_0^{\U_1}$. Choose a $\mathcal{B}(\mathbb{W}_0^{\U_1})\slash\mathcal{B}(\H)$-measurable map $\x_0:\mathbb{W}_0^{\U_1}\to\H$ such that $\bar{\P}\circ\x_0^{-1}=\P\circ\boldsymbol{y}_0^{-1}$. Then, $	\mathcal{A}_{\bar{\P}\circ\x_0^{-1}}(\x_0(\omega),\omega)$ is also a solution of the problem \eqref{1.1} with the initial data $\boldsymbol{y}_0$ and noise $\mathtt{w}$. Since, $\x_0$ is $\mathcal{B}(\mathbb{W}_0^{\U_1})\slash\mathcal{B}(\H)$-measurable, $\x_0(\S_{\g}^m(\omega))=\x_0(\omega)$, for all $m$.  Again, by the Yamada-Watanabe theorem, pathwise uniqueness gives for  every $\mu>0,\ m\in\N$ 
\begin{align}\label{4.013}
	\bar{\P}\left\{\omega:\|	\mathcal{A}_{\bar{\P}\circ\x_0^{-1}}(\x_0(\omega),\S_{\g}^m(\omega))-\ZZ_{\g}\|_\DD\geq\mu\right\}=\P\big\{\|\ZZ_{\g}^m-\ZZ_{\g}\|_{\DD}\geq \mu\big\}, 
\end{align}which implies from Lemma \ref{lem4.02} that for $\mu$ in \eqref{4.013}, we have
\begin{align}\label{4.014}
	\lim_{m\to\infty}\P \big\{\|\ZZ_{\g}^m-\ZZ_{\g}\|_{\DD}\geq \mu\big\}=0.
\end{align}For $\bar{\P}_{\g}^m=\bar{\P}\circ\S_{\g}^{m^{-1}}, \ m\in\N$ along with \eqref{4.013} and \eqref{4.014}, it follows that there exists $m_0\in\N$ such that \begin{align*}
&\bar{\P}_{\g}^{m_0}\big\{\omega:\|\mathcal{A}_{\bar{\P}\circ\x_0^{-1}}(\x_0(\omega),\omega)-\ZZ_{\g}\|_\DD<\mu\big\}\\&= \bar{\P}\big\{\omega:\|\mathcal{A}_{\bar{\P}\circ\x_0^{-1}}(\x_0(\omega),\S_{\g}^m(\omega))-\ZZ_{\g}\|_\DD<\mu\big\}>0.
\end{align*}By \eqref{4.3}, we know that $\bar{\P}_{\g}^{m_0}\ll\bar{\P}$, which implies 
\begin{align*}
	\P\big\{\|\Y-\ZZ_{\g}\|_{\DD}<\mu\big\}=\bar{\P}\big\{\omega:\|\mathcal{A}_{\bar{\P}\circ\x_0^{-1}}(\x_0(\omega),\omega)-\ZZ_{\g}\|_\DD<\mu\big\}>0.
\end{align*}Thus, 
\begin{align}\label{4.0151}
	\mathrm{supp}\big(\P\circ\Y^{-1}\big)\supset\bar{\LL}.
\end{align}Combining \eqref{4.011} and \eqref{4.0151}, we obtain the required result, that is, $\mathrm{supp}\big(\P\circ\Y^{-1}\big)=\bar{\LL}$.
\end{proof}

\section{Examples}\label{Application}
 The results established in this work are applicable to the following models under Hypotheses \ref{hyp1} and \ref{hyp2}. For details, one can see the book \cite{WLMR1} or the article \cite{MRSSTZ}.
 \begin{enumerate}
 	\item The hydrodynamic models like two-dimensional Navier-Stokes equations, two-dimensional magneto-hydrodynamic equations, two-dimensional Boussinesq model for the B\'enard convection, two-dimensional Boussinesq system, two-dimensional magnetic B\'enard equations, three-dimensional Leray-$\alpha$ model, the Ladyzhenskaya model, and some shell models of turbulence, etc., (see \cite{ZBWLJZ,ZBXPJZ,ICAM1,ICAM2,AKMTM5,TMRZ,EM2,XPJYJZ}, and the references therein).
 	\item The fluid dynamic models like porous media equation, $p$-Laplacian equations, fast-diffusion equations, power law fluids, Allen-Cahn equations, Kuramoto-Sivashinsky equations and two-dimensional tamed Navier-Stokes equations (see \cite{WLMR1,TMRZ}).
 	\item SPDEs with locally monotone coefficients can be covered by the framework of this article (see \cite{WL4,WLMR1,TMRZ}).
 	\item The framework of this work covers several models which are not covered by the framework of existing literature on locally monotone SPDEs like Quasilinear SPDEs, convection diffusion equations, Cahn-Hilliard equations and two-dimensional liquid crystal model (see \cite[Section 4]{MRSSTZ}).
 \end{enumerate}

\medskip\noindent
{\bf Acknowledgments:} The first author would like to thank Ministry of Education, Government of India for financial assistance. Kush Kinra would like to thank the Council of Scientific $\&$ Industrial Research (CSIR), India for financial assistance (File No. 09/143(0938) /2019-EMR-I).  M. T. Mohan would  like to thank the Department of Science and Technology (DST) Science $\&$ Engineering Research Board (SERB), India for a MATRICS grant (MTR/2021/000066). We express our gratitude to Professors Ting Ma and Rongchan Zhu for their support to improve the proofs of Lemma \ref{lem4.2} and Theorem \ref{thrm3.3}.
 The authors sincerely would like to thank the reviewer for his/her valuable comments and suggestions.

\medskip\noindent	{\bf  Declarations:} 

\noindent 	{\bf  Ethical Approval:}   Not applicable 

\noindent  {\bf   Competing interests: } The authors declare no competing interests. 

\noindent 	{\bf   Authors' contributions: } All authors have contributed equally. 

\noindent 	{\bf   Funding: } DST, India, MATRICS grant (MTR/2021/000066) (M. T. Mohan).

\noindent 	{\bf   Availability of data and materials: } Not applicable.

\end{document}